%% file: low_level_definability_author_accepted.tex
\documentclass[a4paper,10pt]{article}
\usepackage{amsthm}
\usepackage{amsmath}
\usepackage{amsxtra}
\usepackage{amssymb}
\usepackage{enumitem}
\usepackage{accents}
\usepackage{mathrsfs}
\usepackage{turnstile}
\usepackage{thmtools}
\usepackage{makeidx}
\usepackage{IEEEtrantools}
\usepackage[colorlinks=true,linkcolor=blue]{hyperref}
\usepackage{tikz}
\usepackage{geometry}
\usepackage{url}
\usepackage{bbm}

\newcommand{\strsup}{\mathrm{strsup}}
\newcommand{\QQ}{\mathbb Q}
\newcommand{\RR}{\mathbb R}
\newcommand{\PP}{\mathbb P}

\newcommand{\sub}{\subseteq}
\newcommand{\cross}{\times}
\newcommand{\all}{\forall}

\newcommand{\om}{\omega}
\newcommand{\pow}{\mathcal{P}}
\newcommand{\OR}{\mathrm{OR}}
\newcommand{\Hull}{\mathrm{Hull}}
\newcommand{\cut}{\backslash}
\newcommand{\Tt}{\mathcal{T}}
\newcommand{\Uu}{\mathcal{U}}
\newcommand{\eot}{\mathrm{eot}}
\newcommand{\rg}{\mathrm{rg}}
\newcommand{\dom}{\mathrm{dom}}
\newcommand{\ins}{\trianglelefteq}
\newcommand{\nins}{\ntrianglelefteq}
\newcommand{\pins}{\triangleleft}

\newcommand{\crit}{\mathrm{cr}}
\newcommand{\rest}{\!\upharpoonright\!}
\newcommand{\com}{\circ}
\newcommand{\range}{\rg}
\newcommand{\lh}{\mathrm{lh}}
\newcommand{\Ult}{\mathrm{Ult}}
\newcommand{\sats}{\models}
\newcommand{\J}{\mathcal{J}}

\newcommand{\HOD}{\mathrm{HOD}}
\newcommand{\HC}{\mathrm{HC}}
\newcommand{\ZFC}{\mathrm{ZFC}}

\newcommand{\pistol}{\P}

\newcommand{\es}{\mathbb{E}}
\newcommand{\her}{\mathcal{H}}
\newcommand{\pred}{\mathrm{pred}}
\newcommand{\id}{\mathrm{id}}
\newcommand{\conc}{\ \widehat{\ }\ }

\newcommand{\bfSigma}{\undertilde{\Sigma}}
\DeclareMathOperator{\Th}{Th}
\DeclareMathOperator{\cof}{cof}

\newcommand{\Club}{\mathrm{Club}}

\newcommand{\psub}{\subsetneq}

\newcommand{\lpole}{\left\lfloor}
\newcommand{\rpole}{\right\rfloor}
\newcommand{\univ}[1]{\lpole #1\rpole}
\newcommand{\tu}{\textup}

\newcommand{\eqdef}{=_{\mathrm{def}}}
\newcommand{\passive}{\mathrm{pv}}
\newcommand{\card}{\mathrm{card}}
\newcommand{\dropset}{\mathscr{D}}
\newcommand{\wt}{\widetilde}

\swapnumbers

\declaretheoremstyle[bodyfont=\it]{slanted}
\declaretheoremstyle[bodyfont=\normalfont]{normal}
\declaretheorem[name=Definition,style=normal,qed=$\dashv$,
numberwithin=section]{dfn}
\declaretheorem[name=Definition,style=normal,numbered=no,qed=$\dashv$]{dfn*}
\declaretheorem[name=Definition,style=normal,numbered=no]{dfnnoqed*}
\declaretheorem[name=Theorem,style=slanted,sibling=dfn]{tm}
\declaretheorem[name=Theorem,style=slanted,numbered=no]{tm*}
\declaretheorem[name=Lemma,style=slanted,sibling=dfn]{lem}
\declaretheorem[name=Corollary,style=slanted,sibling=dfn]{cor}
\declaretheorem[name=Corollary,style=slanted,numbered=no]{cor*}
\declaretheorem[name=Remark,style=definition,sibling=dfn]{rem}
\declaretheorem[name=Fact,style=definition,sibling=dfn]{fact}
\swapnumbers
\declaretheoremstyle[headfont=\scshape]{claimstyle}
\declaretheorem[name=Claim,style=claimstyle]{clm}
\declaretheorem[name=Claim,style=claimstyle]{clmtwo}
\declaretheorem[name=Claim,style=claimstyle]{clmthree}
\declaretheorem[name=Claim,style=claimstyle,numbered=no]{clm*}

\declaretheoremstyle[headfont=\scshape]{casestyle}
\declaretheorem[name=Case,style=casestyle]{case}
\declaretheorem[name=Case,style=casestyle]{casetwo}

\declaretheorem[name=Subcase,style=casestyle,numberwithin=casetwo]{scasetwo}

\makeatletter
\newcommand{\itemref}[2][]{
  \item[#2]
  \if\relax\detokenize{#1}\relax
  \else
    \phantomsection
    \def\@currentlabel{#2}
    \label{#1}
  \fi
}
\makeatother

\makeatletter
\def\blfootnote{\gdef\@thefnmark{}\@footnotetext}
\makeatother

\title{Low level definability 
above large cardinals}
\author{Farmer Schlutzenberg\\
farmer.schlutzenberg@tuwien.ac.at\\
TU Wien}
\begin{document}
\maketitle
\blfootnote{Low level definability above large cardinals \copyright 2024--2026 by Farmer Schlutzenberg is licensed under CC BY 4.0. To view a copy of this license, visit \url{https://creativecommons.org/licenses/by/4.0/}.}
\blfootnote{This is the author accepted manuscript version of an article to appear in Notre Dame Journal of Formal Logic.}

\begin{abstract}
We study connections between definability in generalized descriptive set theory and  large cardinals, under ZFC. We show that if $\kappa$ is a limit of measurables then no $\Sigma_1(V_\kappa\cup\mathrm{OR})$ wellorder of a subset of $\mathcal{P}(\kappa)$ of length $\geq\kappa^+$ exists, answering a question of L\"ucke and M\"uller. However, consistently,  a Woodin cardinal exists  and for every uncountable cardinal $\kappa$ which is not a limit of measurables, a $\Sigma_1(\mathcal{H}_\kappa\cup\{\kappa\})$-good wellorder of $\mathcal{H}_{\kappa^+}$ exists. If $\kappa$ is a limit of measurables and $\cof(\kappa)>\omega$ then  no $\Sigma_1(V_\kappa\cup\mathrm{OR})$ almost disjoint family $\mathscr{F}\subseteq\mathcal{P}(\kappa)$ of cardinality $>\kappa$ exists. Consistently, $\Pi_1(\{\kappa\})$ mad families and maximal independent families
$\mathscr{F}\subseteq\mathcal{P}(\kappa)$   exist,   $\kappa$ is a limit of measurables, and  more.
If $\kappa$ is weakly compact and every $\Sigma_1(V_\kappa\cup\{\kappa\})$ subset of $\mathcal{P}(\kappa)$ of cardinality $>\kappa$ contains a perfect subset of the right kind, then there is an inner model with a weakly compact limit of measurables. We prove some related facts regarding $\Sigma_1(V_\lambda\cup\{V_\lambda\}\cup\OR)$ when $I_2(\lambda)$ holds. These depend on an analysis
of fixed points of linear iterations involving $I_2(\lambda)$-extenders.
\end{abstract}

\input{low_level_definability_above_LCs_main_text}

\section*{Acknowledgements}
This research was funded in whole by the Austrian Science Fund (FWF) [10.55776/Y1498].
Therefore we have the following. This research was funded in whole or in part by the Austrian Science Fund (FWF) [10.55776/Y1498]. For open access purposes, the author has applied a CC BY public copyright license to any author accepted manuscript version arising from this submission.

The author would like to  thank the anonymous referee for their time and effort in reading the paper and their
helpful corrections and suggestions for improvements.

\bibliographystyle{plain}
\bibliography{low_level_definability_above_LCs_bibliography}

\end{document}

%% file: low_level_definability_above_LCs_main_text.tex
\section{Introduction}\label{sec:intro}

In this paper we study the definability of particular kinds of  subsets of
generalized Baire space
${^\kappa}\kappa$, for uncountable cardinals $\kappa$, in the presence of large cardinals. The kinds of sets considered are:
\begin{enumerate}
\item[--] wellorders of subsets of $\pow(\kappa)$ of cardinality $>\kappa$, and in particular, $\Gamma$-good wellorders of $\pow(\kappa)$, for complexity classes $\Gamma$,
\item[--] almost disjoint families at $\kappa$ of cardinality $>\kappa$,
\item[--] subsets of $\pow(\kappa)$
of cardinality $>\kappa$
which have no perfect subset,
\item[--]  ultrafilters over $\kappa$
and the club filter at $\kappa$,
regarding which we just make some brief remarks.
\end{enumerate}
The work relates particularly to that in L\"ucke and M\"uller \cite{Sigma_1_def_at_higher_cardinals}, L\"ucke, Schindler and Schlicht \cite{luecke_schindler_schlicht},
and L\"ucke and Schlicht \cite{meas_cards_good_Sigma_1_wo};
in particular, we answer some of the questions posed in \cite{Sigma_1_def_at_higher_cardinals}.

Given infinite cardinals $\mu\leq\nu$,
say $\nu$ is \emph{$\mu$-closed}
iff $\all\alpha<\nu\ [\alpha^\mu<\nu]$.
Given infinite cardinals $\mu\leq\kappa$,
say $\kappa$ is \emph{$\mu$-steady}
iff  there is a cardinal $\nu\leq\kappa$ such that $\nu$ is $\mu$-closed, $\cof(\nu)\neq\mu$
and $\kappa\in\{\nu,\nu^+\}$.
So if $\kappa$ is $\mu$-steady
then $\kappa>\mu^+$. Note that if $\nu$ is $\mu$-closed and $\cof(\nu)<\mu$ then $\nu^+$ is $\mu$-steady but non-$\mu$-closed,
as $\nu^\mu\geq\nu^+$.

Regarding wellorders,
first, in \S\ref{sec:wo_above_one_meas}, we consider cardinals $\kappa$
such that there is a measurable cardinal $\mu<\kappa$
and $\kappa$ is $\mu$-steady.
Under these assumptions, in
Theorem~\ref{tm:above_one_measurable}, we establish some restrictions on wellordered subsets of $\pow(\kappa)$ of cardinality $>\kappa$ which are $\Sigma_1$-definable in certain parameters;
these results are minor refinements of results in \cite{Sigma_1_def_at_higher_cardinals} and
\cite{meas_cards_good_Sigma_1_wo}.
The main new fact here is that there is no $\Sigma_1(V_\mu\cup\{\kappa\})$ injection from $\kappa^+$ into $\pow(\kappa)$, even if $\cof(\kappa)=\om$; this was proved under the added assumption that $\cof(\kappa)>\om$ in \cite[Theorem 7.1]{Sigma_1_def_at_higher_cardinals}. In fact, we show that there is no $\Sigma_1(V_\mu\cup\{\kappa\})$
set $f\sub\kappa^+\cross\pow(\kappa)$ such that for some club $C\sub\kappa^+$, $f\rest C:C\to\pow(\kappa)$ is an injective function.

In \S\ref{sec:wo_above_infinite_meas} we  consider cardinals $\kappa$ which are limits of measurables.
Here we answer \cite[Questions 10.3, 10.4]{Sigma_1_def_at_higher_cardinals},
by showing that in this context,
there is no set $D\sub\pow(\kappa)$
and wellorder $<^*$ of $D$
such that $D$ has cardinality $>\kappa$ and
 $D,<^*$ are $\Sigma_1(\her_\kappa\cup\{\kappa\})$-definable.
 (L\"ucke and M\"uller already established some results in this direction in \cite[Theorem 1.4, Corollary 7.4]{Sigma_1_def_at_higher_cardinals}.) We also  strengthen their
 results \cite[Theorem 1.1, Corollary 7.4(ii)]{Sigma_1_def_at_higher_cardinals},
 as explained in detail in \S\ref{sec:wo_above_infinite_meas}.

 In \S\ref{sec:above_rank-to-rank}, we prove that if $j:V\to M$
 is an $I_2$-embedding and $\lambda$ is the limit its critical sequence, and $D\sub\pow(\lambda)$
 is a set of cardinality $>\lambda$
 which is
  $\Sigma_1(P)$ where $P=V_\lambda\cup\{V_\lambda\}\cup\OR$,
  then $D$ has a simply perfect subset
  (Definition \ref{dfn:simply_perfect}),
  and there is no $\Sigma_1(P)$ wellorder of $D$.
  We also show that there are at most
  $\lambda$-many sets $A\sub V_\lambda$
  such that $\{A\}$ is $\Sigma_1(P)$.
  The result on simply perfect sets relates to some in \cite{descriptive_properties_I2-embeddings};
  see the discussion just prior to the statement of Theorem \ref{tm:3.11}.

  The results in \S\ref{sec:wellorders_above_measurable_cardinals}
  rely on various lemmas
  which establish the existence of various fixed points of iteration maps,
  both for linear iterations with measures,
  and linear iterations of extenders witnessing $I_2$.

 In \S\ref{sec:gdst_in_L[U]},
 we consider the situation in the model $L[U]$ for one measurable cardinal $\mu$. We show that in $L[U]$, for every uncountable cardinal $\kappa$:
 \begin{enumerate}
  \item there is a stationary-co-stationary set $d\sub\kappa^+$
  and an injective function $f:d\to\pow(\kappa)$ such that $d,f$ are both $\Sigma_1(\{\kappa\})$, and
  \item\label{item:no_perf_sset} there is a $\Sigma_1(\{\kappa\})$ set $D\sub\pow(\kappa)$ of cardinality $>\kappa$
  such that there is no perfect embedding $\iota:{^{\cof(\kappa)}}\kappa\to{^\kappa}\kappa$
  with $\rg(\iota)\sub D$.
  \end{enumerate}
  (In case $\kappa\leq\mu^+$
  or $\kappa$ is non-$\mu$-steady,
  these things already follow easily from the results in \cite{luecke_schindler_schlicht}.
  Part~\ref{item:no_perf_sset} relates to \cite[Theorem 1.2]{Sigma_1_def_at_higher_cardinals},
  by which it was already known that if $\kappa$ is singular
  then there must be some such $D$ which is $\Sigma_1(\her_\kappa\cup\{\kappa\})$. But that result did not deal with the case that $\kappa$ is regular.)

 In Theorem~\ref{tm:good_wo_1-small}, we establish the existence, in $M_1$
 (the minimal proper class mouse with a Woodin cardinal),
 of a $\Sigma_1(\her^{M_1}_\kappa\cup\{\kappa\})$-good wellorder of $\her^{M_1}_{\kappa^{+M_1}}$, for all uncountable cardinals $\kappa$
 which are not limits of measurables in $M_1$.

In \S\ref{sec:wc}
we answer \cite[Question 10.1]{Sigma_1_def_at_higher_cardinals},
which asked whether, if $\kappa$ is weakly compact and $\Sigma_1(\her_\kappa\cup\{\kappa\})$ has a certain perfect embedding property,
there must be a (proper class)
inner model satisfying ``there is a weakly compact limit of measurable cardinals''.  In Theorem~\ref{tm:wc} we show the answer is ``yes'' (the ``perfect embedding property'' is clarified in the theorem's statement).

  In \S\ref{sec:adf},
  we consider almost disjoint families. Let $\kappa$ be an infinite cardinal.
Recall that an \emph{almost disjoint family at $\kappa$} is a set $\mathscr{F}\sub\pow(\kappa)$
such that for all $A\in\mathscr{F}$,
$A$ is unbounded in $\kappa$,
and for all $A,B\in\mathscr{F}$
with $A\neq B$, $A\cap B$ is bounded in $\kappa$. We usually drop the phrase ``at $\kappa$'', as $\kappa$ will be clear from context. An almost disjoint family $\mathscr{F}$ is called \emph{maximal} if for all unbounded $C\sub\kappa$, there is $A\in\mathscr{F}$ such that $A\cap C$ is unbounded in $\kappa$.
A \emph{mad family} just means a maximal almost disjoint family.
It was shown by Mathias in \cite[Corollary 4.7]{happy_families}
that there is no  $\bfSigma^1_1$
infinite mad family at $\om$,
and shown by Miller in \cite[Theorem 8.23]{miller}
that if $V=L$ then there is a $\Pi^1_1$ infinite mad family at $\om$.
In this section we establish variants of these results, considering almost disjoint families at cardinals $\kappa>\om$.
 L\"ucke and M\"uller already
 proved an analogue \cite[Theorem 1.3]{Sigma_1_def_at_higher_cardinals} of Mathias' result, showing that if $\kappa$ is an iterable cardinal which is a limit of measurables then there is no $\Sigma_1(\her_\kappa\cup\{\kappa\})$
almost disjoint family $\mathscr{F}$ at $\kappa$ such that $\mathscr{F}$ has cardinality $>\kappa$,
and hence no $\Sigma_1(\her_\kappa\cup\{\kappa\})$ mad family $\mathscr{F}$ at $\kappa$
such that $\mathscr{F}$ has cardinality $\geq\kappa$.
We extend this, showing in Theorem~\ref{tm:kappa_limit_of_measurables_implies_no_Sigma_1_mad_family} that if $\kappa$ a limit of measurable cardinals then:
\begin{enumerate}
 \item[--] there is no $\Sigma_1(\her_\kappa\cup\OR)$ \emph{mad} family $\mathscr{F}\sub\pow(\kappa)$ of cardinality $\geq\kappa$, and
 \item[--] if $\cof(\kappa)>\om$
 then there is no $\Sigma_1(\her_\kappa\cup\OR)$ \emph{almost disjoint} family $\mathscr{F}\sub\pow(\kappa)$ of cardinality $>\kappa$.
\end{enumerate}
In the other direction,
we establish an analogue of Miller's result, and in so doing,
partially address \cite[Question 10.5(ii)]{Sigma_1_def_at_higher_cardinals},
which asked whether sufficiently strong large cardinal properties of $\kappa$ imply that there is no $\Pi_1(\her_\kappa\cup\{\kappa\})$ almost disjoint family of cardinality $>\kappa$.
We  show in Theorem~\ref{tm:regular_kappa_Pi_1_kappa_mad_families} that it is in fact  consistent relative to large cardinals that $\kappa$ is a regular cardinal, with large cardinal properties up to a Woodin limit of Woodin cardinals, and there is a $\Pi_1(\{\kappa\})$ \emph{mad} family of cardinality $>\kappa$.

In \S\ref{sec:indep_families}
we establish in Theorem~\ref{tm:regular_kappa_Pi_1_kappa_mi_families} an analogue of Theorem~\ref{tm:regular_kappa_Pi_1_kappa_mad_families}, giving the consistency relative to large cardinals of cardinals $\kappa$ with large cardinal properties,
together with the existence of $\Pi_1(\{\kappa\})$
maximal independent families  $\sub\pow(\kappa)$ of cardinality $>\kappa$.

In \S\ref{sec:ultrafilters}
we make a simple observation
on the definability of ultrafilters over $\kappa$ which are limits of measurables,
and one
on the definability of the club filter on regular $\kappa$ such that $\kappa>\mu$ for some measurable $\mu$.

Finally in \S\ref{sec:global}
we consider a somewhat different theme. We adapt a result of Schlicht \cite[Theorem 2.19]{schlicht_perfect_set_property}, proving the relative consistency
of the theory ZFC + ``$\kappa$ is $\kappa^+$-supercompact'' + ``for every $X\sub\pow(\kappa)$ with $X\in\HOD_{V_{\kappa+1}}$
and $X$ of cardinality $>\kappa$,
we have that $X$ has a perfect subset and $X$ is not a wellorder. (The methods involved in this adaptation are some standard forcing techniques.)

In \S\ref{sec:questions} we collect some questions together.

\subsection{Notation}\label{subsec:notation}
 ZFC is the background theory for the paper.

For classes $X,Y$, we say $Y$ \emph{is $\Sigma_n(X)$} (or $Y$ \emph{is $\Sigma_n(X)$-definable}) iff there is a $\Sigma_n$ formula $\varphi$
and elements $p_0,\ldots,p_{n-1}\in X$
such that $Y=\{y\bigm|\varphi(p_0,\ldots,p_{n-1},y)\}$. Likewise for $\Pi_n(X)$. And $Y$ is $\Delta_n(X)$ iff it is both $\Sigma_n(X)$ and $\Pi_n(X)$.

For an ordinal $\eta$
let $[\eta]^\om_{\uparrow}$ denote the set of all subsets of $\eta$ of ordertype $\om$.

Recall that \emph{$\mu$-closed}
and \emph{$\mu$-steady}
were defined at the beginning of \S\ref{sec:intro}.

 By a \emph{measure}, we mean a countably complete non-principal ultrafilter,
 and by a \emph{normal measure},
 we mean a normal $\kappa$-complete  non-principal ultrafilter on an uncountable cardinal $\kappa$.

 For an uncountable cardinal $\kappa$, $\her_\kappa$
 denotes the set of all sets hereditarily of cardinality $<\kappa$.

The universe of a structure $M$
is denoted $\univ{M}$.

Recall from \cite{outline} that a premouse $M$
has form $M=(\J_\alpha[\es],\es,F)$ where $\alpha$ is a limit ordinal or $\alpha=\OR$, $\es$ is a sequence of (partial) extenders with certain properties, and either $F=\emptyset$ or $F$ is an extender over $\J_\alpha[\es]$ with certain properties. We write $\es^M=\es$,
$F^M=F$ and $\es_+^M=\es^M\conc\left<F^M\right>$. We say $F$ is the \emph{active extender} of $M$,
and say $M$ is \emph{active} if $F\neq\emptyset$. For limit ordinals $\beta\leq\OR^M$,
we write $M|\beta$ for the initial segment of $M$ of ordinal height $\beta$, including the extender indexed at $\beta$, if there is one; that is, $M|\beta=(\J_\beta[\es^M],\es^M\rest\beta,E)$
where $E=\es_+^M(\beta)$. We write $M||\beta=(M|\beta)^{\passive}$ for $(\J_\beta[\es^M],\es^M\rest\beta,\emptyset)$.
If $\beta\in\OR^M$ is an $M$-cardinal
then $\es^M(\beta)=\emptyset$. We write $<^M$ for the standard order of constructibility of $M$
(this is a wellorder of $\univ{M}$,
and ${<^{M||\beta}}={<^{M|\beta}}$ is an initial segment of $<^M$, for all $\beta\leq\OR^M$).
Actually we officially use the definition of \emph{premouse} of \cite[\S1.1]{V=HODX_pub}, and in particular,
allow extenders of superstrong type to appear in $\es_+^M$ (though these will only be relevant for some of the work).

We also consider $L[U]$ (where $U$ is a filter over some $\kappa$ such that $L[U]\sats$ ``$U$ is a normal measure'') to be a premouse in the above sense, so $\es^{L[U]}$
includes $U$ as its ultimate element, but also includes  many partial measures prior to this (both on $\kappa$
and unboundedly many ordinals $<\kappa$).

\section{Linear Iterations of Measures}
\label{sec:linear_iterations_of_measures}

In this section we collect some facts regarding linear iterations of measures which we will need.
The material in \S\ref{sec:iterations_single_measure} is standard, and is very similar to some material in \cite{reinhardt_iterates}.
\subsection{Iterations of a Single Measure}\label{sec:iterations_single_measure}
\begin{dfn}\label{dfn:linear_iteration}
Let $U$ be a measure.
Then $\Tt_U=\left<M_\alpha,U_\alpha\right>_{\alpha\in\OR}$ denotes the length $\OR$ iteration of $V$ via $U$ and its images. That is, $M_0=V$, $U_0=U$, $M_{\alpha+1}=\Ult(M_\alpha,U_\alpha)$, and letting $i_{\alpha\beta}:M_\alpha\to M_\beta$ be the iteration map, \begin{equation*}\label{eqn:label_for_eqn_1} U_{\alpha+1}=i_{0,\alpha+1}(U)=i_{\alpha,\alpha+1}(U_\alpha),\end{equation*}
and for limit $\lambda$, $M_\lambda$
is the direct limit of the earlier $M_\alpha$ under these iteration maps. We also write $M^{\Tt_U}_\alpha=M_\alpha$, $i^{\Tt_U}_{\alpha\beta}=i_{\alpha\beta}$, etc.
\end{dfn}

In the following lemma, recall that a cardinal $\kappa$ is \emph{$\mu$-closed}
iff $\alpha^\mu<\kappa$ for all $\alpha<\kappa$,
and \emph{$\mu$-steady}
iff there is a cardinal $\nu\leq\kappa$ such that $\nu$ is $\mu$-closed, $\cof(\nu)\neq\mu$,
and $\kappa\in\{\nu,\nu^+\}$.
\begin{lem}\label{lem:many_U-it_fixed_points_above_meas}
 Assume ZFC and let $\mu<\kappa$
 be cardinals such that $\mu$ is measurable and $\kappa$ is $\mu$-closed or $\mu$-steady.  Let $U$ be a $\mu$-complete measure on $\mu$. Then:
 \begin{enumerate}\item\label{item:it_maps``kappa_sub_kappa} $i^{\Tt_U}_{0\lambda}``\kappa\sub\kappa$ for each $\lambda<\kappa$.\item\label{item:nicely_stable} There are unboundedly many  $\mu'<\kappa$
 such that $i^{\Tt_U}_{0\mu'}(\mu)=\mu'$
 and $i^{\Tt_U}_{0\lambda}(\mu')=\mu'$ for all $\lambda<\mu'$.
 \end{enumerate}
\end{lem}
\begin{proof}
Part~\ref{item:it_maps``kappa_sub_kappa}:
 Let $\alpha,\lambda<\kappa$.
 If $\kappa$ is $\mu$-closed then
 \begin{equation*}\label{eqn:label_for_eqn_2}\card(i^{\Tt_U}_{0\lambda}(\alpha))\leq\card(\alpha)^\mu\cdot(\card(\lambda+1))<\kappa.\end{equation*}
 Suppose  $\kappa$ is non-$\mu$-closed.
 Then $\kappa$ is $\mu$-steady,
 and so $\kappa=\nu^+$
 where $\nu$ is $\mu$-closed
 and $\cof(\nu)\neq\mu$
 but $\nu^\mu\geq\kappa$.
 Note then that $\cof(\nu)<\mu$,
 but then $\card(i^{\Tt_U}_{0\lambda}(\nu))\leq\card(\nu\cdot(\lambda+1))=\nu$,
 since to form $i^{\Tt_U}_{0\lambda}(\nu)$, it suffices to consider only functions $f:[\mu]^{<\om}\to\nu$ which are bounded in $\nu$, of which there are only $\nu$-many.
 So $i^{\Tt_U}_{0\lambda}(\nu)<\nu^+=\kappa$.
 But then $i^{\Tt_U}_{0\lambda}(\alpha)<\kappa$ also.

 Part~\ref{item:nicely_stable}:
Fix $\lambda<\kappa$;
 we want to find $\mu'\in[\lambda,\kappa)$ with the right properties.
 Let $\mu_0=i^{\Tt_U}_{0,\lambda+1}(\mu)$ (so $\max(\mu,\lambda+1)<\mu_0$),
 and given $\mu_n$ where $n<\om$, let $\mu_{n+1}=i^{\Tt_U}_{0\mu_n}(\mu_n)$.
 Then $\mu_n<\mu_{n+1}<\kappa$, by part~\ref{item:it_maps``kappa_sub_kappa} and since, for example, $\max(\mu,\lambda+1)<\mu_0$, so
 \begin{equation*}\label{eqn:label_for_eqn_3} \mu_1=i^{\Tt_U}_{0\mu_0}(\mu_0) >
i^{\Tt_U}_{0,\lambda+1}(\mu)=\mu_0, \end{equation*}
and similarly for larger $n$.
Let $\mu'=\sup_{n<\om}\mu_n$.
 Some straightforward cardinal arithmetic shows that $\mu'<\kappa$. (If $\cof(\kappa)=\om$ then $\kappa$ is a $\mu$-closed limit cardinal, and note that for all $n$, we have $\card(\mu_n)\leq\theta$
 where $\theta=\card(\mu^\mu\cdot(\lambda+1)^\mu)<\kappa$, so $\mu'<\theta^+$.)
 Note that $\mu'$ works (certainly $i^{\Tt_U}_{0\beta}(\mu)\geq\beta$ for all $\beta$, and the other requirements follow from the construction).
\end{proof}
\begin{dfn}\label{dfn:2.3}
 For an ordinal $\alpha$,
 the \emph{eventual ordertype}
 $\eot(\alpha)$ of $\alpha$
 is the least ordinal $\eta$
 such that for some $\beta<\alpha$,
 we have $\alpha=\beta+\eta$.
\end{dfn}

\begin{dfn}\label{dfn:2.4}
 Let $U$ be a measure. Given a set $X$ and ordinals $\alpha\leq\beta$
 with $X\in M^{\Tt_U}_\gamma$ for all $\gamma<\beta$,
 we say that $X$ is \emph{$U$-$[\alpha,\beta)$-stable}
 iff
  $i^{\Tt_U}_{\alpha\gamma}(X)=X$ for all $\gamma\in[\alpha,\beta)$. If $U$ is determined by context, we just say \emph{$[\alpha,\beta)$-stable},
 and often we will have a cardinal $\kappa$ also fixed, and then \emph{\tu{(}$U$\tu{)}-$\alpha$-stable} means \emph{\tu{(}$U$\tu{)}-$[\alpha,\kappa)$-stable}.
\end{dfn}
The following is well-known:
\begin{fact}\label{fact:every_ord_ev_stable}
Let $U$ be a measure.
 Let $\eta$ be a limit ordinal and $\xi$ any ordinal.
 Then there is $\alpha<\eta$ such that $\xi$ is $U$-$[\alpha,\eta)$-stable.
 (Otherwise $M^{\Tt_U}_\eta$ is illfounded.)
\end{fact}

\begin{dfn}\label{dfn:*_eta}
 Let $U$ be a  measure
 and $\eta$  a limit ordinal.
 Then $*_{\eta,U}:\OR\to\OR$ denotes the map
 $\alpha\mapsto\alpha^*$ where for ordinals $\alpha$ we define
 $\alpha^*=i^{\Tt_U}_{\beta\eta}(\alpha)$
 for any/all $\beta<\eta$
 such that $\alpha$ is $[\beta,\eta)$-stable. We just write $*_{\eta}$
  if $U$ is clear from context,
  and just write $\alpha^*$ for $*_\eta(\alpha)$ if $\eta$ is also clear.
\end{dfn}

A straightforward calculation shows the following fact:
\begin{lem}\label{lem:*_def}
Let $U,\eta$ be as in Definition~\ref{dfn:*_eta}. Let $\tau=\eot(\eta)$
and $\tau'=*_\eta(\tau)$.
Then for all $\alpha\in\OR$, we have
\begin{equation*}\label{eqn:label_for_eqn_4}
 *_\eta(\alpha)=i^{\Tt_U}_{\eta,\eta+\tau'}(\alpha)=i^{\Tt_{U_\eta}}_{0\tau'}(\alpha), \end{equation*}
 and hence, $*_\eta$ is definable over $M^{\Tt_U}_\eta$ from the parameters $U_\eta$ and $\tau'$.
\end{lem}

\begin{lem}\label{lem:charac_sets_in_M_eta_limit_eta}
Let $U$ be a measure.
 Let $\eta$
  be a limit ordinal.
  Let $A\sub\OR$
  be a set. Then $A\in M^{\Tt_U}_\eta$
  iff $A\in\bigcap_{\alpha<\eta}M^{\Tt_U}_\alpha$
  and there is $\beta<\eta$
  such that $A$ is $U$-$[\beta,\eta)$-stable.
  \end{lem}
  \begin{proof}
  Let $\tau=\eot(\eta)$.
   Suppose $A\in M^{\Tt_U}_\eta$.
 Let $\beta<\eta$ be such that $A\in\rg(i^{\Tt_U}_{\beta\eta})$,
 $\beta+\tau=\eta$
 and $\tau$ is $[\beta,\eta)$-stable.
Let us show that $A$ is $[\beta,\eta)$-stable. For
 $\gamma\in[\beta,\eta)$
 let $A=i^{\Tt_U}_{\gamma\eta}(A_\gamma)$.
 Then note that for such $\gamma$,
 we have \begin{equation}\label{eqn:A_in_M^T_gamma}M^{\Tt_U}_\gamma\sats\text{``}A=i^{\Tt_{U_\gamma}}_{0\tau}(A_\gamma)\text{''}.\end{equation}
 But then applying $i^{\Tt_U}_{\gamma\delta}$ to this when
 $\delta\in[\gamma,\eta)$,
 and letting $A'=i^{\Tt_U}_{\gamma\delta}(A)$,
 we have \begin{equation}\label{eqn:A'_in_M^T_delta}M^{\Tt_U}_\delta\sats\text{``}A'=i^{\Tt_{U_\delta}}_{0\tau}(A_\delta)\text{''.}\end{equation}
  But then comparing lines (\ref{eqn:A_in_M^T_gamma}) (where we can now replace $\gamma$ with $\delta$) and (\ref{eqn:A'_in_M^T_delta}),
 we have $A'=A$.

 Now suppose that $A\in\bigcap_{\alpha<\eta}M^{\Tt_U}_\alpha$ and there is $\beta<\eta$ such that $A$ is $[\beta,\eta)$-stable.
 We show that $A\in M^{\Tt_U}_\eta$.
 We may take $\beta<\eta$
 such that $A$ is $[\beta,\eta)$-stable, $\eta=\beta+\tau$ and $\tau$ is $[\beta,\eta)$-stable.
 Let $A^*=i^{\Tt_U}_{\beta\eta}(A)$,
 so $A^*=i^{\Tt_U}_{\gamma\eta}(A)$ for all $\gamma\in[\beta,\eta)$.
 Then for all $\alpha\in\OR$,
 we have $\alpha\in A$ iff $\alpha^*\in A^*$. Since $A^*\in M^{\Tt_U}_\eta$,
 by Lemma~\ref{lem:*_def},
 we can deduce that $A\in M^{\Tt_U}_\eta$.
  \end{proof}

  \subsection{Iterations of Multiple Measures}
We will now  consider iterations of measures more general than those in \S\ref{sec:iterations_single_measure}. We will  no longer require that $E^\Tt_\alpha=i^\Tt_{0\alpha}(U)$ for some fixed $U$, nor even that $E^\Tt_\alpha\in\rg(i^\Tt_{0\alpha})$.
  \label{sec:iterations_mult_meas}
\begin{dfn}\label{2.9}
 Say that a linear iteration $\Tt$ on $V$ is a \emph{linear iteration of measures}
 iff $M^\Tt_\alpha\sats$ ``$E^\Tt_\alpha$ is a measure (not necessarily normal)'' for all $\alpha+1<\lh(\Tt)$. For a linear iteration of measures, say that $\Tt$ is \emph{above $\mu$}
 if $\mu\leq\crit(i^\Tt)$
 (equivalently, $M^\Tt_\alpha\sats$ ``$E^\Tt_\alpha$ is $\mu$-complete'' for all $\alpha+1<\lh(\Tt)$),
 and \emph{based on $V_\delta$}
 iff $E^\Tt_\alpha\in V_{i^\Tt_{0\alpha}(\delta)}^{M^\Tt_\alpha}$ for all $\alpha+1<\lh(\Tt)$.
\end{dfn}

The following lemma is just a very slight variant on a well-known fact (see Kunen \cite{Kunen_model_negation_AC},
and very related calculations in  \cite[Lemma 4.5]{steel_dmt} and \cite[Lemma 3.17]{vmom_v2}, for example). We give  details for self-containment.
\begin{lem}\label{lem:embedding_agreement}
 Assume ZFC.
 Let  $\Tt$ and $\Uu$ be linear iterations of measures
 on $V$, of successor lengths $\alpha+1$ and $\beta+1$ respectively.
 Suppose there is an inaccessible cardinal $\mu$ such that
 $\alpha<\mu$,
 $\Tt$ is based on $V_\mu$,
 and
 $\Uu$ is above $\mu$.
 Let $V'=M^\Tt_\infty$ and
 $\Uu'=i^{\Tt}\Uu$ \tu{(}the copy\footnote{\label{ftn:copying}See \cite[\S4.1]{outline}, though that version is much more general. Define $\pi_0=i^\Tt$.
 In general we set $E^{\Uu'}_\alpha=\pi_\alpha(E^\Uu_\alpha)$, then define $\pi_{\alpha+1}:M^\Uu_{\alpha+1}\to M^{\Uu'}_{\alpha+1}$ as the natural map with $\pi_{\alpha+1}\com i^\Uu_{\alpha,\alpha+1}=i^{\Uu'}_{\alpha,\alpha+1}\com\pi_\alpha$ (see \cite[Lemma 4.2]{outline}),
 and for limit $\lambda$,
 $\pi_\lambda:M^\Uu_\lambda\to M^{\Uu'}_\lambda$ similarly commutes with the earlier $\pi_\alpha$'s and iteration maps.
 We set $\lh(\Uu')=\lh(\Uu)$.
 If $\lh(\Uu)<\crit(i^\Tt)$
 then $\Uu'=i^\Tt(\Uu)$
 and $\pi_\alpha=i^\Tt\rest M^\Uu_\alpha$, but if $\lh(\Uu)\geq\crit(i^\Tt)$ then this breaks down.} of $\Uu$ under $i^\Tt$, to an iteration on $V'$\tu{)}.
 Then $i^\Uu(V')=M^{\Uu'}_\infty$
 and $i^\Uu\rest V'=i^{\Uu'}$.
\end{lem}

By the usual convention,  $i^\Uu(V')=\bigcup_{\alpha\in\OR}i^\Uu(V_\alpha\cap V')$.
\begin{proof}
Let $\nu$ be a regular cardinal with $\Uu\in V_\nu$. Let $\nu'=i^\Tt(\nu)=\sup i^\Tt``\nu$.
Every $x\in M^{\Uu}_\infty$
has the form
$x=i^\Uu(f)(a)$
for some $f:[\nu]^{<\om}\to V$ and $a\in[i^\Uu(\nu)]^{<\om}$.

Let $k=i^\Tt\rest\nu:\nu\to \nu'=i^\Tt(\nu)$. Let $\pi_\infty:M^{\Uu}_\infty\to M^{\Uu'}_\infty$ be the final copy map induced by $\pi_0=i^\Tt:V\to V'$
and the copying of $\Uu$ to $\Uu'$ (see Footnote~\ref{ftn:copying}).
Recall that
\begin{equation*}\label{eqn:label_for_eqn_5} \pi_\infty\com i^\Uu=i^{\Uu'}\com i^\Tt.\end{equation*}
Every  $x\in M^{\Uu'}_\infty$ has the form
$x=i^{\Uu'}(f)(b)$
for some $f\in V'$ with $f:\nu'\to V'$
and some $b\in[i^{\Uu'}(\nu')]^{<\om}$.
In fact, because $\Uu'$ arises from copying $\Uu$, which uses only measures, we can always take $b\in\rg(\pi_\infty)$ here; so $x$ has form  $x=i^{\Uu'}(f)(\pi_\infty(a))$ for some such $f\in V'$
and $a\in[i^\Uu(\nu)]^{<\om}$.

Using these observations, we define a map
\begin{equation*}\label{eqn:label_for_eqn_6} \sigma:M^{\Uu'}_\infty\to i^\Uu(V'), \end{equation*}
as follows.
For $f\in V'$ with $f:\nu'\to V'$ and for $a\in[i^\Uu(\nu)]^{<\om}$, let
\begin{equation*}\label{eqn:label_for_eqn_7} \sigma\big(i^{\Uu'}(f)(\pi_\infty(a))\big)=i^{\Uu}(f\com k)(a).\end{equation*}

Let us verify that $\sigma$ is well-defined. Let $f,g\in V'$
with $f,g:\nu'\to V'$, let $a,b\in[i^\Uu(\nu)]^{<\om}$,
and suppose
\begin{equation*}\label{eqn:label_for_eqn_8} i^{\Uu'}(f)(\pi_\infty(a))=i^{\Uu'}(g)(\pi_\infty(b)).\end{equation*}
We must see that
\begin{equation}\label{eqn:for_pi_well-def} i^{\Uu}(f\com k)(a)=i^\Uu(g\com k)(b).\end{equation}
By slightly modifying $f,g,a,b$, we may and will assume that $a=b$. So then
\begin{equation*}\label{eqn:label_for_eqn_9} X\eqdef\{u\in[\nu']^{<\om}\bigm|f(u)=g(u)\}\in E'_{\pi_\infty(a)},\end{equation*}
where $E'_{\pi_\infty(a)}$ is the measure over $[\nu']^{<\om}$ derived from $i^{\Uu'}$ with seed $\pi_\infty(a)$; that is,
\begin{equation*}\label{eqn:label_for_eqn_10} A\in E'_{\pi_\infty(a)}\iff A\in V'\wedge A\sub[\nu']^{<\om}\wedge \pi_\infty(a)\in i^{\Uu'}(A).\end{equation*}
Let $E_a$ be the measure over $[\nu]^{<\om}$
derived from $i^\Uu$.

We claim that
$i^\Tt(E_a)=E'_{\pi_\infty(a)}$.
For $i^\Tt``E_a\sub E'_{\pi_\infty(a)}$,
since given any $A\in E_a$, we have $a\in i^\Uu(A)$, so
\begin{equation*}\label{eqn:label_for_eqn_11}\pi_\infty(a)\in \pi_\infty(i^\Uu(A))=i^{\Uu'}(i^\Tt(A)), \end{equation*}
so $i^\Tt(A)\in E'_{\pi_\infty(a)}$.
But because $E_a$ is $\mu$-complete
and $\Tt\in V_\mu$,
given any $B\in i^\Tt(E_a)$,
there is $A\in E_a$ with $i^\Tt(A)\sub B$,
and therefore $B\in E'_{\pi_\infty(a)}$.
So $i^\Tt(E_a)\sub E'_{\pi_\infty(a)}$,
so in fact we have equality, as desired.

Now since $X\in E'_{\pi_\infty(a)}$,  there is $Y\in E_a$ with $i^\Tt(Y)\sub X$.
But then for each $\alpha\in Y$,
we have $i^\Tt(\alpha)\in X$, and so
$(f\com k)(\alpha)=f(i^\Tt(\alpha))=g(i^\Tt(\alpha))=(g\com k)(\alpha)$. Since $Y\in E_a$ and $a=b$, line~(\ref{eqn:for_pi_well-def}) follows.

So $\sigma$ is well-defined,
and similarly, $\sigma:M^{\Uu'}_\infty\to i^\Uu(V')$ is elementary.

We claim now that $\sigma$ is surjective.
For let $x\in i^\Uu(V')$;
we will show $x\in\rg(\pi)$. Fix
 $h:\nu\to V'$ and $a\in[i^\Uu(\nu)]^{<\om}$ with $x=i^\Uu(h)(a)$.
Note that since $\mu$ is inaccessible, we can fix $\delta<\mu$ such that $\Tt$ only uses measures from $V_\delta$ and its images.
For each $\alpha<\nu$,
let $(f_\alpha,b_\alpha)$
be such that $f_\alpha:[\delta]^{<\om}\to V$ and $b_\alpha\in [i^\Tt(\delta)]^{<\om}$ and
 $h(\alpha)= i^\Tt(f_\alpha)(b_\alpha)$. By $\mu$-completeness, and since $i^\Tt(\delta)<\mu$,
 we can fix $Y\in E_a$
 and $b$ such that for all $\alpha\in Y$, we have $b_\alpha=b$.
 Now define $\wt{h}:[\delta]^{<\om}\to V$
 by setting $\wt{h}(x)$
 to be the function  $\wt{h}(x):\nu\to V$ where $\wt{h}(x)(\alpha)=f_\alpha(x)$.
 Let  $f=i^\Tt(\wt{h})(b)$.
 Then note that $f\in V'$ and $f:\nu'\to V'$ and for all $\alpha\in Y$, we have \begin{equation*}\label{eqn:label_for_eqn_12}(f\com k)(\alpha)=i^\Tt(\widetilde{h})(b)\big(i^\Tt(\alpha)\big)=i^\Tt(f_\alpha)(b)=h(\alpha).\end{equation*}
 Therefore
 $\sigma(i^{\Uu'}(f)(\pi_\infty(a)))=i^\Uu(f\com k)(a)=i^\Uu(h)(a)=x$,
which suffices.

 We have shown that $\sigma$ is an isomorphism, so $i^\Uu(V')=M^{\Uu'}_\infty$.
 Finally, the fact that $i^\Uu\rest V'=i^{\Uu'}$ is just because for $x\in V'$, we have
 \begin{equation*}\label{eqn:label_for_eqn_13} i^{\Uu'}(x)=\sigma(i^{\Uu'}(x))=\sigma(i^{\Uu'}(c_x)(\emptyset))=i^\Uu(c_x\com k)(\emptyset)=i^\Uu(x)\end{equation*}
 where $c_x:\nu'\to V'$ is the constant function $c_x(\alpha)=x$.
 \end{proof}

 \begin{dfn}\label{dfn:kappa-weakly_normal}Let $\kappa$ be a strong limit cardinal. Let $\Tt$ be an iteration of measures which is based on $V_\kappa$ and has length $\lambda+1\leq\kappa+1$. We say that $\Tt$ is \emph{$\kappa$-weakly normal}
 iff either:
 \begin{enumerate}
 \item[(i)] $\lambda<\kappa$ and $\Tt$ is based on $V_\delta$ for some $\delta<\kappa$, or
 \itemref[item:weakly_normal_type_ii]{(ii)} we have:
 \begin{enumerate}
 \item $\lambda$ is a limit ordinal $\leq\kappa$,
 \item for every $\alpha<\lambda$, $\Tt\rest\alpha$ is based on $V_\delta$ for some $\delta<\kappa$,
 \item for every $\gamma<\kappa$
 there is $\alpha<\lambda$ such that $i^\Tt_{0\alpha}(\gamma)<\crit(i^\Tt_{\alpha\lambda})$.\qedhere
 \end{enumerate}
 \end{enumerate}
 \end{dfn}

 Note that $i^\Tt_{0\lambda}``\kappa\sub\kappa$ for such iterations $\Tt$.

\begin{lem}\label{lem:ordinals_ev_stable_meas_limit}
 Assume ZFC and let $\kappa$ be a limit of measurable cardinals.
 Let $\xi\in\OR$.
 Then there is $\mu<\kappa$
 such that $i^\Tt(\xi)=\xi$ for all
 linear iterations of measures $\Tt$ on $V$ of successor length $\lambda+1$
 such that $\Tt$ is based on $V_\kappa$,
 has length $\lambda+1\leq\kappa+1$,
 is
 $\kappa$-weakly normal, and is above $\mu$.
 \end{lem}
\begin{proof}
Let us first prove the lemma restricted to iterations which
are also of length $\lambda+1<\kappa$ and based on $V_\delta$ for some $\delta<\kappa$.
 Suppose this version fails. Then we can pick a sequence $\left<\mu_n,\delta_n,\Tt_n\right>_{n<\om}$
 such that:
 \begin{enumerate}
  \item[--] $\mu_n$ is inaccessible and $\mu_n<\delta_n<\mu_{n+1}$,
  \item[--] $\Tt_n$ is a linear iteration of measures on $V$, is above $\mu_n$ and  based on $V_{\delta_n}$, and has length $<\mu_{n+1}$, and
  \item[--] $i^{\Tt_n}(\xi)>\xi$.
 \end{enumerate}
Let $\Tt_1'=i^{\Tt_0}\Tt_1$,
let $\Tt_2'=i^{\Tt_1'}\com i^{\Tt_0}\Tt_2$, etc., defining $\Tt_n'$ for all $n<\om$. Note that $\Uu=\Tt_0\conc\Tt_1'\conc\Tt_2'\conc\ldots$ is a linear iteration on $V$.
By Lemma~\ref{lem:embedding_agreement},
$i^{\Tt_n'}(\xi)=i^{\Tt_n}(\xi)>\xi$.
But then $M^{\Uu}_\infty$ is illfounded, a contradiction.

We now prove the full lemma.
For $\xi\in\OR$,
let $\mu_\xi$  be the least witness for the restricted version of the lemma with respect to $\xi$. We claim that $\mu^*_\xi\eqdef\max(\mu_\kappa,\mu_\xi)+1$  witnesses the full lemma.
So let  $\Tt$ be an iteration of the form considered for the full lemma, which is above $\mu^*_\xi$;
we want to see  $i^\Tt(\xi)=\xi$.
Let $\lambda+1=\lh(\Tt)$.
We may assume that $\Tt$ is not based on $V_\delta$ for any $\delta<\kappa$.
(Suppose otherwise.
If $\lambda=\kappa$
then we get $i^\Tt_{0\lambda}(\delta)\geq\kappa$, contradicting the assumption that $i^\Tt_{0\lambda}``\kappa\sub\kappa$.
So $\lambda<\kappa$,
but then the weak version of the lemma applies.) So clause~\ref{item:weakly_normal_type_ii}
of the definition of $\kappa$-weak normality holds.

\begin{clm*}For every $\alpha<\lambda$,
we have $i^\Tt_{0\alpha}(\xi)=\xi$,
and for
every $\eta<\xi$, there is $\beta\in(\alpha,\lambda)$ such that
\begin{equation}\label{eqn:fix_eta_image} i^\Tt_{\beta\gamma}(i^\Tt_{\alpha\beta}(\eta))=i^\Tt_{\alpha\beta}(\eta) \end{equation}
for all $\gamma\in(\beta,\lambda)$.
\end{clm*}

Assuming the claim, we can complete the proof. For note that  by the claim,
we also have $i^\Tt_{\alpha\beta}(\eta)<i^\Tt_{\alpha\beta}(\xi)=i^\Tt_{0\beta}(\xi)=\xi$.
Since $M^\Tt_\lambda$ is the direct limit of these $M^\Tt_\alpha$'s, every ordinal below $i^\Tt(\xi)$ has form $i^\Tt_{\alpha\lambda}(\eta)$ for some $(\alpha,\eta)$ as in the claim, so it follows
that $i^\Tt(\xi)=i^\Tt_{0\lambda}(\xi)=\xi$.

\begin{proof}[Proof of Claim]
Let $\alpha<\lambda$,  and fix $\delta_\alpha<\kappa$ such that
$\Tt\rest(\alpha+1)$ is based on $V_{\delta_\alpha}$.
Since  $\delta_\alpha$ exists
and
by the weak version of the lemma, $i^\Tt_{0\alpha}(\xi)=\xi$.
Let
 $\eta<\xi$. Let $\mu=(\mu^*_\eta)^{M^\Tt_\alpha}$
 (that is, $\mu=i^{\Tt}_{0\alpha}(f)(\eta)$ where $f(\beta)=\mu^*_\beta$). Note here that $i^\Tt_{0\alpha}(\kappa)=\kappa$
 and $i^\Tt_{0\alpha}(\mu_\kappa)=\mu_\kappa$ (since $\mu^*_\xi>\mu_\kappa$), so $(\mu^*_\eta)^{M^\Tt_\alpha}$ is defined relative to the same parameter $\kappa$ in $M^\Tt_\alpha$.
Fix $\beta+1\in(\alpha+1,\lambda)$
such that $\Tt\rest[\beta,\lambda+1)$ is above $i^\Tt_{\alpha\beta}(\mu)$.

We will show that $\beta$ witnesses the remainder of the claim for $(\alpha,\eta)$; that is,
line~(\ref{eqn:fix_eta_image})
holds
for all $\gamma\in(\beta,\lambda)$.
If for a given $\gamma$,
we had $\Tt\rest[\beta,\gamma+1)\in M^\Tt_\beta$, line~(\ref{eqn:fix_eta_image}) would follow immediately from the weak version of the lemma and the choice of $\mu$, etc.
But to handle the possibility that $\Tt\rest[\beta,\gamma+1)\notin M^\Tt_\beta$,
 we need a little more argument, which is as follows. The iteration $\Tt\rest[\beta,\gamma+1)$
can anyway be ``absorbed'' into some iteration $\Tt'\in M^\Tt_\beta$ (described in detail below),
which is above $i^\Tt_{\alpha\beta}(\mu)$, is based on some $V_\delta^{M^\Tt_\beta}$ with $\delta<\kappa$, and has length $<\kappa$,
and so \begin{equation}\label{eqn:fix_eta_image_prime}i^{\Tt'}(i^\Tt_{\alpha\beta}(\eta))=i^\Tt_{\alpha\beta}(\eta),\end{equation} again by our choice of $\mu$, etc. But this ``absorption''
yields  an elementary $\pi:M^{\Tt}_{\gamma}\to M^{\Tt'}_\infty$
with
$\pi\com i^{\Tt}_{\beta\gamma}=i^{\Tt'}$,
which, together with line~(\ref{eqn:fix_eta_image_prime}),
yields line~(\ref{eqn:fix_eta_image}).

The iteration $\Tt'$ is a typical instance of a ``universal'' iteration.
Recall that $\Tt\rest\gamma+1$
is based on $V_{\delta_\gamma}$
and has length $<\kappa$.
So $\Tt\rest[\beta,\gamma+1)$
is an iteration on $M^\Tt_\beta$,
which is based on $i^\Tt_{0\beta}(V_{\delta_\gamma})$,
has length some $\zeta+1<\kappa$,
and is above $i^\Tt_{\alpha\beta}(\mu)$.
Working in $M^\Tt_\beta$,
let $\Tt'$ be the linear iteration of measures which is based on $V^{M^{\Tt}_\beta}_{i^\Tt_{0\beta}(\delta_\gamma)}$,
is above $i^\Tt_{\alpha\beta}(\mu)$,
and uses every measure available ``sufficiently often''.
That is, define a continuous, strictly increasing sequence $\left<\lambda_\theta\right>_{\theta\leq\zeta}$
and define $\Tt'\rest\lambda_\theta+1$
by recursion on $\theta$,  as follows:
Set $\lambda_0=0$.
Given $\theta<\zeta$ and $\Tt'\rest\lambda_\theta+1$,
let $\left<F_\iota\right>_{\iota<\upsilon}$
enumerate all measures of $M^{\Tt'}_{\lambda_\theta}$ which are in $i^{\Tt'}_{0\lambda_\theta}(i^{\Tt}_{0\beta}(V_{\delta_\gamma}))$ and are $i^\Tt_{\alpha\beta}(\mu)$-complete in $M^{\Tt'}_{\lambda_\theta}$.
Set $\lambda_{\theta+1}=\lambda_\theta+\upsilon$,
and let $\Tt'\rest[\lambda_\theta,\lambda_{\theta+1}]$
be given by using the (images of the) $F_\iota$'s in turn; that is,
$E^{\Tt'}_{\lambda_\theta+\iota}=i^{\Tt'}_{\lambda_\theta,\lambda_\theta+\iota}(F_\iota)$. This completes the definition of $\Tt'$.

We can now define embeddings $\pi_\theta:M^{\Tt}_{\beta+\theta}\to M^{\Tt'}_{\lambda_\theta}$ via a natural copying: Set $\pi_0=\id$.
Given $\pi_\theta$, let $\iota$ be such that $\pi_\theta(E^\Tt_{\beta+\theta})=F_\iota$, using the enumeration of measures from stage $\theta$ as above.
Define $\pi'_\theta:M^{\Tt}_{\beta+\theta}\to M^{\Tt'}_{\lambda_\theta+\iota}$
to be $\pi'_\theta=i^{\Tt'}_{\lambda_\theta,\lambda_\theta+\iota}\com \pi_\theta$,
define $\widetilde{\pi}_{\theta+1}:M^{\Tt}_{\beta+\theta+1}\to M^{\Tt'}_{\lambda_\theta+\iota+1}$ to be the map induced by $\pi_\theta'$, so in fact $\widetilde{\pi}_{\theta+1}=\pi_\theta'\rest M^\Tt_{\beta+\theta+1}$; note we have
\begin{equation*}\label{eqn:label_for_eqn_14} M^\Tt_{\beta+\theta+1}=\Ult(M^\Tt_{\beta+\theta},E^\Tt_{\beta+\theta})\end{equation*}
and
\begin{equation*}\label{eqn:label_for_eqn_15} M^{\Tt'}_{\lambda_\theta+\iota+1}=\Ult(M^{\Tt'}_{\lambda_\theta+\iota},\pi'_\theta(E^\Tt_{\beta+\theta})).\end{equation*}
Then define $\pi_{\theta+1}=i^{\Tt'}_{\lambda_\theta+\iota+1,\lambda_{\theta+1}}\com \widetilde{\pi}_{\theta+1}$.
At limits $\theta$, define $\pi_\theta$
by commutativity as usual. This completes the construction of the embeddings $\pi_\theta$ for $\theta\leq\zeta$, and setting $\pi=\pi_\zeta$, we have $\pi:M^\Tt_\gamma\to M^{\Tt'}_\infty$
and $\pi\com i^\Tt_{\beta\gamma}=i^{\Tt'}$, as desired.
\end{proof}
As mentioned above, this completes the proof of the lemma.
\end{proof}

\section{Wellorders above Measurable Cardinals}
\label{sec:wellorders_above_measurable_cardinals}
\subsection{Above a Single Measurable}\label{sec:wo_above_one_meas}

Recall (see for example \cite{meas_cards_good_Sigma_1_wo})
that for a definability class $\Gamma$,
a wellorder $<^*$ of a set $X$ is \emph{$\Gamma$-good}
iff
the class
$\{(x,{<^*_x})\}_{x\in X}$
is $\Gamma$-definable, where
${<^*_x}=\{y\in X\bigm|y<^*x\}$.
If $<^*$ is $\Sigma_1(Y)$-good,
then clearly the relation ${<^*}$ is also $\Sigma_1(Y)$-definable.

In \cite{meas_cards_good_Sigma_1_wo},
assuming $V=L[U]$ where $U$ is a normal measure on a measurable cardinal $\mu$,
 L\"ucke and Schlicht characterized
those cardinals $\kappa$ such that there is a
$\Sigma_1(\{\kappa\})$-good wellorder of $\pow(\kappa)$:
 precisely those $\kappa$ such that $\kappa\leq\mu^+$
or $\kappa$ is non-$\mu$-steady. (They also analyzed the analogous question in the canonical model $L[U_0,U_1]$ for two measurables.)
In \cite[Question 2]{meas_cards_good_Sigma_1_wo}, they asked, still assuming $V=L[U]$,
given an uncountable cardinal $\kappa$,
whether there is a  $\Sigma_1(\{\kappa\})$ (but not necessarily $\Sigma_1(\{\kappa\})$-good) wellorder of $\pow(\kappa)$. (Of course, their result just mentioned already answered this question positively
in the case that $\kappa\leq\mu^+$ or $\kappa$ is non-$\mu$-steady.)

The question was answered by \cite[Theorem 7.1]{Sigma_1_def_at_higher_cardinals}, where L\"ucke and M\"uller proved the following facts, assuming ZFC, $\mu$ is  measurable
and
$\kappa>\mu^+$ is $\mu$-steady,
as witnessed by $\nu$ (so $\kappa\in\{\nu,\nu^+\}$, etc.).

\begin{enumerate}
 \itemref[item:no_wellorder_of_pow(mu)]{(i)} There is no $\Sigma_1(V_\mu\cup\{\nu,\nu^+\})$ wellorder
 of $\pow(\kappa)$.
 \itemref[item:no_injection_mu^+_to_pow(mu)]{(ii)} If $\cof(\kappa)>\om$ then
 there is no $\Sigma_1(V_\mu\cup\{\kappa\})$ injection from $\kappa^+$ into $\pow(\kappa)$.\footnote{One can also allow the parameter $\kappa^+$
 without affecting their proof;
 but if $\nu<\nu^+=\kappa$,
 allowing the parameter $\nu$ would be a problem for their proof.}
\end{enumerate}

We show in the following theorem that the hypothesis ``if $\cof(\kappa)>\om$'' is  not necessary in part~\ref{item:no_injection_mu^+_to_pow(mu)} of the L\"ucke-M\"uller result. We will also give a slightly different
(and more self-contained) proof of part~\ref{item:no_wellorder_of_pow(mu)},
which also yields the slightly stronger
fact that there is no such wellorder of $[\kappa]^\om$,
nor in fact of $[\eta]^\om_\uparrow$,
for a certain ordinal $\eta<\kappa$:
\begin{dfn}\label{dfn:3.1}For ordinals $\eta$,
 $[\eta]^\om_\uparrow$ denotes
 $\{A\in\pow(\eta)\mid A\text{ has ordertype }\om\}$).
\end{dfn}

\begin{tm}\label{tm:above_one_measurable}
 Assume ZFC and let $\mu$ be a measurable cardinal and $\kappa>\mu^+$ be a $\mu$-steady cardinal.
 Let $\nu>\mu$ be $\mu$-closed and such that $\cof(\nu)\neq\mu$ and $\kappa\in\{\nu,\nu^+\}$.
 Let $p\in V_\mu$ and $\varphi$
 be a $\Sigma_1$ formula.
 Let $U$ be a $\mu$-complete measure on $\mu$
 and $\eta=\crit(U_\om)=i^{\Tt_U}_{0\om}(\mu)$
 \tu{(}as in \ref{dfn:linear_iteration}\tu{)},
 so $\eta$ is the  sup of the ``critical sequence'' associated to $U$.
 For $\delta\in\OR$
 let $S_\delta$ be the class of $U$-$[0,\delta)$-stable sets;
 note that $V_\mu\cup\{\kappa,\kappa^+\}\sub S_\kappa$
 and $V_\mu\cup\{\nu,\nu^+,\kappa,\kappa^+,\eta\}\sub S_\om$.
 Then:
 \begin{enumerate}
  \item\label{item:no_wellorder_of_eta^om_uparrow} There is no $\Sigma_1( S_\om)$ wellorder of $[\eta]^\om_{\uparrow}$
  \tu{(}and hence no such wellorder of $\pow(\kappa)$\tu{)}.
  \item\label{item:no_injection_kappa^+_to_p(kappa)_allowing_countable_cof} There is no $\Sigma_1( S_\kappa)$ injective function $f:\kappa^+\to\pow(\kappa)$; in fact, for any $\Sigma_1(S_\kappa)$ set $f\sub\kappa^+\cross\pow(\kappa)$ and any club $C\sub\kappa^+$, $f\rest C$ is  not an injective function \tu{(}note there is no definability requirement on $C$\tu{)}.
  \item\label{item:wellorders_of_kappa}Every $\Delta_1(S_\kappa)$ subset of $\kappa$
  is in $M^{\Tt_U}_\kappa$ \tu{(}see \ref{dfn:linear_iteration}\tu{)};
  moreover, the supremum of lengths of all $\Sigma_1(V_\mu\cup\{\kappa,\kappa^+\})$ wellorders of $\kappa$ is $<\kappa^{+M^{\Tt_U}_\kappa}$.
  \item\label{item:functions_on_kappa,kappa^+} Every $\Delta_1(S_\kappa)$ class $C\sub\OR$ is
  definable from parameters over $M^{\Tt_U}_\kappa$, and in particular, $C\cap
  \alpha\in M^{\Tt_U}_\kappa$ for each $\alpha\in\OR$. Moreover,
  if
  $C$ is $\Delta_1(\{\vec{p}\})$ where $\vec{p}\in (S_\kappa)^{<\om}$ then for each ordinal $\beta$,
   $C\cap\beta$ is $\Delta_1^{M^{\Tt_U}_\kappa}(\{\vec{q},U_\kappa,\mathscr{F}\})$ where $(\vec{q},U_\kappa)=i^{\Tt_U}_{0\kappa}(\vec{p},U)$
  and $\mathscr{F}=(^{\kappa}\beta)\cap M^{\Tt_U}_{\kappa}$.
 \end{enumerate}
\end{tm}

\begin{proof}Part~\ref{item:no_wellorder_of_eta^om_uparrow}:
 Let $M_\alpha=M^{\Tt_U}_\alpha$, $i_{\alpha\beta}=i^{\Tt_U}_{\alpha\beta}$, and
  $\mu_\alpha=i_{0\alpha}(\mu)$.
 Suppose $<^*$ is a wellorder of $[\eta]^\om_{\uparrow}$,
 $\varphi$ is $\Sigma_1$,
 $\vec{p}\in (S_\om)^{<\om}$ and
 \begin{equation}\label{eqn:varphi_defines_<*}\all X,Y\in[\eta]^\om_{\uparrow}\ \Big[ X<^*Y\iff\varphi(\vec{p},X,Y)\Big].\end{equation}
 Write $<^*_n=i_{0n}({<^*})$.
 Since $i_{0n}(\vec{p},\eta)=(\vec{p},\eta)$ for all $n<\om$,
  $M_n\sats$ ``${<^*_n}$ is a wellorder of $[\eta]^\om_{\uparrow}$
 and line
 (\ref{eqn:varphi_defines_<*}) holds''.

 Since each $M_n$ is closed under $\mu$-sequences, we have $[\eta]^{\om}_{\uparrow}\cap M_n=[\eta]^{\om}_{\uparrow}$,
 and by the upward absoluteness of $\Sigma_1$, it follows that ${<}^*_{n}={<^*}$. So from now we just write ``${<^*}$'' instead of ``${<}^*_{n}$''.

 Let $\mu_n=\crit(i^{\Tt_U}_{n,n+1})$.
  Let $X_{\mathrm{even}}=\{\mu_{2n}\}_{n<\om}$ and $X_{\mathrm{odd}}=\{\mu_{2n+1}\}_{n<\om}$.
 Let $X_0=X_{\mathrm{even}}$.
 Let $X_1=\{\mu_0\}\cup X_{\mathrm{odd}}$.
 Let $X_2=\{\mu_0,\mu_1\}\cup X_{\mathrm{even}}$.
 In general let $X_{2k}=\{\mu_0,\ldots,\mu_{2k-1}\}\cup X_{\mathrm{even}}$
 and $X_{2k+1}=\{\mu_0,\ldots,\mu_{2k}\}\cup X_{\mathrm{odd}}$. Now $\crit(i_{n,n+1})=\mu_n$ and $i_{n,n+1}(\mu_k)=\mu_{k+1}$ for each $k\geq n$. Note that $X_i\in M_n$ for each $i,n<\om$, and that $i_{n+1,n+2}(X_{n+1})=X_n$.

 Since $X_{n+1}\neq X_n$, either $X_{n+1}<^*X_n$ or $X_n<^*X_{n+1}$. Since $<^*$ is wellfounded,
 there must therefore be some $n$ such that
 $X_n<^*X_{n+1}$.
 Fix such an $n$. Then letting $k_{n+1}=i_{n+1,n+2}$, since $k_{n+1}(X_{n+1})=X_n$, we have
  \begin{equation*}\label{eqn:label_for_eqn_16}M_{n+1}\sats\text{``}k_{n+1}(X_{n+1})<^*X_{n+1}\text{''}.\end{equation*}
 Applying $k_{n+1}$ to this statement
 gives (since $k_{n+1}(k_{n+1})=k_{n+2}$) that
 \begin{equation*}\label{eqn:label_for_eqn_17} M_{n+2}\sats\text{``}k_{n+2}(k_{n+1}(X_{n+1}))<^*k_{n+1}(X_{n+1})\text{''}.\end{equation*}
 Now applying $k_{n+2}$ to the latter statement yields (since $k_{n+2}(k_{n+2})=k_{n+3}$)
 \begin{equation*}\label{eqn:label_for_eqn_18} M_{n+3}\sats\text{``}k_{n+3}(k_{n+2}(k_{n+1}(X_{n+1})))<^*k_{n+2}(k_{n+1}(X_{n+1}))\text{''},\end{equation*}
 and so on.
 This produces an infinite descending sequence given by the sequence of images of $X_{n+1}$
 under these maps, a contradiction.

 Part~\ref{item:no_injection_kappa^+_to_p(kappa)_allowing_countable_cof}:
 Let us first consider the case that $C=\kappa^+$.
 Suppose some injection $f:\kappa^+\to\pow(\kappa)$ is $\Sigma_1(S_\kappa)$.
As in \cite{Sigma_1_def_at_higher_cardinals},
it follows that $f\in M^{\Tt_U}_\alpha$
for each $\alpha<\kappa$.
So as in \cite{Sigma_1_def_at_higher_cardinals},
if $\cof(\kappa)>\om$
then $M^{\Tt_U}_{\kappa}=\bigcap_{\alpha<\kappa}M^{\Tt_U}_\alpha$, so $f\in M^{\Tt_U}_\kappa$,
which is impossible since $\pow(\kappa)\cap M^{\Tt_U}_\kappa$ has cardinality $\kappa$.
But now we want to generalize this
to allow the possibility that $\cof(\kappa)=\om$.

Well, given $\xi<\kappa^+$, we can fix $\alpha<\kappa$
such that $\xi$ is $[\alpha,\kappa)$-stable. Then note that $i^{\Tt_U}_{\alpha\beta}(f(\xi))=f(\xi)$ for all $\beta\in[\alpha,\kappa)$.
Let  $f(\xi)^*=i^{\Tt_U}_{\alpha\kappa}(f(\xi))$ for (any/all) such $\alpha$.
So $f(\xi)^*\in M^{\Tt_U}_\kappa$.
Now note that $f(\xi)^*\cap\kappa=f(\xi)$, because if $\gamma<\kappa$ then $\gamma\in f(\xi)$
iff, letting $\alpha\in(\gamma,\kappa)$ with $f(\xi)^*=i^{\Tt_U}_{\alpha\kappa}(f(\xi))$,
we have $i^{\Tt_U}_{\alpha\kappa}(\gamma)\in f(\xi)^*=i^{\Tt_U}_{\alpha\kappa}(f(\xi))$,
but $\gamma<\alpha\leq\crit(i^{\Tt_U}_{\alpha\kappa})$,
so this is iff $\gamma\in f(\xi)^*$,
as desired.
So $f(\xi)\in M^{\Tt_U}_\kappa$
for all $\xi<\kappa^+$,
which as in \cite{Sigma_1_def_at_higher_cardinals}
is a contradiction.

Now in general fix a club $C\sub\kappa^+$ and suppose $f\sub\kappa^+\cross\pow(\kappa)$
is $\Sigma_1(S_\kappa)$ and $f\rest C:C\to\pow(\kappa)$ is an injective function.
Let $C_\alpha=i^{\Tt_U}_{0\alpha}(C)$;
this is also  club in $\kappa^+$.
So $C'=\bigcap_{\alpha<\kappa}C_\alpha$ is also.
But then for each $\beta\in C'$,
we get $f(\beta)\in M^{\Tt_U}_\kappa$
as before, which is again a contradiction.

Part~\ref{item:wellorders_of_kappa}:
 To see that each $\Delta_1(S_\kappa)$ subset of $\kappa$ is in $M^{\Tt_U}_\kappa$, use the proof of part~\ref{item:no_injection_kappa^+_to_p(kappa)_allowing_countable_cof}.
 Now each $\Sigma_1(S_\kappa)$ wellorder of $\kappa$ is in fact $\Delta_1(S_\kappa)$.
 It follows immediately
 that each such wellorder has length $<\kappa^{+M^{\Tt_U}_\kappa}$.
 But then the supremum of their lengths is also $<\kappa^{+M^{\Tt_U}_\kappa}$,
 since $\cof(\kappa^{+M^{\Tt_U}_\kappa})=\mu^+$,
 and we only have $\mu$-many parameters in the set $V_\mu\cup\{\kappa,\kappa^+\}$ (note that we are not considering arbitrary $\Sigma_1(S_\kappa)$ wellorders of $\kappa$ here).

 Part~\ref{item:functions_on_kappa,kappa^+}:
 Let $\varphi,\psi$ be $\Sigma_1$
 formulas and $\vec{p}\in(S_\kappa)^{<\om}$ be such that \begin{equation*}\label{eqn:label_for_eqn_19}\all\alpha\in\OR\ \Big[\varphi(\vec{p},\alpha)\iff\neg\psi(\vec{p},\alpha) \Big].\end{equation*}
 Then by elementarity, for all $\beta<\kappa$, $M^{\Tt_U}_\beta$ satisfies the same statement,
 and so by the upward absoluteness of $\Sigma_1$, we have
 \begin{equation*}\label{eqn:label_for_eqn_20} \all\alpha\in\OR\ \Big[\varphi(\vec{p},\alpha)\iff M^{\Tt_U}_\beta\sats\varphi(\vec{p},\alpha)\Big].\end{equation*}
 It now follows that
 \begin{equation*}\label{eqn:label_for_eqn_21} \all\alpha\in\OR\ \Big[\varphi(\vec{p},\alpha)\iff M^{\Tt_U}_\kappa\sats\varphi(\vec{q},\alpha^*)\Big] \end{equation*}
where $\vec{q}=i^{\Tt_U}_{0\kappa}(\vec{p})$ (here $\alpha^*$ denotes $*_\kappa(\alpha)$).
 The rest now follows easily from
  Lemma~\ref{lem:*_def}. Note that
  we make use of the parameter $({^\kappa}\beta)\cap M^{\Tt_U}_{\kappa}$ in order to compute $*_\kappa\rest\beta$, via that lemma.
\end{proof}

\subsection{Above Infinitely Many Measurables}\label{sec:wo_above_infinite_meas}

In \cite[Corollary 7.4]{Sigma_1_def_at_higher_cardinals},
L\"ucke and M\"uller showed assuming ZFC + $\kappa$ is a limit a of measurable cardinals:
\begin{enumerate}
\itemref{(i)}
no well ordering of $\pow(\kappa)$
is $\Sigma_1(V_\kappa\cup\{\kappa,\kappa^+\})$, and
\itemref[item:LM_Cor_7.4(ii)]{(ii)} if $\cof(\kappa)>\om$ then no injection $f:\kappa^+\to\pow(\kappa)$ is $\Sigma_1(V_\kappa\cup\{\kappa\})$.\footnote{Again, allowing the parameter $\kappa^+$ here does no harm.}
\end{enumerate}

In \cite[Theorem 1.4]{Sigma_1_def_at_higher_cardinals},
they also show, under ZFC, if $\kappa$ is a limit of measurables, $\cof(\kappa)=\om$ and there is a $\Sigma_1(\{\kappa\})$ wellorder $W$ of some set $D\sub\pow(\kappa)$ of cardinality $>\kappa$, then there is a $\bfSigma^1_3$ wellorder of the reals.

This led them to pose
 \cite[Question 10.4]{Sigma_1_def_at_higher_cardinals},
 which asks whether the following theory is consistent:
ZFC + ``there is a cardinal $\kappa$
which is a limit of measurable cardinals
and a family $D\sub\pow(\kappa)$
of cardinality $>\kappa$
and a wellorder  $<^*$ of $D$
such that $<^*$ is $\Sigma_1(V_\kappa\cup\{\kappa\})$''.
And \cite[Question 10.3]{Sigma_1_def_at_higher_cardinals} asks the same question, except with the extra demand that $\kappa$ have cofinality $\omega$.
The answers (to both) are ``no'',
as we show in Theorem~\ref{tm:above_inf_many_meas} below.

 We also sharpen their result~\ref{item:LM_Cor_7.4(ii)}
 above, by removing the hypothesis that $\cof(\kappa)>\om$,
and also by allowing arbitrary defining parameters in $\her_\kappa\cup\OR$;
we will also establish an ``injective on a club'' version.

L\"ucke and M\"uller  also showed in
 \cite[Theorem 1.1]{Sigma_1_def_at_higher_cardinals},
 assuming ZFC + $\kappa$ is a limit of measurable cardinals, that
 for every $\Sigma_1(V_\kappa\cup\{\kappa\})$ set $D\sub\pow(\kappa)$
 of cardinality $>\kappa$,
 there is a perfect embedding $\iota:{^{\cof(\kappa)}}\kappa\to\pow(\kappa)$ with $\rg(\iota)\sub D$. We will also generalize this theorem,
also by allowing any parameters in $V_\kappa\cup\OR$ to be used in defining $D$.

\begin{tm}\label{tm:above_inf_many_meas}
 Assume ZFC and $\kappa\in\OR$ is a limit of measurable cardinals.
 Then:
 \begin{enumerate}
  \item\label{item:every_Sigma_1_has_perfect_subset} For every $\Sigma_1(V_\kappa\cup\OR)$
  set $D\sub\pow(\kappa)$ of cardinality $>\kappa$,
  there is a perfect embedding $\iota:{^{\cof(\kappa)}\kappa}\to\kappa$ with $\rg(\iota)\sub D$.
  \item\label{item:no_Sigma_1_long_wellorder} There is no $\Sigma_1(V_\kappa\cup\OR)$ wellorder $<^*$ of a set $D\sub\pow(\kappa)$ with $D$ of cardinality $>\kappa$.
\item\label{item:no_injection_on_club_limit_of_meas}
 There are no $f,C$
 such that $f\sub\kappa^+\cross\pow(\kappa)$, $f$ is $\Sigma_1(V_\kappa\cup\OR)$,
 $C$ is club in $\kappa^+$ and $f\rest C$ is an injection.
 \end{enumerate}
\end{tm}

\begin{rem}\label{rem:3.4}
 This theorem also yields a direct proof of the ``(i) implies (ii)'' part of \cite[Theorem 1.5]{Sigma_1_def_at_higher_cardinals};
 in fact we have direct implication in that direction, instead of just relative consistency.
\end{rem}

\begin{proof}[Proof of Theorem~\ref{tm:above_inf_many_meas}]

Part~\ref{item:no_injection_on_club_limit_of_meas}: By Theorem~\ref{tm:above_one_measurable}
 part~\ref{item:no_injection_kappa^+_to_p(kappa)_allowing_countable_cof} combined with Lemma~\ref{lem:ordinals_ev_stable_meas_limit}.

Part~\ref{item:no_Sigma_1_long_wellorder}:
 Let $\varphi$
be  $\Sigma_1$ and $\vec{p}\in(V_\kappa\cup\OR)^{<\om}$ and suppose  $D=\{X\in\pow(\kappa)\bigm|\varphi(\vec{p},X)\}$
has cardinality $>\kappa$.
Let $\eta=\cof(\kappa)\leq\kappa$.
Fix a strictly increasing sequence $\left<\kappa_\alpha\right>_{\alpha<\eta}$ of measurable cardinals cofinal in their supremum $\kappa$,
such that:
\begin{enumerate}
\item[(i)]
 $\eta<\kappa_0$ if $\eta<\kappa$,
\item[(ii)] $\vec{p}\cap V_\kappa\in V_{\kappa_0}$,
\itemref[item:stable_tree]{(iii)} $i^\Tt(\xi)=\xi$
for each ordinal $\xi\in\vec{p}$
whenever $\Tt$
is a linear iteration of measures on $V$
of successor length $\lambda+1\leq\kappa+1$,
 is based on $V_\kappa$, is
 $\kappa$-weakly normal, and is above $\kappa_0$.
\end{enumerate}
(Use Lemma~\ref{lem:ordinals_ev_stable_meas_limit} to see there is such a $\kappa_0$.)
 Fix a sequence $\left<U_\alpha\right>_{\alpha<\eta}$ of $\kappa_\alpha$-complete measures $U_\alpha$ on $\kappa_\alpha$.

Fix $X\in D$
such that
 $X\neq i^{\Tt_{U_\alpha}}_{0\kappa}(X\cap\kappa_\alpha)$
 for all $\alpha<\eta$. Such an $X$ exists
 because $D$ has cardinality ${>\kappa}$,
 but there are only $\kappa$-many
 sets $Y\sub\kappa$ such that $Y=i^{\Tt_{U_\alpha}}_{0\kappa}(Y\cap\kappa_\alpha)$ for some $\alpha<\kappa$.
By thinning out the sequence of measurables, we may assume that for all $\alpha<\eta$, we have \begin{equation}\label{eqn:U_alphas_not_cohere_X} i^{\Tt_{U_\alpha}}_{0\kappa_{\alpha+1}}(X\cap\kappa_{\alpha})\neq X\cap\kappa_{\alpha+1}.\end{equation}

Let $A\sub\{\kappa_\alpha\}_{\alpha<\eta}$ be cofinal in $\kappa$.
Write $A=\{\kappa^A_\alpha\}_{\alpha<\eta}$ with $\kappa^A_\alpha<\kappa^A_\beta$
for $\alpha<\beta$.
Then define the iteration $\Uu_A$
to be that iteration $\Uu$ using  measures from $\{U_\alpha\}_{\alpha<\eta}$ and its pointwise images, with $i^{\Uu}(\kappa_\alpha)=\kappa^A_\alpha$
for each $\alpha<\eta$,
and $\crit(E^\Uu_\alpha)<\crit(E^\Uu_\beta)$ for $\alpha<\beta<\eta$
(it is easy to see that this uniquely determines $\Uu$).
Then $\Uu_A$  is of the kind mentioned in clause~\ref{item:stable_tree} above, so $i^{\Uu_A}(\vec{p})=\vec{p}$.
Let $M_A=M^{\Uu_A}_\infty$, which is wellfounded;
let $i_A=i^{\Uu_A}:V\to M_A$. Since $\varphi$ is $\Sigma_1$,
it follows that $i_A(X)\in D$.

Now we may assume that $<^*$ is a $\Sigma_1(\{\vec{p}\})$ wellorder of $D$. Let $\psi$ be a $\Sigma_1$
formula such that for all $X,Y\in\pow(\kappa)$, we have $X<^*Y\iff\psi(\vec{p},X,Y)$.

Continuing with $A$ as above,
suppose that $A\cap\{\kappa_n\}_{n<\om}\psub\{\kappa_n\}_{n<\om}$
but $A\cap\{\kappa_n\}_{n<\om}$ is infinite. Letting $n$ be least such that $\kappa_n\notin A$,
so $\crit(i_A)=\kappa_n$
and $\kappa_{n+1}\leq i_A(\kappa_n)$,
note that $X\neq i_A(X)$,
and in fact $X\cap\kappa_{n+1}\neq i_A(X)\cap\kappa_{n+1}$.
A slight generalization of this also shows that if  $A'$ is likewise and $A\cap\{\kappa_n\}_{n<\om}\neq A'\cap\{\kappa_n\}_{n<\om}$
then $i_A(X)\neq i_{A'}(X)$.

Now we claim that if $A\cap\{\kappa_n\}_{n<\om}\psub\{\kappa_n\}_{n<\om}$ and $A\cap\{\kappa_n\}_{n<\om}$ is infinite then $X<^*i_A(X)$.
For otherwise $i_A(X)<^*X$,
so $\psi(\vec{p},Y,X)$
where $Y=i_A(X)$.
We can apply $i_A$ to this statement, which since $i_A(\vec{p})=\vec{p}$, gives
\begin{equation*}\label{eqn:label_for_eqn_22} M_A\sats\psi(\vec{p},Y',X') \end{equation*}
where $Y'=i_A(Y)=i_A(i_A(X))$ and $X'=i_A(X)$. So in fact $\psi(\vec{p},Y',X')$ holds,
so $i_A(i_A(X))<^*i_A(X)$.
But then we can also apply $i_A$ to the latter statement, and so on, giving a descending $\omega$-sequence through $<^*$, a contradiction.

It similarly follows that if $A,B$ are as above and $\{\kappa_\alpha\}_{\om\leq\alpha<\eta}\sub A\cap B$ and $B\cap\{\kappa_n\}_{n<\om}\psub A\cap\{\kappa_n\}_{n<\om}$,
then $i_A(X)<^*i_B(X)$.
(We have $B\cap\{\kappa_n\}_{n<\om}\in M_A$,  and working in $M_A$, we
can define $M_B$ and the factor map $i_{AB}:M_A\to M_B$, and get $i_{AB}(i_A(X))=i_B(X)$, from which we get $i_A(X)<^*i_B(X)$ much as before. Here is some more detail, assuming $\eta=\om$ to ignore irrelevant details. We have $B\in M_A$,
 and $B\sub i_A(\{\kappa_n\}_{n<\om})=\{\kappa^A_n\}_{n<\om}$,
and so we can define $(M_B)^{M_A}$ and $(i_B)^{M_A}$ in $M_A$;
that is, we apply the  preceding definitions in $M_A$, working from the parameter $i_A(\{\kappa_n\}_{n<\om})=\{\kappa^A_n\}_{n<\om}$, to define $(M_B)^{M_A}$ and $(i_B)^{M_A}$.
A simple instance of normalization\footnote{See for example \cite{iter_for_stacks}, \cite{ACPFMP}; but here things are much simpler.}  yields that $(M_B)^{M_A}=M_B$ and $i_B=(i_B)^{M_A}\com i_A$. Writing $i_{AB}=(i_B)^{M_A}$, we then have $i_{AB}:M_A\to M_B$ and $i_{AB}(i_A(X))=i_B(X)$.)

But now let $B=\{\kappa_{2n}\}_{n<\om}\cup\{\kappa_\alpha\}_{\om\leq\alpha<\eta}$,
and consider $A_n=B\cup\{\kappa_1,\ldots,\kappa_{2n-1}\}$.
Then $A_n\psub A_{n+1}$ for all $n$,
so $i_{A_{n+1}}(X)<^*i_{A_n}(X)$
for all $n$, giving a strictly descending $\om$-sequence through $<^*$, a contradiction.

Part~\ref{item:every_Sigma_1_has_perfect_subset} follows readily from the methods
 of \cite[Theorem 1.1]{Sigma_1_def_at_higher_cardinals},
 combined with the selection of $\{\kappa_\alpha\}_{\alpha<\eta}$ and the set $X$ as above.
For self-containment, here is the construction. For $f:\lambda\to\kappa$
where $0\leq\lambda\leq\eta$,
we will define a successor length iteration of measures
 $\Uu_f$, continuously in $f$, meaning that
 \begin{equation*}\label{eqn:label_for_eqn_23} g\ins f\implies \Uu_g\ins\Uu_f \end{equation*}
 and if $\lambda$ is a limit ordinal then
 $\Uu_f$ is just the successor length iteration of minimal length extending $\Uu_g$ for all $g\pins f$. We will have
 $i^{\Uu_f}(\vec{p})=\vec{p}$
and
$i^{\Uu_f}(\kappa)=\kappa$,
as a consequence of Lemma~\ref{lem:ordinals_ev_stable_meas_limit}.
We will then define $X_f=i^{\Uu_f}(X)\sub\kappa$,
and therefore have  $X_f\in D$. We will choose these iterations such that for distinct $f, g:\eta\to\kappa$,
we have $X_f\neq X_g$,
and in fact letting $\lambda$ be least such that $f(\lambda)\neq g(\lambda)$, say $f(\lambda)<g(\lambda)$, we will have
$i^{\Uu_f}(\kappa_\lambda)<i^{\Uu_g}(\kappa_\lambda)$ and
\begin{equation}\label{eqn:X_f_disagree_X_g} X_f\cap i^{\Uu_g}(\kappa_\lambda)\neq X_g\cap i^{\Uu_g}(\kappa_\lambda).\end{equation}
The iteration $\Uu_f$ will just use the measures $U_\lambda$ and their images, for $\lambda\in\zeta=\dom(f)$,
first iterating $U_0$ and its images for some length, then the image of $U_1$ and its further images for some length, and so on.
The length through which we iterate the image of $U_\lambda$
will be determined by $f(\lambda)$
and the goal of  arranging line~(\ref{eqn:X_f_disagree_X_g}).

We start with $\Uu_\emptyset$ being the trivial iteration.
Say $\dom(f)=\lambda<\eta$
and we have defined $\Uu_{f}$.
We now define $\Uu_{f\conc\left<\beta\right>}$ for each $\beta<\kappa$, by iterating with images of $i^{\Uu_f}(U_\lambda)$.
That is, for each $\beta<\kappa$,
$\Uu_{f\conc\left<\beta\right>}$
will have form $\Uu_{f}\conc\Tt_\beta$,
where $\Tt_\beta$ is the linear iteration of $M^{\Uu_f}_\infty$
of length a certain ordinal $\zeta_\beta+1$,
with $E^{\Tt_\beta}_\gamma=i^{\Tt_\beta}_{0\gamma}(i^{\Uu_f}(U_\lambda))$ for each $\gamma<\zeta_\beta$.
It just remains to specify $\zeta_\beta$,
which we do by recursion on $\beta$.
We set $\zeta_0=0$,
and for limit $\gamma$, $\zeta_\gamma=\sup_{\delta<\gamma}\zeta_\delta$. Finally, given $\zeta_\beta$ where $\beta<\kappa$,
let $\zeta_{\beta+1}=i^{\Tt_\beta}_{0\zeta_\beta}(i^{\Uu_f}_{0\infty}(\kappa_{\lambda+1}))$.

This completes the construction of $\Uu_f$.
As stated earlier, we define $X_f=i^{\Uu_f}(X)$. We claim that $\iota:f\mapsto X_f$ is a perfect embedding
 $\iota:{^\eta}\kappa\to\pow(\kappa)$
 with $\range(\iota)\sub D$.
For $i^{\Uu_f}(\kappa)=\kappa$ for each $f:\eta\to\kappa$,
and clause~\ref{item:stable_tree} above applies to $\Uu_f$,
so $i^{\Uu_f}(\vec{p})=\vec{p}$,
and as before, using Lemma~\ref{lem:ordinals_ev_stable_meas_limit}, it follows that $X_f\in D$.
Now let us see that line~(\ref{eqn:X_f_disagree_X_g})
holds, where $f,g,\lambda$ are as there.
Let $h=f\rest\lambda=g\rest\lambda$.
So $\Uu_f$, $\Uu_g$ each extend $\Uu_h$,
with $\Uu_f$ extending $\Uu_{h\conc\left<f(\lambda)\right>}$
and $\Uu_g$ extending $\Uu_{h\conc\left<g(\lambda)\right>}$.
But note that
\begin{equation*}\label{eqn:label_for_eqn_24} i^{\Tt_{f(\lambda)}}(i^{\Uu_h}(\kappa_{\lambda+1}))=i^{\Tt_{f(\lambda)+1}}(\kappa_\lambda)\leq i^{\Tt_{g(\lambda)}}(\kappa_\lambda), \end{equation*}
but by line~(\ref{eqn:U_alphas_not_cohere_X}),
\begin{equation}\label{eqn:stage_distinction} i^{\Tt_{f(\lambda)}}(i^{\Uu_h}(\kappa_{\lambda+1}))\cap i^{\Tt_{f(\lambda)}}(i^{\Uu_h}(X))\neq i^{\Tt_{f(\lambda)+1}}(i^{\Uu_h}(\kappa_\lambda))\cap i^{\Tt_{g(\lambda)}}(i^{\Uu_h}(X)). \end{equation}
Moreover, the tails of the iterations  have critical points at least $i^{\Tt_{f(\lambda)}}(i^{\Uu_h}(\kappa_{\lambda+1}))$,
so the distinction  in line~(\ref{eqn:stage_distinction})
leads to  line~(\ref{eqn:X_f_disagree_X_g}).
\end{proof}

\subsection{$I_2$-embeddings}
\label{sec:above_rank-to-rank}

In this section we will prove two facts
(in Theorem~\ref{tm:3.11}),
analogous to parts of Theorem~\ref{tm:above_inf_many_meas},
but with iterations of $I_2$-embeddings replacing
iterations of many measurables. The arguments
here will be analogous to some of those
in the earlier sections.

Recall that $I_2$ holds at $\lambda$
iff there is an elementary $j:V\to M$
with $M$ transitive, $\crit(j)=\kappa<\lambda$, $j(\lambda)=\lambda$
and $V_\lambda\sub M$. Since we are working in ZFC, this implies that $\lambda$ is a limit ordinal, and in fact that $\lambda=\sup_{n<\om}\kappa_n$, where $\left<\kappa_n\right>_{n<\om}$
is the critical sequence of $j$; that is, $\kappa_0=\kappa$ and $\kappa_{n+1}=j(\kappa_n)$.
For a limit ordinal $\lambda$,
let's say that $E$ is an \emph{$I_2(\lambda)$-extender}
iff $E$ is a rank-to-rank extender over $V_\lambda$ and $\Ult(V,E)$ is wellfounded. Then $I_2$ holds at $\lambda$ iff there is an $I_2(\lambda)$-extender.
Let's say that $(\lambda,E)$ is an \emph{$I_2$-pair}
iff $\lambda$ is a limit ordinal and $E$ is an $I_2(\lambda)$-extender. Note that in the following theorem,
$\Ult(N,G)$ denotes the ultrapower formed using
only functions in $N$, and $i^N_G:N\to\Ult(N,G)$ is the associated ultrapower embedding.

We will need the following analogue of Lemma \ref{lem:embedding_agreement}:

\begin{tm}\label{tm:natural_isormorphism_I2}
Assume ZFC.
Let $(\lambda,E)$ be an $I_2$-pair
and $M=\Ult(V,E)$.
Let $F$ be a rank-to-rank extender over $V_\lambda$.
Then:
\begin{enumerate}
\item $\Ult(V,F)$ is wellfounded iff $\Ult(M,F)$
is wellfounded.
\item $i^V_F(M)=\Ult(M,F)$
and $i^V_F\rest M=i^M_F$.
\item\label{item:function_correspondence} For every $\gamma<\lambda$, every $f:[\gamma]^{<\om}\to M$
and every $a\in [i^{V_\lambda}_F(\gamma)]^{<\om}$,
there is  $X\in F_a$
such that $f\rest X\in M$.
\end{enumerate}
\end{tm}

\begin{proof}
It is clearly enough to prove part \ref{item:function_correspondence}.
So fix $(\gamma,f,a)$ as there.
Let $\left<\kappa_n\right>_{n<\om}$
be the critical sequence
of $i^V_E$.
For $z\in M$, let $n_z$ be the least $n<\om$
such that there is $g:[\kappa_n]^{<\om}\to V$
and $b\in[i^V_E(\kappa_n)]^{<\om}$
such that $z=[b,g]^V_E=i^V_E(g)(b)$.
Since $F_a$ is countably complete,
we can fix $X\in F_a$ and $n<\om$
such that $n_{f(x)}=n$ for all $x\in X$.
We claim this $X$ works; that is, we have $f\rest X\in M$.
To see this, for $x\in X$, let $g_x:[\kappa_n]^{<\om}\to V$
and $b_x\in[i^V_E(\kappa_n)]^{<\om}$
be such that $x=[b_x,g_x]^{V}_E=i^V_E(g_x)(b_x)$.
Let $G:X\to V$
be $G(x)=g_x$
and $B:X\to V$
be $B(x)=b_x$.
Then $i^V_E(G)\in M$, $X\in V_\lambda\sub M$, $B\in V_\lambda\sub M$
and
 $k=i^V_E\rest[\kappa_n]^{<\om}\in V_\lambda\sub M$.
 But note that for $x\in X$, we have
 \begin{equation*}\label{eqn:added_label_for_eqn_1}  f(x) = i^V_E(g_x)(b_x)=i^V_E(G)(k(x))(B(x)), \end{equation*}
 and since $i^V_E(G),k,B,X\in M$, therefore $f\rest X\in M$.
\end{proof}

\begin{lem}\label{lem:linear_length_om_rank-to-rank_iterations_eventually_fix_ordinal}
Assume ZFC and let $\lambda$ be a limit ordinal. Then:
\begin{enumerate}
\item\label{item:linear_length_om_I_2_tree_ev_fix_ordinal}Let $\Tt$ be a length $\om$ linear iteration of $V$,
such that for each $n<\om$, $E^\Tt_n$ is a rank-to-rank
extender over $V_\lambda^{M^\Tt_n}$ \tu{(}thus, $V_\lambda^{M^\Tt_n}=V_\lambda$\tu{)}. Let $k_n:V\to\Ult(V,E^\Tt_n)$ be the ultrapower map \tu{(}note that $\dom(k_n)=V$, whereas $\dom(i^{\Tt}_{n,n+1})=M^\Tt_n$\tu{)}. Then for each ordinal $\xi$, there is $n<\om$ such that $k_m(\xi)=\xi$ for all $m\in[n,\om)$.
\item\label{item:linear_length_om_I_2_tree_ev_fix_ordinal_case_iterate_j} Let $E$ be an $I_2(\lambda)$-extender,
$j_0:V\to M=\Ult(V,E)$ the ultrapower map,
and  $j_{n+1}=j_n(j_n)=j_0(j_n)$,
 $E_{n+1}=j_n(E_n)=j_0(E_n)$, and $j_n^+=k_n:V\to \Ult(V,E_n)$
 be the ultrapower map. Then for each $\xi\in\OR$,
 there is $n<\om$ such that $j_m^+(\xi)=\xi$ for all $m\in[n,\om)$.
 \end{enumerate}
\end{lem}
\begin{proof}
Part \ref{item:linear_length_om_I_2_tree_ev_fix_ordinal}:
Because $\Tt$ is linear, the direct limit model $M^\Tt_\om$ is wellfounded.
Therefore there is $n<\om$ such that for all $m\in[n,\om)$,
we have $i^\Tt_{m,m+1}(\xi)=\xi$.
But by Theorem~\ref{tm:natural_isormorphism_I2}, $k_m\rest\OR=i^\Tt_{m,m+1}\rest\OR$, which suffices.

Part \ref{item:linear_length_om_I_2_tree_ev_fix_ordinal_case_iterate_j} is just a special case of part \ref{item:linear_length_om_I_2_tree_ev_fix_ordinal}.
\end{proof}

\begin{dfn}\label{dfn:appropriate}
Let $f:\om\to\om$. We say that $f$ is \emph{appropriate}
if $f(n)>0$ for infinitely many $n$.
\end{dfn}
\begin{dfn}\label{dfn:Tt_Ef}
Assume ZFC and let $(\lambda,E)$ be an $I_2$-pair.
Let $E_0=E$ and $E_{n+1}=i^V_{E_n}(E_n)=i^V_E(E_n)$.
Let $f:\om\to\om$ be  appropriate.
The \emph{$(E,f)$-iteration}, which we denote $\Tt_{Ef}$, is the length $\om+1$ linear iteration $\Tt$ on $V$ where
we use $E_0$ and its images $f(0)$ times,
then images of $E_1$ $f(1)$ times, etc.
That is, for all $k<f(0)$, $E^{\Tt}_k=i^\Tt_{0k}(E_0)$,
and more generally, for all $n<\om$ and all $k<f(n)$,
letting $m=f(0)+\ldots+f(n-1)$, we have
$E^\Tt_{m+k}=i^\Tt_{0,m+k}(E_n)$.
\end{dfn}

Note that if $\Tt=\Tt_{Ef}$
then $i^\Tt_{0\om}(\kappa_n)=\kappa_\ell$ where  $\ell=n+f(0)+\ldots+f(n)$, and therefore $i^\Tt_{0\om}(\lambda)=\lambda$
and $i^\Tt_{0\om}``\lambda\sub\lambda$.

We next establish an analogue of Lemma \ref{lem:ordinals_ev_stable_meas_limit}:

\begin{lem}\label{lem:I_2_stable_xi_with_infinite_iters}
Assume ZFC and let $(\lambda,E)$ be an $I_2$-pair.
Then for each $\xi\in\OR$
there is $n<\om$ such that for all appropriate
$f:\om\to\om$, if $f(k)=0$ for all $k< n$ then $i^{\Tt_{Ef}}_{0\om}(\xi)=\xi$.
\end{lem}
\begin{proof}
This is like the proof of Lemma \ref{lem:ordinals_ev_stable_meas_limit},
with Lemma \ref{lem:I_2_stable_xi_with_infinite_iters} playing the role
of the ``simple version''
of Lemma \ref{lem:ordinals_ev_stable_meas_limit}, which was
established first and used as a lemma in the proof of the full version of Lemma \ref{lem:ordinals_ev_stable_meas_limit}.

That is, fix $\xi$, and let $n$ witness Lemma~\ref{lem:linear_length_om_rank-to-rank_iterations_eventually_fix_ordinal} part~\ref{item:linear_length_om_I_2_tree_ev_fix_ordinal_case_iterate_j} with respect to $E,\xi$. We claim that $n$ also witnesses the lemma under consideration for $\xi$. For let $f:\om\to\om$
be appropriate with $f(k)=0$ for all $k<n$. Let $\Tt=\Tt_{Ef}$. We want to see that $i^{\Tt}_{0\om}(\xi)=\xi$. For this, we will define an order-preserving
map $\pi:i^\Tt_{0\om}(\xi)\to\xi$ (actually, $\pi$ will be the identity).

Now we have $i^{\Tt}_{0k}(\xi)=\xi$
for all $k<\om$, by choice of $n$ and a straightforward induction through these finite stages. So let $k<\om$
and $\eta<\xi$; we want to define $\pi(i^\Tt_{k\om}(\eta))$. Let $n'<\om$ witness Lemma \ref{lem:linear_length_om_rank-to-rank_iterations_eventually_fix_ordinal} part \ref{item:linear_length_om_I_2_tree_ev_fix_ordinal_case_iterate_j} applied in $M^\Tt_k$,
to $E'=E^\Tt_k$ and $\eta$. Let $E'_0=E'$ and $E'_{m+1}=E'_m(E'_m)$.
Then note that for each $\ell\in[k,\om)$, $E^\Tt_{\ell}=i^\Tt_{k\ell}(E'_m)$ for some $m=m_\ell$,
and moreover, there is $\ell$ such that for all $\ell'\in[\ell,\om)$, we have $m_{\ell'}\geq n'$.
Fixing this $\ell$, we define $\pi(i^\Tt_{k\om}(\eta))=i^\Tt_{k\ell}(\eta)$.
The fact that this is order-preserving follows easily
from the fact that
 $i^\Tt_{\ell\ell'}(i^\Tt_{k\ell}(\eta))=i^\Tt_{k\ell}(\eta)$ for all $\ell'\in[\ell,\om)$.
 (Moreover, $\pi$ is in fact the identity,
 because for every $\eta<\xi$,
 there is $\ell$ such that $i^\Tt_{\ell\ell'}(\eta)=\eta$
 for all $\ell'\in[\ell,\om)$.)
\end{proof}

\begin{dfn}\label{dfn:simply_perfect}
For an uncountable cardinal $\lambda$, letting $\theta=\cof(\lambda)$,
a \emph{simply perfect function}
(for ${^\lambda}2$)
is a function $\iota:{^{<\theta}}2\to{^{<\lambda}}2$
such that for all $s,t\in{^{<\theta}}2$,
if $s\psub t$ then $\iota(s)\psub \iota(t)$,
$\iota(s\conc\left<0\right>)\not\sub\iota(s\conc\left<1\right>)\not\sub\iota(s\conc\left<0\right>)$, and $\dom\big(\bigcup_{\alpha<\theta}\iota(x\rest\alpha)\big)=\lambda$
for all $x\in{^\theta}2$.
A \emph{simply perfect subset}
of ${^{\lambda}}2$
is the range of a simply perfect function.
\end{dfn}

We can now proceed to the main result of this section, Theorem \ref{tm:3.11}.
Its part \ref{item:I_2_PSP} should be compared with
\cite[Theorems 1.4, 3.1]{descriptive_properties_I2-embeddings}, which are related.
Note that in terms of the parameters allowed, Theorem \ref{tm:3.11}(\ref{item:I_2_PSP}) is orthogonal to those results, in that \cite{descriptive_properties_I2-embeddings}
handles parameters in $V_\lambda\cup\{V_\lambda\}\cup P$ for a certain collection $P\sub V_{\lambda+1}$, whereas we handle parameters in $V_\lambda\cup\{V_\lambda\}\cup\OR$.\footnote{In the first arXiv preprint of the present paper, \S\ref{tm:above_inf_many_meas}
was very brief, and the argument contained a significant error and gap. The author noticed this in February 2026,
and added \ref{tm:natural_isormorphism_I2},
\ref{lem:linear_length_om_rank-to-rank_iterations_eventually_fix_ordinal},
\ref{dfn:Tt_Ef}, \ref{lem:I_2_stable_xi_with_infinite_iters},
part \ref{item:I_2_Sigma_1_singletons} of \ref{tm:3.11}, and corrected the proof of \ref{tm:3.11}
(and added the proof of  part \ref{item:I_2_Sigma_1_singletons}).}

\begin{tm}\label{tm:3.11}
 Assume ZFC + $I_2(\lambda)$.
 Let $D\sub\pow(\lambda)$
 have cardinality $>\lambda$.
 Then:
 \begin{enumerate}
  \item\label{item:I_2_Sigma_1_singletons}
 Let $A\sub V_\lambda$
 be such that $\{A\}$ is $\Sigma_1(V_\lambda\cup\{V_\lambda\}\cup\OR)$. Then for every $I_2(\lambda)$-extender $E$,
 there is $n<\om$ such that $i^V_{E_n}(A)=A$,
 where $E_n$ is as in Lemma \ref{lem:linear_length_om_rank-to-rank_iterations_eventually_fix_ordinal} part \ref{item:linear_length_om_I_2_tree_ev_fix_ordinal_case_iterate_j}. Therefore there are only $\lambda$-many such sets $A$.
 \item\label{item:I_2_no_big_wellorder} There is no
 $\Sigma_1(V_\lambda\cup\{V_\lambda\}\cup\OR)$
 wellorder of $D$.
 \item\label{item:I_2_PSP} If $D$
 is  $\Sigma_1(V_\lambda\cup\{V_\lambda\}\cup\OR)$
 then $D$ has a simply perfect subset.
 \end{enumerate}
\end{tm}
The crucial difference
between this and the results on measurable cardinals is that we also have $V_\lambda$ itself available as a parameter.
\begin{proof}
Part \ref{item:I_2_Sigma_1_singletons}:
Suppose $A$ is the unique set such that $\varphi(\vec{p},A)$,
where $\varphi$ is $\Sigma_1$ and $\vec{p}\in(V_\lambda\cup\{V_\lambda\}\cup\OR)^{<\om}$.
Let $E$ be an $I_2(\lambda)$-extender.
Let $n^*<\om$ witness Lemma~\ref{lem:linear_length_om_rank-to-rank_iterations_eventually_fix_ordinal} part~\ref{item:linear_length_om_I_2_tree_ev_fix_ordinal_case_iterate_j} for $E$, simultaneously for all ordinals $\xi\in\vec{p}$,
and with $\vec{p}\cap V_\lambda\sub V_{\crit(E_{n^*})}$.
Then $n^*$ is as desired, since $i_{E_{n^*}}^V(\vec{p})=\vec{p}$,
so $\Ult(V,E_{n^*})\sats\varphi(\vec{p},i^V_{E_{n^*}}(A))$,
so by $\Sigma_1$-upward absoluteness, $V\sats\varphi(\vec{p},i^V_{E_{n^*}}(A))$, and
so $i^V_{E_{n^*}}(A)=A$.

The fact that there are only $\lambda$-many such
sets $A$ follows, because $A$ is determined by $A\cap V_{\kappa_{n^*}}$ and $E$,
since $A\cap V_{\kappa_{n^*+1}}=i^V_{E_{n^*}}(A\cap V_{\kappa_{n^*}})$, and  $A\cap V_{\kappa_{n^*+2}}=i^V_{E_{n^*}}(A\cap V_{\kappa_{n^*+1}})$, etc. (Note we can fix one such $E$ here,
and use it to determine all such sets $A$ from some  proper segment of themselves.)

Part \ref{item:I_2_no_big_wellorder}:
Let $\vec{p}\in V_\lambda\cup\{V_\lambda\}\cup\OR$
and $\varphi$ be $\Sigma_1$,
and suppose that $\varphi(\vec{p},\cdot,\cdot)$ defines a wellorder $<^*$ of a set $D\sub\pow(\lambda)$ of cardinality $>\lambda$.
 Let $E$ be a rank-to-rank extender over $V_\lambda$
 with $M=\Ult(V,E)$ wellfounded.
 Let $n^*<\om$ be such that for all $x\in\vec{p}$,
 we have $i^{\Tt_{Ef}}_{0\om}(x)=x$
 for all appropriate $f$ with $f(k)=0$ for all $k<n^*$;
we get such an $n^*$ from Lemma \ref{lem:I_2_stable_xi_with_infinite_iters}. By replacing $E$ with its $n^*$th iterate $E_{n^*}$, we may assume $n^*=0$.

 Since $D$ has cardinality $>\lambda$,
 we can fix $X\in D$ such that $i^V_{E_n}(X)\neq X$ for all $n<\om$ (see the proof of part~\ref{item:I_2_Sigma_1_singletons}).
  Thus, for each $n<\om$ there is a unique $\ell\in(n,\om)$ such that
  \begin{equation*}\label{eqn:added_label_for_eqn_2} i^V_{E_n}(X\cap\kappa_{\ell-1})\neq X\cap\kappa_{\ell}\text{ and
   if }\ell-1>n\text{ then }i^V_{E_n}(X\cap\kappa_{\ell-2})=X\cap\kappa_{\ell-1}.\end{equation*}
   Let $\ell_n$ be this $\ell$. Thus (and equivalently), letting $\Tt_{E_n}$ be the linear iteration of $V$ using only $E_n$ and its images,
   \begin{equation*}\label{eqn:added_label_for_eqn_3}  i^{\Tt_{E_n}}_{0,((\ell_n-1)-n)}(X\cap\kappa_n)=X\cap\kappa_{\ell_n-1}, \end{equation*}
   \begin{equation*}\label{eqn:added_label_for_eqn_4} i^{\Tt_{E_n}}_{0,(\ell_n-n)}(X\cap\kappa_n)\neq X\cap\kappa_{\ell_n},\end{equation*}
   and note that $i^{\Tt_{E_n}}_{0,\ell_n-n}(\kappa_n)=\kappa_{\ell_n}$ and for all $i\in[n,\ell_n)$,
   we have $\ell_i=\ell_n$.

   Define the sequence $\left<m_n\right>_{n<\om}$
   recursively by $m_0=0$ and $m_{n+1}=\ell_{m_n}$.
 Now for $n<\om$ let $f_n:\om\to\om$
 be the function where $f_n(m_k)=m_{k+1}-m_k$
 for all $k\geq n$, and $f_n(i)=0$ for all other inputs $i<\om$. In particular, $f_n(m_\ell)=0$ for $\ell<n$.
 Note that $f_n$ is appropriate,
  $\crit(i^{\Tt_{Ef_n}}_{0\om})=\kappa_{m_n}$,
and for $\ell\in[n,\om)$ and $i<f_n(m_\ell)=m_{\ell+1}-m_\ell$,
\begin{equation*}\label{eqn:added_label_for_eqn_5} \crit(i^{\Tt_{Ef_n}}_{f_n(m_n)+\ldots+f_n(m_{\ell-1})+i,\om})=i^{\Tt_{Ef_n}}_{0,f_n(m_n)+\ldots+f_n(m_{\ell-1})+i}(\kappa_{m_\ell}).\end{equation*}

For $n<\om$, let $X_n=i^{\Tt_{Ef_n}}_{0\om}(X)$.
Then $X_n\in D$,
since  $i^{\Tt_{Ef_n}}_{0\om}(\vec{p})=\vec{p}$  (recalling that we could assume $n^*=0$), and
by elementarity and $\Sigma_1$-upward absoluteness.
Let $\Uu_n=\Tt_{Ef_n}\rest((m_{n+1}-m_n)+1)$
(this is just the portion of $\Tt_{Ef_n}$ using $E_{m_n}$
and its images), and $i^{\Uu_n}$ its final iteration map;
note that $\Uu_n$ is a tree on $V$. Then
note that
\begin{equation*}\label{eqn:added_label_for_eqn_6}  \Tt_{Ef_{n}}=\Uu_n\conc i^{\Uu_n}(\Tt_{E{f_{n+1}}})=
\Uu_n\conc i^{\Uu_n}\Tt_{Ef_{n+1}};
\end{equation*}
that is,
$\Tt_{Ef_n}$ consists of $\Uu_n$ followed by
$i^{\Uu_n}(\Tt_{Ef_{n+1}})$, which is just the same thing as
$\Uu_n$ followed by the
copy $i^{\Uu_n}\Tt_{Ef_{n+1}}$  of $\Tt_{Ef_{n+1}}$ under $i^{\Uu_n}$.
So
\begin{equation}\label{eqn:X_n=i^Uu_n(X_n+1)}X_n=i^{\Uu_n}(X_{n+1}).
\end{equation}
Moreover, note that $X_n\neq X_{n+1}$,
because
\begin{equation*}\label{eqn:added_label_for_eqn_7} X_{n+1}\cap\kappa_{m_{n+1}}=X\cap\kappa_{m_{n+1}}\neq X_n\cap\kappa_{m_{n+1}}.\end{equation*}
Also $i^{\Uu_n}(\vec{p})=\vec{p}$,
since $i^{\Tt_{Ef_n}}(\vec{p})=\vec{p}$.
So we are now in the position to reach a contradiction
just like at the end of the proof of Theorem~\ref{tm:above_inf_many_meas}
part~\ref{item:no_Sigma_1_long_wellorder}.
That is, since $X_n,X_{n+1}\in D$ but  $X_n\neq X_{n+1}$,
either $X_n<^*X_{n+1}$ or $X_{n+1}<^*X_n$.
Suppose $X_n<^*X_{n+1}$. Then $V\sats\varphi(\vec{p},X_n,X_{n+1})$.
But since
$i^{\Uu_n}(X_{n+1})=X_n$, $i^{\Uu_n}(\vec{p})=\vec{p}$ and by line (\ref{eqn:X_n=i^Uu_n(X_n+1)}),
therefore $M^{\Uu_n}_\infty\sats\varphi(\vec{p},i^{\Uu_n}(X_n),X_n)$, so (by $\Sigma_1$ upward absoluteness) $V\sats\varphi(\vec{p},i^{\Uu_n}(X_n),X_n)$,
so
\begin{equation*}\label{eqn:added_label_for_eqn_8}  i^{\Uu_n}(X_n)<^*X_n<^*X_{n+1},\end{equation*}
and then applying $i^{\Uu_n}$ again, it similarly follows that
\begin{equation*}\label{eqn:added_label_for_eqn_9}  i^{\Uu_n}(i^{\Uu_n}(X_n))<^*i^{\Uu_n}(X_n)<^*X_n<^*X_{n+1},\end{equation*}
etc, producing an infinite descending sequence,
a contradiction.
So we must have $X_{n+1}<^*X_n$ for all $n$,
again giving us an infinite descending sequence
\[ \cdots<^* X_2<^*X_1<^*X_0.\]

Part \ref{item:I_2_PSP}:
With the tools developed above and methods similar
to those used in the proof of Theorem~\ref{tm:above_inf_many_meas} part~\ref{item:every_Sigma_1_has_perfect_subset},
it is easy to
build a function $f:{^{<\om}}2\to{^{<\lambda}}2$ of the right form. Given $x:\om\to 2$,
we will get $X_x=\bigcup_{n<\om}f(x\rest n)=i^{\Tt_{Eg_x}}(X)$
for a certain appropriate function $g_x$,
where $X$ is a set as before.
The function $g_x$ is produced continuously
in $x$, and if $x,y:\om\to\om$ with $x\neq y$,
then the distinction between $X_x$ and $X_y$
is produced by a distinction between $g_x$ and $g_y$, much like how we produced the distinction between $X_n$ and $X_{n+1}$ in the previous proof. We leave the details to the reader.
\end{proof}

\section{Descriptive Properties of $L[U]$}\label{sec:gdst_in_L[U]}

Throughout this section, we assume ZFC, $\mu$ is a measurable cardinal and $V=L[U]$ where $U$ is a normal measure on $\mu$. We make some observations on  descriptive set theory at uncountable cardinals $\kappa$ in this context, using the basic theory of $L[U]$ (see \cite{kunen}, \cite{jech}), by which $\mu$ is the unique measurable and $U$ the unique normal measure on $\mu$.
Note that since GCH holds in $L[U]$, a cardinal $\kappa$ is $\mu$-steady
iff $\kappa>\mu^+$, $\cof(\kappa)\neq\mu$ and $\kappa$ is not of the form $\nu^+$ where $\nu$ is a cardinal with $\cof(\nu)=\mu$.
Say a filter $D$ is \emph{$\kappa$-good}
iff $D$ is a filter over some $\gamma<\kappa$ and
$D\in L[D]\sats$ ``$D$ is a normal measure on $\gamma$''. We officially work with $L[U]$ in the modern, fine structural hierarchy, using notation as described in \S\ref{subsec:notation}.
So the extender sequence $\es^{L[U]}$ of $L[U]$
contains not only a total extender $E_U$ with derived normal measure $U$, but also many partial extenders determining  partial measures, on many ordinals $\leq\mu$,
and further, letting $\gamma=\lh(E_U)$, then $\es^U$ also has (extenders determining) partial measures on cofinally many ordinals $<\gamma$ (for example, the extender corresponding to $(\es^U\rest\mu)^\#$, which has critical point $>\mu$).

Part~\ref{item:injective_func_exists} of the following theorem
should be contrasted with Theorem~\ref{tm:above_one_measurable} part~\ref{item:no_injection_kappa^+_to_p(kappa)_allowing_countable_cof},
and part~\ref{item:no_perf_subset}
 already follows from \cite[Theorem 1.2]{Sigma_1_def_at_higher_cardinals} in the case that $\kappa$ is singular (in $L[U]$). But we will give a self-contained proof.

\begin{tm}\label{tm:L[U]_gdst}
 Assume $V=L[U]$ where $U$ is a normal measure and let $\mu$ be measurable. Then
 for every uncountable cardinal $\kappa$, we have:
 \begin{enumerate}
 \item\label{item:long_wo} There is a set $X\sub\pow(\kappa)$
 of cardinality $\kappa^+$
 and a wellorder $<^*$ of $X$
 such that $X,{<^*}$ are both $\Sigma_1(\{\kappa\})$.
 \item \label{item:injective_function_in_non-steady_case}
 By \cite{meas_cards_good_Sigma_1_wo}, if either $\kappa\leq\mu^+$ or $\kappa$ is non-$\mu$-steady
 then $\es^{L[U]}\rest\kappa^{+}$ and  $<_{L[U]}\rest\her_{\kappa^+}$ are $\Sigma_1(\{\kappa\})$,
 so the set of master codes of levels of $L[U]$
 \tu{(}in the fine hierarchy given by $\es^{L[U]}$\tu{)} projecting to $\kappa$
 is $\Sigma_1(\{\kappa\})$.
\item\label{item:function_f_for_U} If $\kappa\leq\mu^+$ or $\kappa$ is non-$\mu$-steady then there is a $\Sigma_1(\{\kappa\})$ function $f:\kappa^+\to\pow(\kappa)$ such that
  for each $\xi<\kappa^+$,
  $f(\xi)$ is a pair $(T,A)$ such that
  for some $\eta\in(\xi,\kappa^{+})$,
we have $A\in (L[U]|\eta)\sats$ ``$\kappa$ is the largest cardinal
  and $A$ is the $<_{L[U]}$-least wellorder of $\kappa$ in ordertype $\xi$'',
  \begin{equation}\label{eqn:L_eta[U]_is_hull} L[U]|\eta=\Hull^{L[U]|\eta}(\kappa),\end{equation}
  $T$  is the theory
  \begin{equation}\label{eqn:theory_T} T=\Th^{L[U]|\eta}(\kappa),\end{equation}
  and if $\kappa$ is singular then $L[U]|\eta\sats$ ``$\kappa$ is singular''.\footnote{Here when defining the hull $\Hull^{L[U]|\eta}(\kappa)$
  and the theory $\Th^{L[U]|\eta}(\kappa)$,
  the language is that of premice; in particular, there are symbols for $\es^{L[U]|\eta}$ and the active extender $F^{L[U]|\eta}$.  We encode finite tuples of ordinals via some standard  G\"odel coding, so that we can simply talk about tuples $(\varphi,\vec{x})$ where $\varphi$ is a formula and $\vec{x}$ a finite tuple of parameters in $\kappa$. Likewise, $(A,T)$ is encoded simply as a single subset of $\kappa$.}
  \item\label{item:injective_func_exists} If $\kappa>\mu^+$ and $\kappa$ is $\mu$-steady then there is a  set $d\sub\kappa^+$ which is stationary-co-stationary in $\kappa^+$,
  and a function $f:d\to\pow(\kappa)$,
  such that $d,f$ are both $\Sigma_1(\{\kappa\})$,
  and for each $\xi\in d$,
  $f(\xi)$ is a pair $(T,A)$
  such that for some $\kappa$-good filter $D$ and some ordinal $\eta\in(\kappa,\kappa^+)$,
  we have $A\in (L[D]|\eta)\sats$ ``$\kappa$ is the largest cardinal
  and $A$ is the $<_{L[D]}$-least wellorder of $\kappa$ in ordertype $\xi$'',
  \begin{equation}\label{eqn:L_eta[D]_is_hull} L[D]|\eta=\Hull^{L[D]|\eta}(\kappa),\end{equation}
  $T$  is the theory
  \begin{equation}\label{eqn:theory_T_D} T=\Th^{L[D]|\eta}(\kappa),\end{equation}
  and if $\kappa$ is singular then $L[D]|\eta\sats$ ``$\kappa$ is singular''.
  \item\label{item:no_perf_subset} There is a $\Sigma_1(\{\kappa\})$
  set $X\sub\pow(\kappa)$
  of cardinality $\kappa^+$,
  such that $X$ has no simply perfect subset
  \tu{(}Definition \ref{dfn:simply_perfect}\tu{)}.
  \end{enumerate}
\end{tm}

\begin{proof}
Part~\ref{item:long_wo} will be an immediate corollary of parts\nolinebreak[3] \ref{item:injective_function_in_non-steady_case}~and~\ref{item:injective_func_exists}, where in the case of part~\ref{item:injective_func_exists}, we
take $X=\range(f)$ and for $A,B\in X$, set $A<^*B$ iff $f^{-1}(A)<f^{-1}(B)$.

Part~\ref{item:injective_function_in_non-steady_case}: See  \cite{meas_cards_good_Sigma_1_wo}.

Part~\ref{item:function_f_for_U}: This is a routine consequence of part~\ref{item:injective_function_in_non-steady_case}.

 Part~\ref{item:injective_func_exists}:
 Define $f$ as the set of pairs $(\xi,(T,A))$
 such that
 $\kappa<\xi<\kappa^+$
 and there is a $\kappa$-good filter $D$ and an ordinal $\eta\in(\xi,\kappa^+)$ such that $A\in (L[D]|\eta)\sats$ ``$\card(\xi)=\kappa$'', $A$
  is the $<_{L[D]}$-least
  wellorder of $\kappa$ in ordertype $\xi$,
  $\xi$ is $D$-$1$-stable
  but non-$D$-$0$-stable,  if
 $\kappa$ singular
  then $L_\eta[D]\sats$ ``$\kappa$ is singular'',
  there is no $\eta'<\eta$
  with these properties with respect to $D,\xi$,
  and $T$ is the theory as in line~(\ref{eqn:theory_T_D}). Because of the minimality of $\eta$,
  line~(\ref{eqn:L_eta[D]_is_hull}) then also holds.

  \begin{clm}\label{clm:stat_domain}
  There are stationarily many $\xi<\kappa^+$ such that there are $T,A$
  with $(\xi,(T,A))\in f$,
  as witnessed by $D=U$ and some $\eta$.\end{clm}

  \begin{proof}

  By Lemma~\ref{fact:every_ord_ev_stable}, for each $\xi\leq\kappa^+$,
  there is $\alpha<\kappa$ such that $\xi$ is $\alpha$-stable.
 So for some $\alpha<\kappa$,
  there are cofinally many $\xi<\kappa^+$ which are $\alpha$-stable. Then since $\kappa,\kappa^+$ are $0$-stable, it follows
  that in fact there are cofinally many $\xi<\kappa^+$ which are $0$-stable.

  Let $\eta_0$ be the least $\eta>\kappa$
  such that if $\kappa$ is singular then $L_\eta[U]\sats$ ``$\kappa$ is singular''.
  Now it suffices to see that there are stationarily many $\xi\in[\eta_0,\kappa^+)$ which are $1$-stable but non-$0$-stable, since for all such $\xi$,
  there is $(T,A)$ as claimed.
  So let $C\sub\kappa^+$ be club in $\kappa^+$.
  Let  $\xi$ be the least ordinal $\geq\eta_0$ of cofinality $\mu$ which is a limit point of $C$
  (so $\xi\in C$ also)
  and is a limit of $0$-stable ordinals. Because $\cof(\xi)=\mu$,
  $\xi$ is non-$0$-stable
  and $i^{\Tt_U}_{1\beta}$ is continuous at $\xi$ for each $\beta\in[1,\kappa)$ (we have $\cof^{M^{\Tt_U}_1}(\xi)=\cof(\xi)=\mu$,
  so $\cof^{M^{\Tt_U}_\beta}(i^{\Tt_U}_{1\beta}(\xi))=\mu$ for all $\beta\in[1,\kappa)$).
    But also because $\xi$ is a limit of $0$-stable ordinals
    (hence also $1$-stable),
  it follows that $\xi$ is $1$-stable, which suffices.
\end{proof}

  \begin{clm}
   $f$ is a function.\end{clm}
   \begin{proof}Let $\xi<\kappa^+$  and $A,T,A',T'$ be such that $(\xi,(A,T))\in f$, as witnessed by $D,\eta$, and $(\xi,(A',T'))\in f$, as witnessed by $D',\eta'$.
  Then observe that $D=D'$, $A=A'$,
  $\eta=\eta'$, and $T=T'$,
  as desired.\end{proof}

  Let $d=\dom(f)$.
 \begin{clm} $d$ is co-stationary in $\kappa^+$.\end{clm}
 \begin{proof}By Theorem~\ref{tm:above_one_measurable} part~\ref{item:no_injection_kappa^+_to_p(kappa)_allowing_countable_cof},
  or more directly, just observe (by a simplification of the proof of Claim~\ref{clm:stat_domain}) that there are stationarily many $\xi<\kappa^+$ which are $0$-stable.\end{proof}

  It just remains to verify:
  \begin{clm}
   $d,f$ are $\Sigma_1(\{\kappa\})$.
  \end{clm}
\begin{proof}
It is straightforward to see that most of the defining clauses are $\Sigma_1(\{\kappa\})$, but there is a subtlety regarding the clause ``$\xi$ is $D$-$1$-stable'', and also a minor subtlety regarding the singularity of $\kappa$.
 First, regarding singularity, note that we do not include  a clause expressing ``if $\kappa$ is singular''; instead, just in the case that $\kappa$ \emph{is} singular,
   we demand that $L_\eta[D]\sats$ ``$\kappa$ is singular''.
  Second, regarding $D$-$1$-stability, let $\theta=\cof^{L[D]}(\kappa)$. Note that by the requirements on $(\xi,T,A,\eta,D)$,
$\eta$ is the least $\eta'>\xi$
such that $L[D]|\eta'\sats$ ``$\card(\xi)=\kappa$ and $\cof(\kappa)=\theta$''; therefore, $\eta=\eta''+\om$ for some $\eta''$ such that $L[D]|\eta''$ projects to $\kappa$, and there is a bijection $\pi:\kappa\to\xi$ and a cofinal function $h:\theta\to\kappa$ which are
both definable from parameters over $L[D]|\eta''$.
Let $\mu_D=\crit(D)$. We have $\mu_D<\kappa$,
so $\pow(\mu_D)\cap L[D]\sub L[D]|\kappa\sub L[D]|\eta$.

It suffices to see that
for all $g\in L[D]$ with $g:\mu_D\to\eta$,
there is $X\in D$ such that $g\rest X\in L[D]|\eta$, since then $i^{\Tt_D^{L[D]}}_{0\alpha}(\xi)=i^{\Tt^{L[D]|\eta}_D}_{0\alpha}(\xi)$ for all $\alpha<\kappa$,
where $\Tt^{L[D]|\eta}_D$ is the linear iteration of $L[D]|\eta$
(whereas $\Tt_D^{L[D]}$ is the linear iteration of $L[D]$).

Now because $\kappa$ is $\mu$-steady,
$\cof(\kappa)\neq\mu$,
and it follows that $\theta=\cof^{L[D]}(\kappa)\neq\mu_D$. (Indeed, if $D\neq U$
then we can fix $\alpha<\kappa$ such that
 $(D,\mu_D)=i^{\Tt_U}_{0\alpha}(U,\mu)$.
 But $i^{\Tt_U}_{0\alpha}(\kappa)=\kappa$
 since $\kappa$ is $\mu$-steady,
 so letting $\theta'=\cof(\kappa)\neq\mu$,
 we have $i^{\Tt_U}_{0\alpha}(\theta')=\theta$,
 so $\theta\neq\mu_D$.)
 So fix $g:\mu_D\to\eta$ with $g\in L[D]$.
 Then there is $X_0\in D$ such that $g\rest X_0$ is bounded in $\eta$, since $\eta=\eta'+\om$. Thus,
 we may assume $g:\mu_D\to\eta'$.
Let $\pi:\kappa\to\eta'$ be a bijection in $L[D]|\eta$.
Let $g'=\pi^{-1}\com g:\mu_D\to\kappa$.
 Because $\cof^{L[D]}(\kappa)\neq\mu_D$,
 there is $X\in L[D]$ such that $g'\rest X$ is bounded in $\kappa$, and therefore $g'\rest X\in L[D]|\kappa$. But then $g\rest X=\pi\com g'\rest X\in L[D]|\eta$,
 as desired.
\end{proof}
  Part~\ref{item:no_perf_subset}:
 We have two cases:

  \begin{case}\label{case:mu^+<kappa_and_kappa_mu-steady} $\mu^+<\kappa$
  and $\kappa$ is $\mu$-steady.

  Let $f:d\to\pow(\kappa)$ be as in part~\ref{item:injective_func_exists}
  and $X=\rg(f)$. We claim this set works.
  For clearly $X$ is $\Sigma_1(\{\kappa\})$ and $\card(X)=\kappa^+$, so we just need to see that $X$ has no simply perfect subset.

  Suppose otherwise. Let
    $\theta=\cof(\kappa)$ and
  $\sigma:{^{<\theta}}2\to{^{<\kappa}}2$ be a simply perfect function with $(\bigcup_{\alpha<\theta} \sigma(x\rest\alpha))\in X$ for all $x\in{^\theta}2$.
  Let $\PP=\mathrm{Add}(\theta,1)$.
  Note that if $\theta=\om$, this is just Cohen forcing. Then $\PP$ is $\theta$-closed.
  Let $g$ be $(V,\PP)$-generic.
  Then $({^{<\theta}}V)\cap V[g]\sub V$.

  Work in $V[g]$.
 Let $x=(\bigcup g)\in{^\theta}2$.
 Let $(A,T)=\bigcup_{\alpha<\theta} \sigma(x\rest\alpha)$.
 We have $(A,T)\notin V$,
 since $x\notin V$, and $(A,T)$ determines $x$ via $\sigma$.
 But we will observe that $(A,T)$
satisfies the conditions required for elements of $X$ (except that it is in $V[g]\cut V$); that is, there are some (uniquely determined) corresponding $D,\eta,\xi$
with those properties holding in $V[g]$. But then by the uniqueness of $L[U]$ and its iterates, we get $D,T,A\in V$, and therefore $x\in V$, a contradiction.

So, we have $A,T\in V[g]$.
We want to see there are $D,\eta,\xi$
with the right properties. First note that $A,T$ are unbounded in $\kappa$,
by genericity and since for every $y\in({^\theta}2)^V$, the sets $A',T'$ coded by $\sigma(y)$ (which is in $X$) are unbounded in $\kappa$.
Let us verify that $A,T$ satisfy the right model-theoretic properties (ignoring wellfoundedness and iterability for the moment). Clearly $T$  contains the basic first order theory needed (like ``$V=L[\dot{D}]$'', etc.)
and is complete and consistent,
and $T$ is  a Skolemized theory, since levels of $L[D]$ have a definable wellorder (that is, for each formula of form $\text{``}\exists x\ \psi(x,\vec{\alpha})\text{''}$ in $T$,
where $\vec{\alpha}\in\kappa^{<\om}$,
 there is a term $t$ such that ``$\psi(t(\vec{\alpha}),\vec{\alpha})\text{''}$ is in $T$).
 We also need to know that we faithfully represent the ordinals $\leq\kappa$. But clearly for all ordinals $\alpha<\beta<\kappa$,
 the formula ``$\alpha<\beta<\dot{\kappa}$, where $\dot{\kappa}$ is the largest cardinal'' is in $T$.
 And for each term $t$
 and all $\vec{\alpha}\in\kappa^{<\om}$, if ``$t(\vec{\alpha})<\dot{\kappa}$ where $\dot{\kappa}$ is the largest cardinal'' is in $T$,
 then there is $\beta<\kappa$
 such that ``$t(\vec{\alpha})=\beta$'' is in $T$: the latter follows
  again
from genericity, and since in $V$, for each $s\in{^{<\theta}2}$,
taking any $y\in{^{\theta}2}$ with $s\sub y$, we have $(\bigcup_{\alpha<\theta} \sigma(y\rest\alpha))\in X$, and so we can extend $s$ to some $s'\sub y$
to ensure an equation of the desired form gets into (the generic) $T$.
This gives the desired
model-theoretic properties.

Let $M_T$ be the model determined by $T$.
Suppose first that $\theta>\om$. Then $M_T$ is wellfounded, since then every countable substructure of $M_T$ can be realized in the model $L[D']|\eta'$ corresponding to some branch in $V$. So $M_T=L[D]|\eta$ where $L[D]|\eta\sats$ ``$D$ is a $\kappa$-good filter''.
Since $\om_1^{V[g]}=\om_1^V<\kappa\sub M_T$, it also follows that $D$ is truly iterable, and so $(A,T)$ does indeed satisfy the requirements for elements of $X$,
yielding the desired contradiction.

Now suppose that $\theta=\om$.
We use a different argument to see that $M_T$ is wellfounded.
Working in $V$, consider  the tree $S$ of attempts to build a real $y\in{^\om}2$ and a descending sequence of ordinals through the model $M_y$ determined by $\bigcup_{n<\om}(\sigma(y\rest n))$.
(Here the $n$th level of $S$
specifies some finite segment of $y$ of length $\geq n$, and terms defining the first $n$ elements of the desired descending sequence through (the to-be-completed) $\OR^{M_y}$.)
Then since (in $V$) every branch through $\sigma$ gives a wellfounded model, $S$ is wellfounded.
It follows that $M_T$ is wellfounded.
The rest of the argument is as in the case that $\theta>\om$ (the value of $\theta$ was not relevant to the iterability of $D$; we still have $\om_1^{V[G]}=\om_1^V<\kappa\sub M_T$).
\end{case}

  \begin{case} $\om<\kappa\leq\mu^+$ or [$\mu^+<\kappa$ and $\kappa$ is non-$\mu$-steady].

  This is a slight variant of the previous case, substituting
  the calculation of \cite{meas_cards_good_Sigma_1_wo}
  to identify $\es^{L[U]}$ appropriately, instead of the methods we used above. We leave the details to the reader.
  \qedhere
  \end{case}
\end{proof}

\begin{dfn}\label{dfn:4.2}
 Fix an uncountable cardinal $\kappa$.
 Relative to $\kappa$,
 and some definability class $\Gamma$, we say that $A\sub\kappa$
 is a \emph{$\Gamma$-singleton}
 iff $\{A\}$ is $\Gamma$-definable. As usual, we say that some class $C$ of singletons forms a \emph{basis} for some definability class $\Gamma$
 iff every non-empty $X\sub\pow(\kappa)$
 which is $\Gamma$-definable
 has an element $x\in C$.
\end{dfn}

\begin{tm}\label{tm:L[U]_non-basis}
 Assume $V=L[U]$ where $U$ is a normal measure on $\mu$. Let $\kappa>\mu^+$ be a $\mu$-steady cardinal.
 Then:
 \begin{enumerate}
 \item\label{item:every_Sigma_1_singleton_in_M_kappa} For every $0$-stable ordinal $\alpha$,
 every  $\Sigma_1(\{\alpha\})$-singleton is in $M^{\Tt_U}_\kappa$. In particular,
 every $\Sigma_1(\{\kappa\})$-singleton is in $M^{\Tt_U}_{\kappa}$.
 \item\label{item:Pi_1_singletons_not_all_in_M^T_kappa} Not every
 $\Pi_1(\{\kappa\})$ singleton
 is in $M^{\Tt_U}_\kappa$.
  \item\label{item:Sigma_1_singletons_non-basis} The  $\Sigma_1(\{\kappa\})$ singletons do not form a basis for $\Delta_1(\{\kappa\})$
  \tu{(}thus, neither $\Sigma_1(\{\kappa\})$ nor $\Delta_1(\{\kappa\})$ have the uniformization property\tu{)}.
  \item\label{item:Pi_1_singletons_non-basis} In fact, the $\Pi_1(\{\kappa\})$
  singletons do not form a basis for  $\Sigma_1(\{\kappa\})$.
 \end{enumerate}
\end{tm}
\begin{proof}
 Part~\ref{item:every_Sigma_1_singleton_in_M_kappa}: This is an immediate corollary of Theorem~\ref{tm:above_one_measurable} part~\ref{item:wellorders_of_kappa},
 since for  $A$ a set of ordinals,
 $A$ is a $\Sigma_1(S)$-singleton
 iff $A$ is $\Delta_1(S)$.

 Part~\ref{item:Pi_1_singletons_not_all_in_M^T_kappa}:
 $U$, coded naturally as a subset of $\kappa$ (using the $L[U]$-order of constructibility) is a such a $\Pi_1(\{\kappa\})$ singleton, since it is the unique
 $\kappa$-good filter $D$ such that every bounded subset of $\kappa$ is in $L_\kappa[D]$.

 Part~\ref{item:Sigma_1_singletons_non-basis}: Let $X=\pow(\kappa)\cut M^{\Tt_U}_\kappa$. Then $X$ is $\Delta_1(\{\kappa\})$, since the following three conditions are equivalent: (i)
 $A\in X$, (ii)
  $A\sub\kappa$ and there is a $\kappa$-good filter $D$
  such that  $A\notin M^{\Tt_D^{L[D]|\kappa}}_\kappa$,
  and (iii)
  $A\sub\kappa$
 and for every $\kappa$-good filter $D$,
 we have $A\notin M^{\Tt_D^{L[D]|\kappa}}_\kappa$.
 (Here $\Tt^{L[D]|\kappa}_D$ is the iteration on $L[D]|\kappa$ given by iterating $D$.)
 But by part~\ref{item:every_Sigma_1_singleton_in_M_kappa},
 there is no $\Sigma_1(\{\kappa\})$
 singleton $A\in X$.

 Part~\ref{item:Pi_1_singletons_non-basis}: Let $X$ be the set of all $A\sub\kappa$
 such that for some limit ordinal $\eta<\kappa$,
 we have $A\in\bigcap_{\alpha<\eta}M^{\Tt_U}_\alpha$
 but $A\notin M^{\Tt_U}_\eta$.
 Note that by Lemma~\ref{lem:charac_sets_in_M_eta_limit_eta}, $X$ is $\Sigma_1(\{\kappa\})$.
 Suppose there is a $\Pi_1(\{\kappa\})$
 singleton $A\in X$, and let $\psi$ be a $\Pi_1$ formula such that $A$ is the unique $B\sub\kappa$ such that $\psi(\kappa,B)$. Let $\eta$ be least such that $A\notin M^{\Tt_U}_\kappa$;
 so $\eta$ is a limit, as $A\in X$.
 By $\Pi_1$ downward absoluteness, for all $\alpha<\eta$,
 $M^{\Tt_U}_\alpha\sats\psi(\kappa,A)$. But by elementarity of $i^{\Tt_U}_{0\alpha}$, therefore $M^{\Tt_U}_\alpha\sats$ ``$A$ is the unique $B\sub\kappa$ such that $\psi(\kappa,B)$''. It follows that $i^{\Tt_U}_{0\alpha}(A)=A$.
 So $A$ is $[0,\eta)$-stable,
 so by Lemma~\ref{lem:charac_sets_in_M_eta_limit_eta},
 $A\in M^{\Tt_U}_\eta$, a contradiction.
\end{proof}

\section{$\Gamma$-Good Wellorders in 1-Small Mice}\label{sec:simply_def_good_WOs}

As mentioned at the start of \S\ref{sec:wo_above_one_meas},
L\"ucke and Schlicht \cite{meas_cards_good_Sigma_1_wo} characterized the uncountable cardinals $\kappa$ of $L[U]$ for which there is a $\Sigma_1(\{\kappa\})$-good wellorder of $\her^{L[U]}_{\kappa^{+L[U]}}$. In the cases that these exist, the restriction of $<^{L[U]}$ is itself $\Sigma_1(\{\kappa\})$-good.
Since under ZFC, if $\kappa$ is a limit of measurable cardinals then there is no $\Sigma_1(V_\kappa\cup\OR)$ wellorder of $\pow(\kappa)$,
the following result is close to optimal:
\begin{tm}\label{tm:good_wo_1-small}
 Let $M$ be a 1-small mouse. Let $\kappa$ be an uncountable cardinal of $M$ which  is not a limit of measurable cardinals of $M$ and such that $M\sats$ ``$\kappa^+$ exists'' \tu{(}that is, $\kappa^{+M}<\OR^M$\tu{)}.
 Suppose that for every $N\pins M|\kappa^{+M}$ and every $\Tt\in M$ which is a limit length  $\om$-maximal tree on $N$ \tu{(}so $\delta(\Tt)<\OR^M$\tu{)}, letting $\eta$ be such that $\delta(\Tt)+\eta=\OR^M$, it is not the case that $\J_\eta[M(\Tt)]\sats$ ``$\delta(\Tt)$ is Woodin''.
 Then there is a wellorder of $\her_{\kappa^{+M}}^M$ which is $\Sigma_1^{\univ{M}}(\her_\kappa^M\cup\{\kappa\})$-good.
\end{tm}

It is of note that in case $\kappa$ is a limit cardinal of $M$, the  wellorder we use to prove the theorem is not just the restriction $<^*$ of the standard order of constructibility $<^M$ of $M$ to $\her^M_{\kappa^{+M}}$; there are in fact cases in which this restriction
is \emph{not} $\Sigma_1^{\univ{M}}(\her_\kappa^M\cup\{\kappa\})$-good; see Theorem~\ref{tm:om_1-it} for an example. (But this seems to leave open the possibility
that this ordering is, however, $\Sigma_1^{\univ{M}}(\her_\kappa^M\cup\{\kappa\})$ (\emph{-definable}, but not -good).) In case $\kappa$ is a successor cardinal of $M$, we will be able to simply use the order of constructibility.
But if $\kappa$ is a limit cardinal,
we will combine the mouse order with various orders of constructibility.
Here is the basic idea of the order in the latter case, if it is not quite precise: For $X\in\her_{\kappa^{+M}}^M$,
let  $N_X$ be least in the mouse order with $X\in N_X$.
Given sets $X,Y$,
we want to order $X<^*Y$
iff either $N_X$ is strictly below $N_Y$
in the mouse order,
or $N_X=N_Y$ and $X<^{N_X}Y$, where $<^{N_X}$ is the usual order of constructibility of $N_X$.
We will refine this approximation in what follows.

For the proof we use
the following definition,
which is based on standard arguments for absoluteness of iterability in the absence of a proper class model with a Woodin cardinal, like those in \cite[\S2]{cmip}, and like \cite[Lemma 2.1]{schindler_star}.
Its role here will be analogous to that in V\"a\"an\"anen-Welch \cite{power_set_Sigma_1_Card}:
\begin{dfn}\label{dfn:5.2}
 Let $N$ be a sound $1$-small premouse and $\kappa<\xi$ be  ordinals such that $L_\xi[N]\sats$ ``$\om_1$ exists and $\om_1\leq\kappa$''. Working in   $L_\xi[N]$, let $S_{N\kappa}$ be the tree of attempts to build, via finite approximations,
 a tuple \begin{equation*}\label{eqn:label_for_eqn_25}(\bar{N},\sigma,\Tt,\rho,\left<\rho_\alpha\right>_{\alpha+1<\lh(\Tt)},\iota,Q,\pi)\end{equation*} such that:
 \begin{enumerate}\item $\bar{N}$ is a structure in the language of premice, whose universe is $\omega$,
  \item
$\sigma:\bar{N}\to N$ is elementary,
\item $\Tt$ is a putative $\om$-maximal iteration tree on $\bar{N}$,
\item $\rho:\lh(\Tt)\to\kappa$ is order-preserving,
\item  $\rho_\alpha:\OR(M^{\Tt}_\alpha)\to\kappa$ is order-preserving, for each $\alpha+1<\lh(\Tt)$,
\item
 if $\lh(\Tt)=\alpha+1$
 then:
 \begin{enumerate}
  \item $\iota$ is  a strictly descending sequence through $\OR(M^{\Tt}_{\alpha})$,
  \item if $\alpha$ is a limit, then $Q\pins M^\Tt_\alpha$ is a Q-structure for $M(\Tt\rest\alpha)$ and
  $\pi:\OR^Q\to\kappa$ is order-preserving,
  \end{enumerate}
  \item if $\lh(\Tt)$ is a limit then:
  \begin{enumerate}
   \item
   $Q=\J_\gamma(M(\Tt))$
   for some ordinal $\gamma$,
and  either $\rho_{n+1}^Q<\delta(\Tt)\leq\rho_n^Q$ for some $n<\om$, or there is a failure of Woodinness of $\delta(\Tt)$
  definable over $Q$,
  \item $\pi:\OR^Q\to\kappa$ is order-preserving,
  \item  $\iota$ is a ranking into $\kappa$
  of the tree of attempts to build a $\Tt$-cofinal branch $c$ with $Q\ins M^{\Tt}_c$
  or $M^{\Tt}_c\pins Q$.

 \end{enumerate}
 \end{enumerate}
 Note that we can form $S_{N\kappa}$ as a tree on $\om\cross\max(\OR^N,\kappa)$.

We say that $N$ is \emph{$\kappa$-strongly iterable}
 in $L_\xi[N]$ iff $L_\xi[N]\sats$ ``there is a ranking of $S_{N\kappa}$ into the ordinals'',
 and \emph{strongly iterable}
 in $L_\xi[N]$ iff $N$ is $\om_1^N$-strongly iterable in $L_\xi[N]$.
\end{dfn}

\begin{proof}[Proof of Theorem~\ref{tm:good_wo_1-small}]
The case that $\kappa$ is a successor cardinal is easier. We deal with it first.

\begin{casetwo}
 $\kappa=\gamma^{+M}$ for some $M$-cardinal $\gamma\geq\om$.

 In this case, the restriction of $M$'s usual order of constructibility  is $(\Sigma_1(\her_\kappa\cup\{\kappa\}))^{\univ{M}}$-good. To see this, it suffices
 to show that $M|\kappa^{+M}$ (or equivalently, $\es^{M|\kappa^{+M}}$) is $\Sigma_1^{\univ{M}}(\her_\kappa^M\cup\{\kappa\})$.
 \begin{scasetwo}\label{scase:cutpoint_eta_above_lgcd}
  There is a cutpoint $\eta$ of $M$
  such that $\gamma\leq\eta<\kappa=\gamma^{+M}$.
  (Note this includes the case that $\om=\gamma=\eta<\kappa=\om_1^M$.
 Recall that $\eta$ is a \emph{cutpoint} of $M$
  iff there is no $E\in\es_+^M$ with $\crit(E)<\eta<\lh(E)$.)

  Fix such a cutpoint $\eta$. Let $x=M|\eta$. Then
  $M|\kappa$ is just the stack of sound  premice $N\in M$ such that $M|\eta\pins N$,  $\rho_\om^N\leq\gamma$,
  $\eta$ is a cutpoint of $N$ and there is an ordinal $\xi\in(\kappa,\kappa^{+M})$ such that $L_\xi[N]\sats$ ``$N$ is $\kappa$-strongly iterable'' (note that then $\om_1^M\leq\kappa\in L_\xi[N]$).
  This yields a definition of $N|\kappa=N|\gamma^{+M}$ which is $\Sigma_1^{\univ{M}}(\{x,\kappa\})$.

  We now define $M|\kappa^{+M}$
  (as a premouse).
  Since $\kappa$ is regular in $M$,
  $M|\kappa^{+M}$ is just the Jensen stack over $M|\kappa$,
  as computed in $M$. (This is the stack of all sound premice $P$ such that $M|\kappa\pins P$, $\rho_\om^P=\kappa$,
  and $P$ satisfies condensation;
  we don't need to assert any iterability here. See \cite[Fact 3.1]{V=HODX_pub}.) This is also $\Sigma_1^{\univ{M}}(\{x,\kappa\})$, and yields a $\Sigma_1^{\univ{M}}(\{x,\kappa\})$ definition of $M|\kappa^{+M}$, and hence
  of the usual order of constructibility over  $\her^M_{\kappa^{+M}}$. (Note  $\kappa$ might not be a cutpoint of $M$; equivalently, $\gamma$ might be measurable in $M$.)
 \end{scasetwo}
\begin{scasetwo}
 Otherwise (there is no such cutpoint).

 Set $x=M|\gamma$.
 For all $\eta\in[\gamma,\kappa)=[\gamma,\gamma^{+M})$,
 there is an extender $E\in\es^M$ with $\crit(E)<\gamma\leq\eta<\lh(E)$,
 and note that any such $E$ is $M$-total.
 By \cite[Corollary 2.18]{mim},
 \cite[Theorem 1.4]{V=HODX_pub}, it follows that for sound premice $N$
 such that $M|\gamma\ins N$
 and $\rho_\om^N=\gamma$,
 we  have $N\pins M|\kappa$
 iff there is an $M$-total extender $E\in M$ (not required to be in $\es^M$) such that $\mu=\crit(E)<\gamma$, $\Ult(M|\mu^{+M},E)$ is wellfounded
 and $N\pins\Ult(M|\mu^{+M},E)$.
 Note that this condition is $\Sigma_1^{\univ{M}}(\{x,\kappa\})$, and this suffices to define $M|\kappa=M|\gamma^{+M}$.

 From here we compute $M|\kappa^{+M}$ as before.
\end{scasetwo}

\end{casetwo}

 \begin{casetwo}\label{case:kappa_limit_card}
  $\kappa$ is a limit cardinal of $M$.

  In this case we will not use the  order of constructibility  of $M$ for our wellorder; instead we will use the order sketched just after the statement of Theorem~\ref{tm:good_wo_1-small}.

  \begin{scasetwo}\label{scase:limit_card_and_cutpoint}
   There is a cutpoint $\eta$
   of $M$ with $\delta<\eta<\kappa$, where $\delta$ is the supremum of $\om$ and all measurable cardinals of $M|\kappa$.

   Take $\eta$ such, and we may assume that $\eta=\theta^{+M}$ for some $\theta$. Set $x=M|\eta$. Working in $M$,
   say that a premouse $N$ is \emph{good} iff $M|\eta\pins N$, $\eta$ is a cutpoint of $N$,
   $\kappa<\OR^N$, $\rho_\om^N=\kappa$,
   $N$ is sound,
   $N\sats$ ``there are no measurable cardinals in the interval $[\eta,\kappa)$'',
   there is $\xi<\kappa^+$
   such that $L_\xi[N]\sats$ ``$N$ is strongly iterable'',
   and there is $\gamma\in(\kappa,\OR^N)$ such that $L_\gamma[N|\kappa]$ is admissible.
   Given a good $N$,
   let $\gamma^N$ be the least $\gamma$ such that $L_\gamma[N|\kappa]$ is admissible.

   \begin{clmtwo}\label{clm:N_good_analysis}
    Let $N$ be good.
   Let $(\Tt,\Uu)$ be the comparison of $M$ versus $N$, with $\om$-maximal trees. This can be executed working inside $M$. Moreover, $\Uu$ is trivial,
  and either:
    \begin{enumerate}
     \item[--] $N\ins M$ and $\Tt$ is trivial, or
     \item[--] $N\nins M$, $\Tt$ is non-trivial, and:
     \begin{enumerate}\item $\Tt$ is linear,
     \item $\Tt$ is above $\eta$, based on $M|\kappa$,
     and $\lh(\Tt)=\kappa+1$,
     \item $N\pins M^\Tt_\kappa$,
     \item
     $\Tt\in M$,
     \item $1\in\mathscr{D}^\Tt$,
     \item
    for every $\alpha+1<\lh(\Tt)$,  $\eta<\crit(E^\Tt_\alpha)$ and $M^\Tt_\alpha|\crit(E^\Tt_\alpha)\sats$ ``There are no measurable cardinals $\geq\eta$'',
    \item for all $\alpha+1\leq\beta\leq\kappa$,
     if $(\alpha+1,\beta]^\Tt$ does not drop in model
     then:
     \begin{enumerate}\item for all $\gamma\in(\alpha,\beta)$,
     $E^\Tt_\gamma$ is the unique
     $M^\Tt_\gamma$-total extender  $E\in\es_+(M^\Tt_\gamma)$
     with $\crit(E)=i^{*\Tt}_{\alpha+1,\gamma}(\crit(E^\Tt_\alpha))$,
     and \item if $\beta<\kappa$ then:
     \begin{enumerate}\item $\lh(E^\Tt_\beta)\leq i^{*\Tt}_{\alpha+1,\beta}(\lh(E^\Tt_\alpha))$ (where if $E^\Tt_\alpha=F^{M^\Tt_\alpha}$ then this just means that $\lh(E^\Tt_\beta)\leq\lh(F^{M^\Tt_\beta})$, which is true anyway),
     and \item $\lh(E^\Tt_\beta)<i^{*\Tt}_{\alpha+1,\beta}(\lh(E^\Tt_\alpha))$ iff $E^\Tt_\beta\neq i^{*\Tt}_{\alpha+1,\beta}(E^\Tt_\alpha)$ iff $\beta+1\in\dropset^\Tt$.
     \end{enumerate}
     \end{enumerate}
    \end{enumerate}
\end{enumerate}
Moreover,
the tree $\Tt$ mentioned above is uniquely determined.
   \end{clmtwo}

The claim above follows from standard analysis of the comparison, combined with the fact that  $M$ itself is the universe computing the comparison and that $\eta$ is a cutpoint of the models, and neither model has measurable cardinals in the interval $[\eta,\kappa)$.
   The following claim is now a straightforward consequence:
   \begin{clmtwo}
    Let $N,N'$ be good, with $N,N'\nins M$, and let $\Tt,\Tt'$ be the respective trees as in Claim~\ref{clm:N_good_analysis}.
    Then either:
    \begin{enumerate}
     \item $N|\kappa=N'|\kappa$
     and $\Tt=\Tt'$
     and either $N\ins N'$ or $N'\ins N$,
     or
     \item $N|\kappa\neq N'|\kappa$, there is some (uniquely determined) $\alpha<\kappa$
     such that
     \begin{equation*}\label{eqn:label_for_eqn_26}\Tt\rest(\alpha+1)=\Tt'\rest(\alpha+1)\end{equation*}
     and $E^\Tt_\alpha\neq E^{\Tt'}_\alpha$,
     and if $\lh(E^\Tt_\alpha)<\lh(E^{\Tt'}_\alpha)$,
     then $\Tt\rest[\alpha,\kappa+1)$ can be considered as a tree on $N'|\kappa$,
     with properties like those of the trees in Claim~\ref{clm:N_good_analysis},
     and  $\Tt\rest[\alpha,\kappa+1)\in N'|\gamma^{N'}$
     and $N\in N'|\gamma^{N'}$;
     otherwise $\lh(E^{\Tt'}_\alpha)<\lh(E^\Tt_\alpha)$ and it is symmetric.
    \end{enumerate}

   \end{clmtwo}

   And the next claim is also a straightforward consequence of the previous two:
   \begin{clmtwo}
  $\in$ wellorders the good premice.
   \end{clmtwo}

   Now given $X\in\her^M_{\kappa^{+M}}$, let $N_X$ be the $\in$-least good premouse $N$ with $X\in N$.
   We define a good wellorder $<^*$
   over $\her^M_{\kappa^{+M}}$
   by setting $X<^*Y$ iff either:
   \begin{enumerate}
    \item[--] $N_X\in N_{Y}$, or
    \item[--] $N_X=N_Y$ and $X<^{N_X}Y$ (where $<^{N_X}$ is the order of constructibility of $N_X$).
   \end{enumerate}

   \begin{clmtwo}
    $<^*$ is a $\Sigma_1^{\univ{M}}(\{x,\kappa\})$-good wellorder of $\her^M_{\kappa^{+M}}$.\end{clmtwo}
    \begin{proof}
     This is straightforward;
     just note that if $N,N'$ are good and $N'\in N$
     with $N|\kappa\neq N'|\kappa$,
     then in fact $N'\pins M^{\Tt}_\kappa$
     for some $\Tt\in N$ as described above.
     So $N$ can identify the collection of good premice $N'\in N$ by making use of this.
     Thus, we get that $X<^*Y$ iff there are good premice $N,P$ with $X\in N$ and $Y\in P$,
     and $N$ thinks
      (using the preceding discussion) that there is no good premouse $N'\in N$ with $X\in N'$,
      and likewise $P$ thinks there is no good $P'\in P$ with $Y\in P'$
      (this ensures that $N=N_X$ and $P=N_Y$),
      and either $N\in P$, or $N=P$ and $X<^NY$.
    \end{proof}

  \end{scasetwo}

  \begin{scasetwo} Otherwise.

   Fix an $M$-cardinal $\eta$
   such that $\delta<\eta<\kappa$,
   where $\delta$ is the supremum of the $M$-measurables which are $<\kappa$
   (by the subcase hypothesis, there are some),
   and $\eta=\theta^{+M}$ for some $\theta$. Set $x=M|\eta$.
   Working in $M$, we define \emph{good} premice
   as before (from the parameter $M|\eta$), except that we don't demand that $\eta$ be a cutpoint of $N$.

   \begin{clmtwo}\label{clm:N_good_analysis_2}
    Let $N$ be good. Then either:
    \begin{enumerate}
     \item[--] $N\ins M$, or
     \item[--] there is an $\om$-maximal  iteration tree $\Tt$ on $M$ such that:
     \begin{enumerate}
     \item $\eta<\lh(E^\Tt_0)$, $\Tt$ is based on $M|\kappa$,
     and $\lh(\Tt)=\kappa+1$,
     \item $N\ins M^\Tt_\kappa$,
     \item
     $\Tt\in M$,
     \item if $\crit(E^\Tt_\alpha)<\eta$ then $\pred^\Tt(\alpha+1)=0$ and $\alpha+1\notin\dropset^{\Tt}$,
    \item if $\crit(E^\Tt_\alpha)\geq\eta$ then $\pred^\Tt(\alpha+1)=\alpha$,
     \item\label{item:overstated} if $[0,\alpha+1]^\Tt\cap\dropset^\Tt=\emptyset$
     then $\crit(i^\Tt_{0,\alpha+1})<\eta$,
     \item
if $\eta\leq\crit(E^\Tt_\alpha)$ then $M^\Tt_\alpha||\lh(E^\Tt_\alpha)\sats$ ``There are no measurable cardinals $>\delta$'',
    and
    \item for all $\alpha+1\leq\beta\leq\kappa$,
     if $\eta\leq\crit(E^\Tt_\alpha)$ and $(\alpha+1,\beta]^\Tt$ does not drop in model
     then:
     \begin{enumerate}\item for all $\gamma\in(\alpha,\beta)$,
     $E^\Tt_\gamma$ is the unique
     $M^\Tt_\gamma$-total extender  $E\in\es_+(M^\Tt_\gamma)$
     with $\crit(E)=i^{*\Tt}_{\alpha+1,\gamma}(\crit(E^\Tt_\alpha))$,
     and \item if $\beta<\kappa$ then:
     \begin{enumerate}\item $\lh(E^\Tt_\beta)\leq i^{*\Tt}_{\alpha+1,\beta}(\lh(E^\Tt_\alpha))$,
     and
     \item if $\pred^\Tt(\beta+1)=\beta$
     then
     \begin{equation*}\label{eqn:label_for_eqn_27} \lh(E^\Tt_\beta)<i^{*\Tt}_{\alpha+1,\beta}(\lh(E^\Tt_\alpha))\iff E^\Tt_\beta\neq i^{*\Tt}_{\alpha+1,\beta}(E^\Tt_\alpha)\iff\beta+1\in\dropset^\Tt.\end{equation*}
     \end{enumerate}
     \end{enumerate}
     \item for all $\alpha+1<\kappa$,
     if $\mu=\crit(E^\Tt_\alpha)<\eta$
     then letting $\mu'$
     be least such that $\mu'\geq\eta$ and $\mu'$ is measurable in $M^\Tt_{\alpha+1}$, then $\mu'\leq i^{\Tt}_{0,\alpha+1}(\mu)$, and letting
      $E\in\es_+(M^\Tt_{\alpha+1})$
     be the unique extender which is $M^\Tt_{\alpha+1}$-total
     with $\crit(E)=\mu'$, we have:
     \begin{enumerate}\item
     $\lh(E^\Tt_{\alpha+1})\leq\lh(E)$, and
     \item if $\pred^\Tt(\alpha+2)=\alpha+1$
     then
     \begin{equation*}\label{eqn:label_for_eqn_28} \lh(E^\Tt_{\alpha+1})<\lh(E)\iff E^\Tt_{\alpha+1}\neq E\iff \alpha+2\in\dropset^\Tt.\end{equation*}
     \end{enumerate}
    \end{enumerate}
\end{enumerate}
Moreover, in case $N\nins M$,
the tree $\Tt$ mentioned above is uniquely determined.
   \end{clmtwo}
   \begin{proof}
   We use an argument from \cite{mim} (also see \cite{extmax}), using a property related to a feature of $K$ (see \cite{maxcore}).
   Let $(\Tt,\Uu)$ be the comparison of $(M,N)$.
   The main thing is to see that  $\Uu$ is trivial,
   and for this, we need to see that $\Uu$ is above $\eta$.
   But otherwise, letting $\alpha$ be least (or in fact any ordinal) such that $\crit(E^\Uu_\alpha)<\eta$,
   we get that $E^\Uu_\alpha\in M$
   and $(M^\Tt_\alpha||\lh(E^\Uu_\alpha),E^\Uu_\alpha)$
   is a premouse. By \cite[Theorem 3.8]{extmax}, and since $\Tt\in M$, it follows that $E^\Uu_\alpha\in\es_+(M^\Tt_\alpha)$, so $E^\Uu_\alpha$ was not causing a disagreement, a contradiction.
   \end{proof}

   Note also that $\Tt$ can only use finitely many extenders $E$ with $\crit(E)<\eta$,
    since $\Tt$ is linear in the intervals between those $E$'s. Thus, we again get that if $N,N'$ are good with $N|\kappa\neq N'|\kappa$,
   then letting $\Tt,\Tt'$ be the corresponding trees on $M$, either $\Tt$ can be converted into an essentially equivalent tree $\Uu$ on $N'$
   with $\Uu\in N'$ and  $N\pins M^\Uu_\kappa$,
   or vice versa. So the rest is just like in Subcase~\ref{scase:limit_card_and_cutpoint}.\qedhere
   \end{scasetwo}
 \end{casetwo}
\end{proof}

Applying the previous result to $M_1|\delta^{M_1}$, combined with \ref{tm:above_inf_many_meas}, we have:

\begin{cor}\label{cor:confusing_statement}
 Suppose $M_1$ exists and is fully iterable.
 Then $M_1\sats$ ``for every uncountable cardinal $\kappa$,
 the following are equivalent:\begin{enumerate}\item[\tu{(}a\tu{)}] $\kappa$ is not a limit of measurable cardinals,\item[\tu{(}b\tu{)}] there is a $\Sigma_1^{\univ{M}}(\her_\kappa\cup\{\kappa\})$ wellorder of a subset of $\pow(\kappa)$ of cardinality $>\kappa$,
 \item[\tu{(}c\tu{)}] there is a $\Sigma_1^{\univ{M}}(\her_\kappa\cup\{\kappa\})$ wellorder of $\her_{\kappa^+}$,
\item[\tu{(}d\tu{)}] there is a $\Sigma_1^{\univ{M}}(\her_\kappa\cup\{\kappa\})$-good wellorder of $\her_{\kappa^+}$''.
\end{enumerate}
\end{cor}

In \cite{luecke_schindler_schlicht},
it is shown that if $M_1^\#(x)$ exists for all $x\sub\om_1$, then there is no $\Sigma_1(\{\om_1\})$ wellorder of $\pow(\om)$. In every proper class $1$-small mouse $M$, the standard $M$ order of construction, restricted to $\RR^M$,
is $\Sigma_1(\{\om_1\})$. So by the following fact,  the least proper class mouse $M$ where this fails is $L[M_1^\#]$:
\begin{tm}\label{tm:in_M_1^sharp}
 In $L[M_1^\#]$, there is no $\Sigma_1(\{\omega_1\})$ wellorder of $\pow(\om)$,
 and in fact no OD wellorder of $\pow(\om)$.
\end{tm}
\begin{proof}
It is well known that $\om_1^{L[M_1^\#]}$ is measurable in $\HOD^{L[M_1^\#]}$
(see \cite[Remark 3.2]{odle_v2} for a proof). But if $L[M_1^\#]\sats$ ``there is an OD wellorder of $\pow(\om)$'',
then we would have $\RR^{M_1^\#}\sub\HOD^{L[M_1^\#]}$, a contradiction.
\end{proof}

\begin{rem}\label{rem:5.5}
 In \cite{mim}, \cite{extmax}
and \cite{V=HODX_pub},
there are techniques described via which, in various settings, a mouse $M$ can identify its own internal extender sequence $\es^M$ definably over its universe $\univ{M}$. Using some of those results, we get that in short extender mice $M$,
if $\kappa<\OR^M$ is an $M$-cardinal and either $M\sats$ ``$\kappa^+$ exists''
or $M\sats$ ZF$^-$,
the usual $M$-order of constructibility is $\Delta_4^{\her^M_{\kappa^{+M}}}(\{x\})$ where $x=M|\om_1^M$ (that is, $x$ is the initial segment of $M$ through to $\om_1^M$, including its extender sequence there (which consists of partial extenders only)). Moreover,
in many cases, we can make do with $\Delta_2$ instead of $\Delta_4$.
The author does not know whether one can in fact always make do with $\Delta_2$.  If $M$ is tame, the parameter $M|\om_1^M$ can  be replaced with a real (see \cite[Theorem 1.1]{odle_v2}).\end{rem}
\begin{rem}\label{rem:5.6}
V\"a\"an\"anen and
 Welch  argue in \cite{power_set_Sigma_1_Card}
that if $M$ is a proper class 1-small mouse satisfying ``there is no proper class model with a Woodin cardinal'' then in $\univ{M}$, the relation ``$x=\pow(y)$'' is $\Sigma_1(\mathrm{Card})$.
Their argument uses facts about the core model $K$ below 1 Woodin cardinal, and that $M=K^M$.
It uses in particular \cite[Lemma 3.9]{power_set_Sigma_1_Card}
(which as literally written
in \cite{PCF_and_Woodins}
also assumes a large enough measurable cardinal $\Omega$, which might not be available here; in \cite{power_set_Sigma_1_Card}
it is pointed out that this is not needed, given the more recent development of the theory of $K$ at this level without the measurable cardinal in \cite{Kwithoutmeas}).
Using instead an argument like that in the proof of Claim~\ref{clm:N_good_analysis_2} of the proof of Theorem~\ref{tm:good_wo_1-small},
one can prove the V\"a\"an\"anen-Welch result just mentioned without referring to $K$. That is, work in a proper class 1-small mouse $M$ which satisfies ``there is no proper class inner model with a Woodin cardinal'' and let $N$ be a mouse such that $\OR^N=\lambda^{++}$ for some $\lambda$, and $\mathrm{Card}^N=\mathrm{Card}\cap N$. (As discussed there, the iterability of $N$ can be expressed in an appropriate manner.) We want to see that $N\ins M$. So let $(\Tt,\Uu)$ be the comparison of $(M|\lambda^{++},N)$,
with both trees at degree $0$.
We want to see that $\Tt,\Uu$ are both trivial. Suppose not.
Then easily, $\Tt$ cannot be trivial.
Let $\gamma=\lh(E^\Tt_0)$ or $\gamma=\lh(E^\Uu_0)$,
for whichever of these extenders is non-empty,
and let $\eta=\card^M(\gamma)$.
Then $\Uu$ is above $\eta$,  like in the proof of Claim~\ref{clm:N_good_analysis_2}.
But then $\Uu$ drops along $b^\Uu$ (since $\gamma<\eta^{+M}=\eta^{+N}$). So $b^\Tt$ does not drop, and $M^\Tt_\infty\pins M^\Uu_\infty$.
Standard weasel arguments now show that $\lh(\Tt,\Uu)=\lambda^{++}+1$ and $\OR(M^\Tt_{\lambda^{++}})=\lambda^{++}<\OR(M^\Uu_{\lambda^{++}})$,
and for club many  $\beta<\lambda^{++}$, we have  $\beta<^\Uu\lambda^{++}$ and $\beta=\crit(i^\Uu_{\beta\lambda^{++}})$
and $i^\Uu_{\beta\lambda^{++}}(\beta)=\lambda^{++}$.
But then $M^\Tt_{\lambda^{++}}$
has largest cardinal $i^\Tt_{0\lambda^{++}}(\lambda^{+})$, whereas in $M^\Uu_{\lambda^{++}}$,
$\lambda^{++}$ is inaccessible, a contradiction.
\end{rem}

In case $\kappa$ is a limit cardinal of $M$,
the $\Sigma_1^{\univ{M}}(\her_\kappa^M\cup\{\kappa\})$-good wellorder of $(\her_{\kappa^+})^M$
defined in the proof of Theorem~\ref{tm:good_wo_1-small} is not in general the same as the usual order of constructibility. In fact, if $M|\kappa$
is closed under sharps, then they disagree,
because the set $M|\kappa$ itself is positioned strictly above
sets of higher $V$-rank which appear in the dropping iterates of $M|\kappa$ which arise in the proof.
So one can ask whether the usual order of constructibility could be $\Sigma_1^{\univ{M}}(\her^M_\kappa\cup\{\kappa\})$-good. If it is,
then clearly $M|\kappa$ is $\Delta_1^{\univ{M}}(\her_\kappa^M\cup\{\kappa\})$. But we show below that if $M=M_1$
and $\kappa$ is $\om_1$-iterable in $M_1$,
then this is not the case:

\begin{tm}\label{tm:om_1-it}
Let $\kappa\in M_1$ be such that $M_1\sats$ ``$\kappa$ is an uncountable cardinal
which is not a limit of measurables''.
Then:
\begin{enumerate}
 \item\label{item:non-Mahlo} If $M_1\sats$ ``$\kappa$ is not  Mahlo'' then  $M_1|\kappa$
is $\Delta_1(\her_\kappa^{M_1}\cup\{\kappa\})$.
\item\label{item:omega_1-it} If $M_1\sats$ ``$\kappa$ is $\om_1$-iterable'' then $M_1|\kappa$ is not $\Delta_1(\her_\kappa^{M_1}\cup\{\kappa\})$.
\end{enumerate}
\end{tm}

\begin{proof}
Part~\ref{item:non-Mahlo}:  Work in $M_1$. If $\kappa$ is a successor cardinal use the proof of Theorem~\ref{tm:good_wo_1-small}.
Suppose $\kappa$ is a limit cardinal.
Fix an $M_1$-cardinal $\eta<\kappa$
as in the proof of Theorem~\ref{tm:good_wo_1-small}.

Suppose  $\kappa$ is singular.
Then $M_1|\kappa$
is the unique premouse $P$
such that $M_1|\eta\pins P$, $\OR^P=\kappa$,
$P\sats$ ``there are no measurable cardinals $\geq\eta$'',
there is a  sound premouse $Q$
such that $P\pins Q$,
$\rho_\om^Q=\kappa$,
$Q\sats$ ``$\kappa$ is singular'', and
there is $\xi$ such that $L_\xi[Q]\sats$ ``$Q$ is strongly iterable''.
(Compare $M_1$ with $Q$,
and note that because $\kappa$ is singular in $Q$, we can't have that the $M_1$ side drops on its main branch.)

Now suppose $\kappa$ is inaccessible, but
 not Mahlo. So there is a club $C\sub\kappa$
 such that $\beta$ is singular for all $\beta\in C$. So there is a sequence $\vec{f}=\left<f_\beta\right>_{\beta\in C}$ of witnessing singularizing functions.
 We can have our $\Delta_1$ description be much as before, but instead of requiring $Q\sats$ ``$\kappa$ is singular'', we require
 $Q\sats$ ``$\kappa$ is not Mahlo, as witnessed as by some $(C,\vec{f})$''.
 Now argue much like before.

 Part~\ref{item:omega_1-it}: Suppose $\kappa$ is $\om_1$-iterable but $M_1|\kappa$ is $\Delta_1(\her_\kappa^{M_1}\cup\{\kappa\})$.
 So $\{M_1|\kappa\}$
 is $\Sigma_1(\her_\kappa^{M_1}\cup\{\kappa\})$.
 Let $\varphi$ be $\Sigma_1$ and $p\in M_1|\kappa$ be such that
 $M_1|\kappa$ is the unique $X\in M_1$ such that $M_1\sats\varphi(p,\kappa,X)$.
 So there is some $\alpha<\kappa^{+M_1}$
 such that $M_1|\alpha\sats\varphi(p,\kappa,X)$.
 Let $\beta\in(\alpha,\kappa^{+M_1})$
 be such that $M_1|\beta\sats$ ZFC$^-+$``$\kappa$ is the largest cardinal''
 and $U$ be some $\om_1$-iterable weakly amenable ultrafilter over $M_1|\beta$. Let \begin{equation*}\label{eqn:label_for_eqn_29}\pi:(\bar{M},\bar{U})\to(M_1|\beta,U) \end{equation*}
 be elementary with $\bar{\kappa}=\card(\bar{M})<\kappa$
 and $\pi(\bar{\kappa})=\kappa$
 and $p\in M_1|\bar{\kappa}$.
 Then $(\bar{M},\bar{U})\sats\varphi(p,\bar{\kappa},M_1|\bar{\kappa})$,
 $\bar{U}$ is weakly amenable to $\bar{M}$
 and $(\bar{M},\bar{U})$ is $\om_1$-iterable.
 Let $(\bar{M}',\bar{U}')$
 be the $\kappa$th iterate
 and $i:(\bar{M},\bar{U})\to(\bar{M}',\bar{U}')$
 the iteration map.
 Then $i$ is elementary as a map $\bar{M}\to\bar{M}'$ and $i(\bar{\kappa})=\kappa$, so $\bar{M}'\sats\varphi(p,\kappa,\bar{M}'|\kappa)$,
 so in fact $\varphi(p,\kappa,\bar{M}'|\kappa)$
 holds, but $\bar{M}'|\kappa\neq M_1|\kappa$,
 a contradiction.
\end{proof}

\section{Weak Compactness and the Perfect Set Property}\label{sec:wc}

We now answer \cite[Question 10.1]{Sigma_1_def_at_higher_cardinals}; the answer is ``yes'':

\begin{tm}\label{tm:wc}
 Assume ZFC and let $\kappa$ be a weakly compact cardinal such that for every $\Sigma_1(\her_\kappa\cup\{\kappa\})$ set $D\sub\pow(\kappa)$ of cardinality $>\kappa$,
 there is a perfect embedding $\iota:{^\kappa}\kappa\to\pow(\kappa)$ with $\rg(\iota)\sub D$. Then there is a
 proper class inner model satisfying ``there is a weakly compact limit of measurable cardinals''.
\end{tm}
\begin{proof}
We may assume that there is no proper class inner model satisfying ``there is a strong cardinal'', so the core model $K$ below $0^\pistol$ exists. We may further assume that $K\sats$ ``there is no weakly compact limit of measurable cardinals'', and in particular ``no measurable limit of measurable cardinals''.

\begin{clmthree}\label{clm:kappa_is_wc}
$\kappa$ is  weakly compact in $K$.\end{clmthree}
\begin{proof}
Let $T\in K$ be a tree on $\kappa$ of the relevant form. Using the weak compactness of $\kappa$ in $V$,
let $j:M\to N$ be elementary,
where $M,N$ are transitive, $V_\kappa\cup\{\kappa,T\}\in M$, for some $\alpha<\kappa^{+K}$,
we have $T\in K|\alpha\in M$,
$M\sats\ZFC^-$+``$K|\alpha$ is iterable'' and
${^{<\kappa}}M\sub M$,
with $\crit(j)=\kappa$. We may assume that $N=\Ult_0(M,U)$
where $U$ is the normal measure derived from $j$,
and therefore ${^{<\kappa}}N\sub N$. Let $b\in j(T)$ be a branch at level $\kappa$ of $j(T)$. Then $b\sub T$, so it suffices to see that $b\in K$. But $b\sub\kappa$
and $b\in j(K|\kappa)$,
so $b\in j(K|\kappa)|\beta$
for some $\beta<\kappa^{+j(K|\kappa)}$
such that $j(K|\kappa)|\beta$ projects to $\kappa$. And $N\sats$ ``$j(K|\kappa)$ is iterable'',
so $j(K|\kappa)$ really is iterable, since ${^{<\kappa}}N\sub N$. But then considering that $K$ is below a measurable limit of measurables, it follows that $j(K|\kappa)|\beta\pins K$, which suffices.
\end{proof}

By the claim, it suffices to see that $\kappa$ is limit of measurables in $K$. So  suppose
we can fix $\eta<\kappa$
such that all measurables of $K$ below $\kappa$
are $<\eta$, and $\eta=\theta^{++K}$ for some $\theta$. Since $K$ is below a measurable limit of measurables, $\eta$ is a cutpoint of $K$.

 \begin{clmthree}\label{clm:kappa^+K=kappa^+_wc}$\kappa^{+K}=\kappa^+$.
 \end{clmthree}
\begin{proof}
Suppose not.
Then by the weak compactness of $\kappa$,
we can fix some transitive set $M\sats\ZFC_n$ (for some large $n$),
of cardinality $\kappa$,
with $\pow(\kappa)\cap K\sub M$
and ${^\om}M\sub M$,
another transitive $N$,
and a $\Sigma_n$-elementary $j:M\to N$
with $\crit(j)=\kappa$,
and such that if $\kappa$ is measurable in $K$,
then the (unique) normal
measure $D$ on $\kappa$
in $K$ is also in $M$,
and in fact $D$ is coded by a subset of $\kappa$ in $M$,
and hence $D\in N$ also in this case;
we may also assume that $N=\Ult_0(M,U)$
where $U$ is the normal $M$-ultrafilter over $\kappa$ derived from $j$, so then ${^\om}N\sub N$.

Let $U$ be the normal measure over $K$ derived from $j$.
Let $K'=\Ult(K,U\cap K)$.
Then since ${^\om}M\sub M$,
$K'$ is wellfounded.
(If  $\left<f_n\right>_{n<\om}\sub K$
with $f_n:\kappa\to\OR$
and $A_n=\{\alpha<\kappa\bigm|f_{n+1}(\alpha)<f_n(\alpha)\}\in U$
for all $n<\om$,
then since $\vec{A}=\left<A_n\right>_{n<\om}\in M$, we get $j(\vec{A})\in N$,
but $\mu\in j(A_n)$ for each $n$,
so $\bigcap j(\vec{A})\neq\emptyset$,
so $\bigcap \vec{A}\neq\emptyset$,
which gives a contradiction.)

Since $K'$ is proper class and by our smallness assumptions, it follows that $K'$ is fully iterable. And like in \cite[\S3]{cmip}, $K'$ is universal.
Let $\ell:K\to K'$ be the ultrapower map.
By \cite[Theorem 8.13]{cmip},
$K'$ is a normal iterate of $K$
and $\ell$ is the iteration map.
Note then that since $\kappa$ is not a limit of measurables in $K$, $U\cap K=D\in K$ is the unique normal measure on $\kappa$ in $K$, and $\ell=i^K_D$ is the ultrapower map. We have $D\in M\cap N$.
But $\pow(\kappa)\cap K=\pow(\kappa)\cap K^N$, like in the proof of Claim~\ref{clm:kappa_is_wc}.
Note that $K''=\Ult(K^N,D)$ is wellfounded,
by the same reasoning for the wellfoundedness of $K'$. Because $n$ is large enough, \cite[Theorem 8.13]{cmip} applies in $N$ to the ultrapower map $K^N\to K''$,
and so it follows that $D\in K^N$,
so $\kappa$ is measurable in $K^N$,
so $\kappa$ is a limit of measurables of $K$, contradiction.
\end{proof}

Recall  $\eta=\theta^{++K}$ bounds the measurables of $K$ below $\kappa$.
Now let $X$ be the set of subsets of $\kappa$ coding a pair $(M,T)$,
where $M$ is an iterable premouse of height $\kappa$ such that $K|\eta\pins M$,
$M$ has no measurable cardinals $\geq\eta$, there is a sound premouse $P$
such that $M\pins P$, $\rho_\om^P=\kappa$, $P$ satisfies the condensation theorem (see \cite{outline}) and $T$ is the theory of $P$ in parameters in $\kappa$.
Note that $X$ is $\Sigma_1(\{K|\eta,\kappa\})$,
and $\card(X)>\kappa$,
since $\kappa^{+K}=\kappa^+$
and for cofinally many proper segments $P\pins K|\kappa^{+K}$,
we have $(K|\kappa,T_P)\in X$
where $T_P=\Th^P(\kappa)$.
So by the theorem's hypothesis, we can fix a perfect function $\iota:{^\kappa}\kappa\to\pow(\kappa)$
with $\rg(\iota)\sub X$.

\begin{clmthree}
There are at most $\kappa$-many elements of $X$
of form $(M,T)$ for some $M\neq K|\kappa$.\end{clmthree}
\begin{proof}This is much like in the proof of Theorem~\ref{tm:good_wo_1-small}, in its Case~\ref{case:kappa_limit_card}. But here is the outline. Fix such a pair $(M,T)$.
The comparison of $K|\kappa$ versus $M$, producing iterations $\Tt$ on $K|\kappa$ and $\Uu$ on $M$,
is such that $b^\Uu$ does not drop,
but $b^\Tt$ drops and $M^\Uu_\infty\pins M^\Tt_\infty$.
(Otherwise we would have that neither $b^\Tt$ nor $b^\Uu$ drops,
and $M^\Tt_\infty=M^\Uu_\infty$,
but then both sides are trivial, because $K$ has no measurables $\geq\eta$, and so $M=K|\kappa$,
contradiction.)
So in fact, $\Uu$ is trivial, since $M$ has no measurables $\geq\eta$, and $K|\eta=M|\eta$.
It follows that after some stage $\alpha<\kappa$,
$\Tt$ simply consists of linearly iterating some fixed measure out to $\kappa$, producing $M=M^\Tt_\kappa|\kappa$. So there are at most $\kappa$-many such $M$.
 \end{proof}

 So by diagonalizing if necessary, we may assume that all elements of $\rg(\iota)$ are of form $(K|\kappa,T)$. Note that also, letting
 $(M,T)\in X$ with $M=K|\kappa$
 and $P$ the model with theory $T$,
 we have $P\pins K$, by Jensen's condensation argument \cite[Fact 3.1]{V=HODX_pub}. We now get a contradiction like in the proof of Theorem~\ref{tm:L[U]_gdst} part~\ref{item:no_perf_subset},
 in its Case~\ref{case:mu^+<kappa_and_kappa_mu-steady},
 in the subcase there that $\theta>\om$.
 That is, let $G$ be $(V,\PP)$-generic where $\PP={^{<\kappa}}\kappa$.
 Let $(M_G,T_G)$ be the pair coded by $\bigcup \iota``G$. Then $M_G=K|\kappa$,
 since we thinned out $\rg(\iota)^V$ in this way. And like in the proof of Theorem~\ref{tm:L[U]_gdst}, $T_G$ is the theory of a premouse $P\notin V$, which satisfies condensation and has $\rho_\om^P=\kappa$. But $\kappa$ is regular in $V[G]$
 and $\kappa^{+V[G]}=\kappa^+$,
 so by Jensen's condensation argument,
 $P\pins  K$, so $P\in V$,
 contradiction.
\end{proof}

\section{Almost Disjoint Families}
\label{sec:adf}

We summarized our results on almost disjoint families in \S\ref{sec:intro}.

\begin{tm}\label{tm:kappa_limit_of_measurables_implies_no_Sigma_1_mad_family}
 Assume ZFC and let $\kappa$ be an uncountable cardinal.
 Then:
 \begin{enumerate}
  \item\label{item:cof_om_meas_lim_adf} If $\cof(\kappa)=\om$  then there is a $\Delta_1(\{\kappa\})$ almost disjoint family
  $\mathscr{F}\sub\pow(\kappa)$
 of cardinality $>\kappa$,
 such that for all $A\in\mathscr{F}$,
 $A$ is unbounded in $\kappa$
 and has ordertype $\om$.
 \item\label{item:kappa_a_limit_of_measurables} Suppose $\kappa$ is a limit of measurable cardinals.
 Then:
 \begin{enumerate}
  \itemref[item:no_madf_limit_of_meas]{\tu{(}a\tu{)}} There is no $\Sigma_1(V_\kappa\cup\OR)$ mad family $\mathscr{F}\sub\pow(\kappa)$ of cardinality $\geq\kappa$.
 \itemref[item:cof>om_meas_lim_no_adf]{\tu{(}b\tu{)}} If $\cof(\kappa)>\om$
 then there is no $\Sigma_1(V_\kappa\cup\OR)$
 almost disjoint family $\mathscr{F}\sub\pow(\kappa)$ of cardinality $>\kappa$.
 \end{enumerate}
 \item\label{item:no_madf_above_1_meas}Let $\mu<\kappa$
 and suppose that $\mu$ is measurable and $\kappa$ is $\mu$-steady.
 Let $U$ be a normal measure on $\mu$.
 Let $S=\{\alpha\in\OR\bigm|i^{\Tt_U}_{0\lambda}(\alpha)=\alpha\text{ for all }\lambda<\kappa\}$
 \tu{(}so $\kappa,\kappa^+\in S$\tu{)}.
 Then there is no $\Sigma_1(V_\mu\cup S)$ mad family $\mathscr{F}$  at $\kappa$
  of cardinality $>2^\mu$.
  \end{enumerate}
\end{tm}
\begin{proof}
 Part~\ref{item:cof_om_meas_lim_adf}
 is routine: just let $A\in\mathscr{F}$ iff $A\sub\kappa$,
 $A$ has ordertype $\omega$, $A$ is unbounded in $\kappa$ and for every $\alpha\in A$, $\alpha$ is the G\"odel code of a pair $(\beta,\gamma)$ such that $\beta$ is the G\"odel code of $A\cap\alpha$.
 This easily works.

 Part~\ref{item:kappa_a_limit_of_measurables}\ref{item:cof>om_meas_lim_no_adf}: Suppose otherwise and fix $\vec{p}\in V_\kappa\cup\OR$ and some
 $\Sigma_1(\{\vec{p}\})$ almost disjoint family $\mathscr{F}\sub\pow(\kappa)$ of cardinality $>\kappa$.
 Let $\eta=\cof(\kappa)$
 and let $\left<\kappa_\alpha\right>_{\alpha<\eta}$ be a strictly increasing sequence of measurable cardinals cofinal in $\kappa$,
 with $\eta<\kappa_0$ if $\eta<\kappa$,
 and $\vec{p}$ fixed by all iterations of measures which are based on $V_\delta$ for some $\delta<\kappa$, are above $\kappa_0$, and have length $<\kappa$. (Such a $\kappa_0$ exists by Lemma~\ref{lem:ordinals_ev_stable_meas_limit}.)
 Let $\left<U_\alpha\right>_{\alpha<\eta}$ a sequence with $U_\alpha$ a normal measure on $\kappa_\alpha$.
 Since $\mathscr{F}$ has cardinality $>\kappa$, we can fix $X\in\mathscr{F}$ such that
 $i^{\Tt_{U_\alpha}}_{0\kappa}(X\cap\kappa_\alpha)\neq X$  for all $\alpha<\eta$. By thinning out the sequence, we may assume
 that
 $i^{\Tt_{U_\alpha}}_{0\kappa_{\alpha+1}}(X\cap \kappa_\alpha)\neq X\cap\kappa_{\alpha+1}$ for every $\alpha<\eta$. Let $\lambda=\sup_{n<\om}\kappa_n$; so $\lambda<\kappa$. For $\xi\in[\lambda,\kappa)$  let $n_\xi$ be the least $n<\om$
 such that $i^{\Tt_{U_n}}_{0\kappa_{n+1}}(\xi)=\xi$;
 again by Lemma~\ref{lem:ordinals_ev_stable_meas_limit},
 such an $n$ exists. Then since $\cof(\kappa)>\om$
 and $X$ is unbounded in $\kappa$,
 there is $n<\om$ and an unbounded set $Y\sub X$
 such that $n_\xi=n$ for all $\xi\in Y$.
 Fix such $n,Y$.
 Let $j=i^{\Tt_{U_n}}_{0\kappa_{n+1}}$.
 Then $j(X)\in\mathscr{F}$ since $j$ fixes $\vec{p}$. And
 $j(X)\neq X$ by construction.
 But $Y\sub X\cap j(X)$,
 and since $Y$ is unbounded in $\kappa$, this is a contradiction.

 Part~\ref{item:no_madf_above_1_meas}:
 Suppose otherwise.
 Fix $\vec{p}\in V_\mu\cup S$
 and a $\Sigma_1(\{\vec{p}\})$
 mad family $\mathscr{F}$ at $\kappa$ of cardinality $>2^\mu$.

We have $i^{\Tt_U}_{0\lambda}(\mathscr{F})=\mathscr{F}\cap M^{\Tt_U}_{\lambda}$ for all $\lambda<\kappa$. For $i^{\Tt_U}_{0\lambda}(\mathscr{F})\sub\mathscr{F}$
since
 $i^{\Tt_U}_{0\lambda}(\vec{p})=\vec{p}$
 and by $\Sigma_1$ upward absoluteness, and conversely,
 if $X\in \mathscr{F}\cap M^{\Tt_U}_{\lambda}$,
 then by the maximality of $\mathscr{F}$ and elementarity,
 there is $Y\in i^{\Tt_U}_{0\lambda}(\mathscr{F})$
 such that $Y\cap X$ is unbounded in $\kappa$,
 but as just mentioned, it follows that $Y\in\mathscr{F}$
 also,
 and therefore $Y=X$.

For each $\xi<\kappa$ let $\mu_\xi=i^{\Tt_U}_{0\xi}(\mu)$.
 For $\alpha<\kappa$ let $\gamma_\alpha$ be the least $\gamma$ such that $\mu_\gamma>\alpha$, and choose some $n_\alpha<\om$ and $f_\alpha:[\mu]^{n_\alpha}\to\mu$
 such that for some $c=\{c_0<\ldots<c_{n_\alpha-1}\}\in[\gamma_\alpha]^{<\om}$,
 we have $\alpha=i^{\Tt_U}_{0\gamma_\alpha}(f_\alpha)(b_{\alpha})$,
 where $b_{\alpha}=\{\mu_{c_0},\ldots,\mu_{c_{n_\alpha-1}}\}$.
 Note then that $\alpha=i^{\Tt_U}_{0\gamma}(f_\alpha)(b_\alpha)$
 for all $\gamma\in[\gamma_\alpha,\kappa]$.

 Let us say that $\mu'<\kappa$
is \emph{nicely stable}
iff it satisfies the properties of Lemma~\ref{lem:many_U-it_fixed_points_above_meas} part~\ref{item:nicely_stable};
that is, $i^{\Tt_U}_{0\mu'}(\mu)=\mu'$ and $i^{\Tt_U}_{0\lambda}(\mu')=\mu'$ for all $\lambda<\mu'$. So the nicely stable ordinals are unbounded in $\kappa$.

 Now fix
  $X\in\mathscr{F}$  such that $X\neq i^{\Tt_U}_{0\kappa}(X\cap\mu)$; such an $X$ exists because $\mathscr{F}$ has cardinality $>2^\mu$.
 Let $\eta=\cof(\kappa)$
 and fix a normal function $g:\eta\to\kappa$ with $\mu\leq g(0)$
 (but $\kappa$ might be regular).
 Define a sequence $\left<\alpha_\beta,A_\beta\right>_{\beta<\eta}$ as follows:
Let $\alpha_0$
be the least $\alpha\in X$
such that $\alpha>g(0)$.
Let $A_0$ be the set of the least $(n_{\alpha_0}+1)$-many nicely stable ordinals which are $>\gamma_{\alpha_0}$. Given $\left<\alpha_\beta,A_\beta\right>_{\beta<\lambda}$ where $\lambda<\eta$,
let $\alpha_\lambda$ be the least $\alpha\in X$ such that $\alpha>g(\lambda)$ and $\alpha>\bigcup_{\beta<\lambda}A_\beta$, and let $A_\lambda$
be the set of the least $(n_{\alpha_\lambda}+1)$-many nicely stable ordinals which are $>\gamma_{\alpha_\lambda}$.

Now define, for each $\beta<\eta$, \begin{equation*}\label{eqn:label_for_eqn_30}\alpha'_\beta=i^{\Tt_U}_{0\max(A_\beta)}(f_{\alpha_\beta})(A_\beta\cut\{\max(A_\beta)\}),\end{equation*}
so also
\begin{equation*}\label{eqn:label_for_eqn_31} \alpha'_\beta=i^{\Tt_U}_{0\kappa}(f_{\alpha_\beta})(A_\beta\cut\{\max(A_\beta)\}).\end{equation*}
 Then $\alpha_\beta\leq\alpha'_\beta<\max(A_\beta)$ for all $\beta<\eta$. Let  $Y=\{\alpha'_\beta\bigm|\beta<\eta\}$. So $Y$ is cofinal in $\kappa$, so  by the maximality of $\mathscr{F}$, we can fix  $Z\in\mathscr{F}$ with $Y\cap Z$ cofinal in $\kappa$.

We claim that $Z\neq i^{\Tt_U}_{0\kappa}(Z\cap\mu)$. For suppose otherwise. Since $X\neq i^{\Tt_U}_{0\kappa}(X\cap\mu)$ but $X,Z\in\mathscr{F}$,
$X\cap Z$ is bounded in $\kappa$.
So we can fix $\beta$
such that $\alpha_\beta\notin Z$
but $\alpha_\beta'\in Z$.
(That is, for all sufficiently large $\beta$, we have $\alpha_\beta\notin Z$, since $\alpha_\beta\in X$. But $Y\cap Z$ is cofinal in $\kappa$, so there is such a $\beta$
with $\alpha_\beta'\in Z$.)
But because $Z=i^{\Tt_U}_{0\kappa}(Z\cap\mu)$
and $\alpha_\beta=i^{\Tt_U}_{0\kappa}(f_{\alpha_\beta})(b_{\alpha_\beta})$
and $\alpha_\beta'=i^{\Tt_U}_{0\kappa}(f_{\alpha_\beta})(A_\beta\cut\{\max(A_\beta)\})$,
in fact $\alpha_\beta\in Z$ iff $\alpha_\beta'\in Z$,
a contradiction.

So there is $\lambda<\kappa$
such that $W=i^{\Tt_U}_{0\lambda}(Z)\neq Z$, but  $W\in\mathscr{F}$. So $W\cap Z$ is bounded in $\kappa$.
So we can find $\beta<\eta$ such that $\alpha_\beta'\in Z\cut W$
and  $\lambda<\min(A_\beta)$.
But every ordinal in $A_\beta$ is nicely stable, so $i^{\Tt_U}_{0\lambda}(A_\beta)=A_\beta$,
from which it easily follows
that $i^{\Tt_U}_{0\lambda}(\alpha_\beta')=\alpha_\beta'$,
so $\alpha_\beta'\in W$, a contradiction.

Part~\ref{item:kappa_a_limit_of_measurables}\ref{item:no_madf_limit_of_meas}:
This is an easy corollary of part~\ref{item:no_madf_above_1_meas}
and Lemma~\ref{lem:ordinals_ev_stable_meas_limit}.
\end{proof}

Let $M$ be a premouse.
If $M\sats$ ``$\om_1$ exists''
then $\HC^M$
is the universe of $M|\om_1^M$,
but $M|\om_1^M$ itself (which has $\es^M\rest\om_1^M$ as a predicate) need not be definable from parameters over $\HC^M$.
In the following theorem,
we might have $\kappa=\om_1^M$,
in which case $\mathfrak{m}=M|\om_1^M\in\her^M_{\om_2^M}\cut\her^M_{\om_1^M}$.
But if $\om_1^M<\kappa$
then $\mathfrak{m}\in\her_\kappa^M$.

Recall that a premouse $M$ is \emph{non-tame}
if there is $E\in\es^M_+$
and some $\delta\in[\crit(E),\nu(E)]$ such that $M|\lh(E)\sats$ ``$\delta$ is Woodin'';
otherwise $M$ is \emph{tame}.
Tame mice can satisfy ``there is a strong cardinal which is a limit of Woodin cardinals'' and more. They cannot satisfy ``there is a Woodin limit of Woodin cardinals''.

\begin{tm}\label{tm:regular_kappa_Pi_1_kappa_mad_families}
Let $M$ be a $(0,\om_1+1)$-iterable premouse, and let $\kappa\in\OR^M$
be such that $M\sats$ ``$\kappa>\om$ is a regular cardinal and $\kappa^+$ exists''.
Let $\mathfrak{m}=M|\om_1^M$.
Then:
\begin{enumerate}
 \item \label{item:mouse_Pi_1_mad_family}
$M\sats$ ``there is a mad family $\mathscr{F}\sub\pow(\kappa)$ of cardinality $\kappa^+$ which is $\Pi_1(\{\kappa,\mathfrak{m}\})$''.
\item\label{item:tame_Pi_1_mad_family} If $M$ is tame then $M\sats$ ``there is a mad family $\mathscr{F}\sub\pow(\kappa)$
of cardinality $\kappa^+$ which is $\Pi_1(\{\kappa,x\})$
for some $x\in\RR$''.
\item\label{item:M_n,wlw} If $M=M_n$ for some $n<\om$ \tu{(}the minimal proper class mouse with $n$ Woodin cardinals\tu{)}
or $M$ is the minimal proper class mouse satisfying ``there is a Woodin limit of Woodin cardinals'' then $M\sats$ ``there is a mad family $\mathscr{F}\sub\pow(\kappa)$ of cardinality $\kappa^+$ which is $\Pi_1(\{\kappa\})$''.
\end{enumerate}
\end{tm}
\begin{rem}\label{rem:7.3}
There are many other examples
of ``minimal''
mice $M$ for which the same conclusion as in part~\ref{item:M_n,wlw} holds.\end{rem}
\begin{proof}
Part~\ref{item:mouse_Pi_1_mad_family}:
 By \cite[Theorem 3.11]{V=HODX_pub},
 $M|\kappa$
 is definable over $\her_\kappa^M$
 from the parameter $\mathfrak{m}$ (or predicate $\mathfrak{m}$, in case $\kappa=\om_1^M$,
 but in this case it is trivial).
 And working in $M$,
 where $\kappa$ is regular,
 $M|\kappa^+$ is simply definable from the parameter $M|\kappa$,
 as the Jensen stack over $M|\kappa$. (See \cite[Fact 3.1]{V=HODX_pub} and/or Subcase \ref{scase:cutpoint_eta_above_lgcd} of the proof of Theorem \ref{tm:good_wo_1-small}.)

 Now work in $M$.
 We will define a mad family $\mathscr{F}\sub\pow(\kappa)$ as desired. We will use the standard $M$-constructibility order
 to guide a recursive construction to build $\mathscr{F}$. Let us sketch how we will ensure that this will be $\Pi_1(\{\kappa,\mathfrak{m}\})$-definable.
 Given a set $X$ of ordinals,
 and an ordinal $\alpha$,
 let $\alpha\oplus X=\{\alpha+\beta\bigm|\beta\in X\}$.  For sets $X,A\sub\OR$
 and $\alpha\in\OR$, say that $X$ is \emph{directly encoded into $A$ at $\alpha$}
 if there is $\beta\in\OR$
  such that $\alpha+X=[\alpha,\beta)\cap A$; say that $X$ is \emph{directly encoded into $A$}
  if there is some such $\alpha$. The key to the construction will be to ensure that for each $A\in\mathscr{F}$, every bounded set $X\sub\kappa$ is directly encoded into $A$.
 Then from any $A\in\mathscr{F}$,
 we can easily recover $\her_\kappa^M$, and so by the previous paragraph, recover $M|\kappa$ and (proper segments of) $M|\kappa^{+}$, and hence the recursive construction used to build $\mathscr{F}$.

 We now proceed to the detailed construction of $\mathscr{F}$.
 The same construction will also be used in parts\nolinebreak[3] \ref{item:tame_Pi_1_mad_family}~and~\ref{item:M_n,wlw}.
 Fix a surjection
 $\pi:\kappa\to\pow({<\kappa})$
 which is definable over $\her_\kappa^M$ from  the parameter $\mathfrak{m}$
 and such that $\pi(\alpha)\sub\alpha$ for each $\alpha<\kappa$.
 We will define a sequence $\left<A_\alpha\right>_{\alpha<\kappa^{+}}$
 by recursion on $\alpha$,
 and set $\mathscr{F}=\{A_\alpha\}_{\alpha<\kappa^+}$.  We will also define a sequence $\left<C_\alpha\right>_{\alpha<\kappa^+}$ of clubs $C_\alpha\sub\kappa$.

 Let us start with $A_0$ and $C_0$.
 We define $A_0$ through a $\kappa$-sequence of stages $\alpha<\kappa$, at stage $\alpha$ extending with a certain empty interval,  followed by a direct encoding of $\pi(\alpha)$.
 Recall that for a set $X$ of ordinals, the \emph{strict supremum} $\strsup(X)$ of $X$ is the least ordinal $\eta$ such that $X\sub\eta$.
 Define a sequence $\left<B_\alpha\right>_{\alpha<\kappa}$ of sets as follows.
 Set $B_0=\emptyset$.
 Given $B_\alpha$,
 let $B_{\alpha+1}=B_\alpha\cup((\beta+\beta)\oplus\pi(\alpha))$ where $\beta=\strsup(B_\alpha)$.
 (We only really need ``$\beta+\beta$'' instead of just ``$\beta$'' in case $\alpha$ is a limit ordinal with $\alpha=\strsup(B_\alpha)=\sup(B_\alpha)$.) Given $B_\alpha$ for $\alpha<\eta$, where $\eta$ is a limit, let $B_\eta=\bigcup_{\alpha<\eta}B_\alpha$.  Now set $A_0=\bigcup_{\alpha<\kappa}B_\alpha$.
 Also let $C_0\sub\kappa$
 be the club of all limit ordinals $\eta<\kappa$
 such that $\eta=\sup B_\eta$.

 Suppose we have defined $\left<A_\alpha,C_\alpha\right>_{\alpha<\lambda}$ where $0<\lambda<\kappa^{+M}$,
and the following conditions hold,
for all $\alpha<\lambda$:
 \begin{enumerate}
  \itemref[item:A_alpha_unbd]{(i)} $A_\alpha$ is unbounded in $\kappa$,
  \itemref[item:all_bounded_sets_coded]{(ii)} every bounded subset of $\kappa$ is directly encoded into $A_\alpha$,
  \itemref[item:intersections_bounded]{(iii)} $A_\beta\cap A_\alpha$ is bounded in $\kappa$ for all  $\beta<\alpha$,
  \itemref[item:club_C_alpha]{(iv)}  $C_\alpha\sub\kappa$ is club and
  $[\eta,\eta+\eta)\cap A_\alpha=\emptyset$
   for all $\eta\in C_\alpha$.
 \end{enumerate}

 Now let $B\sub\kappa$
 be $<^M$-least
 such that $B$ is unbounded in $\kappa$ but there is no $\alpha<\lambda$ such that $B\cap A_\alpha$ is unbounded in $\kappa$. (Such a set $B$ exists. For $\kappa\cut A_\alpha$ is unbounded in $\kappa$ for each $\alpha<\kappa$
 (say by clause~\ref{item:all_bounded_sets_coded}), so  since $\kappa$ is regular,
 a routine diagonalization produces such a $B$.)

 We will define $A_\lambda$ and $C_\lambda$,
 ensuring that $B\cap A_\lambda$ is unbounded in $\kappa$,
 and maintaining the conditions~\ref{item:A_alpha_unbd}--\ref{item:club_C_alpha},
 and ensure also a bit more.

 Fix the $<^M$-least surjection $\sigma:\kappa\to\lambda$.
 We construct $A_\lambda$
 in $\kappa$-many stages $\alpha<\kappa$. At successor stages $\alpha+1$, we find a suitable point at which to directly encode $\pi(\alpha)$, followed by adding some element of $B$, to what we have produced so far. At limit stages, we insert some space, followed by some element of $B$.

 So,
 we define an increasing sequence of sets $\left<a_\alpha\right>_{\alpha<\kappa}$;
 each $a_\alpha$ will be an initial segment of $A_\lambda$. Let $a_0=\emptyset$.
 Suppose we have defined $a_\alpha$
 where $\alpha<\kappa$.
 Let $C'_{\lambda\alpha}=\bigcap_{\beta<\alpha}C_{\sigma(\beta)}$;
 note that $C'_{\lambda\alpha}\sub\kappa$ is club. Let $\eta\in C'_{\lambda\alpha}$ be least such that $\eta\geq\max(\alpha,\strsup(a_\alpha))$. Then let $a'_{\alpha+1}=a_\alpha\cup(\eta\oplus\pi(\alpha))$.
 Let $\beta$
 be the least element of $B$ such that $\beta>\sup(a'_{\alpha+1})$,
 and let $a_{\alpha+1}=a'_{\alpha+1}\cup\{\beta\}$.

 Now suppose we have defined $a_\alpha$ for all $\alpha<\zeta$,
 where $\zeta<\kappa$ is a limit.
 Let $a'_\zeta=\bigcup_{\alpha<\zeta}a_\alpha$. Then set $a_\zeta=a'_\zeta\cup\{\beta\}$,
 where $\beta$ is least such that $\beta\in B$ and $\beta\geq\sup (a'_\zeta)+\zeta$.
 (Note the ``$+\zeta$''
 here introduces ``extra space'';
 this is needed to see that we get the club $C_\lambda$ with the desired properties.)
 Now set $A_\lambda=\bigcup_{\alpha<\kappa}a_\alpha$.
 Also set $C_\lambda$
 to be the set of all limit ordinals $\zeta$ such that $\zeta=\sup(a'_\zeta)$.

Clearly $A_\lambda\cap B$ is unbounded in $\kappa$, and properties
 \ref{item:A_alpha_unbd}, \ref{item:all_bounded_sets_coded}
 and \ref{item:club_C_alpha}
hold for all $\alpha\leq\lambda$.
For \ref{item:intersections_bounded},
fix $\beta<\lambda$; we verify that
$A_\beta\cap A_\lambda$ is bounded in $\kappa$. Well,
$A_\beta\cap B$ is bounded by choice of $B$. All elements of $A_\lambda\cut B$ are in $a'_{\alpha+1}\cut a_\alpha$ for some $\alpha<\kappa$. But for all sufficiently large $\alpha<\kappa$,
we have $\beta\in\sigma``\alpha$,
so $C'_{\lambda\alpha}\sub C_\beta$,
so letting $\eta$ be as in the definition of $a'_{\alpha+1}$, we have $\eta\in C_\beta$, so since $\pi(\alpha)\sub\alpha$,
\begin{equation*}\label{eqn:label_for_eqn_32} A_\beta\cap(a'_{\alpha+1}\cut a_\alpha)\sub A_\beta\cap[\eta,\eta+\eta)=\emptyset,\end{equation*}
which suffices.

Let $\mathscr{F}=\{A_\alpha\}_{\alpha<\kappa^+}$.
Then $\mathscr{F}$ is a mad family of cardinality $\kappa^+$.
(Suppose $\mathscr{F}$ is not maximal,
and let $B$ be the $<^M$-least counterexample. Let $\xi$ be position of $B$ with respect to $<^M$. Then $\xi<\kappa^{+}$, and note that there must be some $\lambda\leq\xi$ such that $B$ is the $<^M$-least counterexample to the maximality of $\{A_\alpha\}_{\alpha<\lambda}$. But then $A_\lambda\cap B$ is unbounded in $\kappa$, contradiction.)

We now want to see that $\mathscr{F}$ is $\Pi_1(\{\kappa,\mathfrak{m}\})$.
Well, given $A\sub\kappa$,
we have $A\in\mathscr{F}$ iff
the following two conditions hold:
\begin{enumerate}
 \item Every bounded subset of $\kappa$ is directly encoded into $A$.
 \item For all premice $P$ such that $\OR^P=\kappa$
 and $A\cap\alpha\in P$ for all $\alpha<\kappa$ and $\mathfrak{m}=P|\om_1^P$ and $P\sats$ ``I am the inductive condensation stack  over $\mathfrak{m}$
 (see \cite[Definition 3.12]{V=HODX_pub})'', and all sound premice $Q$ such that $P\pins Q$, $\rho_\om^Q=\kappa$, $Q\sats$ ``$\kappa^+$ exists'',
 and $Q$ satisfies the condensation theorem
 with respect to elementary $\pi:\bar{Q}\to Q$ with $\crit(\pi)=\rho_\om^{\bar{Q}}$ and $\pi(\crit(\pi))=\kappa$ (see \cite{outline}),
 if $A\in Q$ then $Q\sats$ ``there is $\alpha<\kappa^+$ such that $A=A_\alpha$'' (where we define $(\left<A_\alpha\right>_{\alpha<\kappa^{+Q}})^Q$ over $Q$
 from $\mathfrak{m},\kappa$
 just as $\left<A_\alpha\right>_{\alpha<\kappa^+}$ was defined over $M$ from $\mathfrak{m},\kappa$ above).
\end{enumerate}
This equivalence is straightforward to verify. Since these conditions are $\Pi_1(\{\mathfrak{m},\kappa\})$, this completes the proof of part~\ref{item:mouse_Pi_1_mad_family}.

Part~\ref{item:tame_Pi_1_mad_family}:
As $M$ is tame and by
\cite[Theorem 4.2]{odle_v2},
working in $M$,
there is $N\pins M|\om_1$ such that $\mathfrak{m}$
is the unique premouse $P$
with universe $\HC$,
such that $N\pins P$ and $P$ is above-$\OR^N$, $(\om,\om_1)$-iterable. Setting $x=N$, we get a $\Pi_1^{\univ{M}}(\{\kappa,x\})$ description of $\mathscr{F}$
(with $\mathscr{F}$ defined as above),
since for each $A\sub\kappa$,  we have $A\in\mathscr{F}$
iff the following conditions hold:
\begin{enumerate}
 \item Every bounded subset of $\kappa$ is directly encoded into $A$.
 \item For every premouse $P$
 with $\OR^P=\kappa$
 and $\{A\cap\alpha\bigm|\alpha<\kappa\}\sub P$
 and $N\pins P$,
 and every above-$N$, $(\om,\om_1^P)$-strategy for $P|\om_1^P$
 (using $\HC^P=\HC$
 to determine the countable iteration trees involved),
 if $P\sats$ ``I am the inductive condensation stack over $P|\om_1^P$'',
 then for all $Q$ with $P\pins Q$ and as in part~\ref{item:mouse_Pi_1_mad_family},
 if $A\in Q|\kappa^{+Q}$
 then $A=A_\alpha^Q$ for some $\alpha<\kappa^{+Q}$.
\end{enumerate}

Part~\ref{item:M_n,wlw}:
We just give a sketch here.
In case $M=M_n$, then $\mathfrak{m}$
is definable over $\HC=\HC^M$,
which, combined with the foregoing arguments, suffices. If $M$ is the minimal proper class mouse with a Woodin limit of Woodins,
then we use that $\mathfrak{m}$ is
$(\om,\om_1+1)$-iterable in $M$.\footnote{This is a standard fact, but is non-trivial. One can compute the relevant Q-structure for determining the correct $\Tt$-cofinal branch through an $\om$-maximal tree $\Tt$  on $\mathfrak{m}$ by a kind of background construction. Extenders $E$ overlapping $M(\Tt)$ can be added by deriving them from core embeddings, when a model $N_\alpha$ of the construction projects $<\delta(\Tt)$, setting $N_{\alpha+1}=(N_\alpha||\gamma,E)$ for the appropriate cardinal $\gamma$ of $N_\alpha$. One does this until reaching some $N_\alpha$ such that the $\delta(\Tt)$-core of $N_\alpha$ is the desired Q-structure for $M(\Tt)$ (this might also project $<\delta(\Tt)$).} This is enough to define $\mathfrak{m}$ over $\her_\kappa$
if $\kappa>\om_1$.
If $\kappa=\om_1$ we need to do some more work, because with $\Pi_1(\{\kappa\})$ we can then only directly assert that the premouse $P$ under consideration (with universe $\HC$) is $(\om,\om_1)$-iterable. We claim
that given this, it follows that $P=\mathfrak{m}$. The proof of this uses some of the ideas from \cite{odle_v2}, but here there is less going on, so we explain how it works in the present context. Suppose $P\neq\mathfrak{m}$, and let $\Gamma$ be an $(\om,\om_1)$-strategy for $P$,
and $\Sigma$ an $(\om,\om_1+1)$-strategy for $\mathfrak{m}$. Compare  the pair $(\mathfrak{m},P)$ using $(\Sigma,\Gamma)$, producing a pair $(\Tt,\Uu)$ of trees, each of length $\om_1$.
Let $b=\Sigma(\Tt)$. Let $Q\ins M^\Tt_b$ be the Q-structure for $\delta(\Tt)=\om_1$, so $Q$ determines $b$ in the usual manner. We claim that $Q$ similarly determines a $\Uu$-cofinal branch, which gives the usual contradiction. For let $\pi:W\to \her_{\om_2}$ be $\Sigma_{10}$-elementary, where $W$ is countable and transitive, with $(\Tt,\Uu),b\in\rg(\pi)$. Let $\pi(\bar{Q},\bar{b},\bar{\Tt},\bar{\Uu},\eta)=(Q,b,\Tt,\Uu,\om_1)$. Then $[0,\eta)^\Tt=\bar{b}=b\cap\eta$ is determined by $\bar{Q}$. Let $R=M(\Tt)=M(\Uu)$. Suppose first that $R\sats$ ``There is a proper class of Woodins''. Then by the minimality of $M$,
$Q=\J_\alpha(R)$ for some $\alpha\in\OR$, so $\bar{Q}=\J_{\bar{\alpha}}(R|\eta)$ for some $\bar{\alpha}$. It follows that $\bar{Q}\pins R$ and $\bar{Q}$ is a (hence the) Q-structure for $\eta$ in $R$, but then $\bar{Q}$ did determine $[0,\eta)^\Uu$, and hence $Q$ determines a $\Uu$-cofinal branch, as desired. Now suppose
instead that $R\sats$ ``There is not a proper class of Woodins''. Then $\delta<\eta$, where $\delta$ is the sup of Woodins in $R$.
We have $\bar{Q}\ins M^\Tt_\eta$, but since $\eta$ is not Woodin in $R$, it follows that $\OR^{\bar{Q}}<\min(\lh(E^\Tt_\eta),\lh(E^\Uu_\eta))$. But then $\bar{Q}\pins R=M(\Uu)$ and $\bar{Q}\pins M^\Uu_\eta$, so $\bar{Q}$ determines $[0,\eta)^\Uu$, which again suffices.
\end{proof}

\section{Independent families}\label{sec:indep_families}

\begin{dfn}\label{dfn:8.1}
 For $\kappa$ an uncountable cardinal, an \emph{independent family at $\kappa$}
 is a set $\mathscr{F}\sub\pow(\kappa)$
 such that for all finite sets $\mathscr{A},\mathscr{B}\sub\mathscr{F}$ with $\mathscr{A}\cap\mathscr{B}=\emptyset$,
 letting $\mathscr{B}'=\{\kappa\cut B\bigm|B\in\mathscr{B}\}$, we have
 \begin{equation*}\label{eqn:label_for_eqn_33} \card\Big(\big(\bigcap\mathscr{A}\big)\cap\big(\bigcap\mathscr{B}'\big)\Big)=\kappa.\end{equation*}

 A \emph{maximal independent family}
 is an independent family
 $\mathscr{F}$ for which there is no independent family $\mathscr{F}'$ with $\mathscr{F}\psub\mathscr{F}'$.
\end{dfn}

There is a result for maximal independent families analogous to
that for mad families given by Theorem~\ref{tm:regular_kappa_Pi_1_kappa_mad_families}:

\begin{tm}\label{tm:regular_kappa_Pi_1_kappa_mi_families}
Let $M$ be a $(0,\om_1+1)$-iterable premouse, and let $\kappa\in\OR^M$
be such that $M\sats$ ``$\kappa>\om$ is a regular cardinal and $\kappa^+$ exists''.
Let $\mathfrak{m}=M|\om_1^M$.
Then:
\begin{enumerate}
 \item \label{item:mouse_Pi_1_mi_family}
$M\sats$ ``there is a maximal independent family $\mathscr{F}\sub\pow(\kappa)$ of cardinality $\kappa^+$ which is $\Pi_1(\{\kappa,\mathfrak{m}\})$''.
\item\label{item:tame_Pi_1_mi_family} If $M$ is tame then $M\sats$ ``there is a maximal independent family $\mathscr{F}\sub\pow(\kappa)$
of cardinality $\kappa^+$ which is $\Pi_1(\{\kappa,x\})$
for some $x\in\RR$''.
\item\label{item:M_n,wlw_mi_family} If $M=M_n$ for some $n<\om$ \tu{(}see Theorem~\ref{tm:regular_kappa_Pi_1_kappa_mad_families} part~\ref{item:M_n,wlw}\tu{)}
or $M$ is the minimal proper class mouse satisfying ``there is a Woodin limit of Woodin cardinals'' then $M\sats$ ``there is a maximal independent family $\mathscr{F}\sub\pow(\kappa)$ of cardinality $\kappa^+$ which is $\Pi_1(\{\kappa\})$''.
\end{enumerate}
\end{tm}

\begin{proof}
 We use a method similar to that for mad families. The key is again to arrange that for every $A\in\mathscr{F}$, every bounded subset of $\kappa$ is directly encoded into $A$. But the methods for instantiating this are somewhat different to those we used for mad families.

 Work in $M$.
 Fix a canonical surjection $\pi:\kappa\to\pow({<\kappa})$
 with $\pi(\alpha)\sub\alpha$
 for each $\alpha<\kappa$,
 as in the proof of Theorem~\ref{tm:regular_kappa_Pi_1_kappa_mad_families}.
 We construct a sequence $\left<\mathscr{F}_\alpha\right>_{\alpha\leq\kappa^+}$ of independent families $\mathscr{F}_\alpha\sub\pow(\kappa)$, recursively in $\alpha$,
 maintaining that each $\mathscr{F}_\alpha$ is an independent family of cardinality $\kappa$, and for every $A\in\mathscr{F}_\alpha$,
 every bounded subset of $\kappa$ is directly encoded into $A$,
 and such that $\mathscr{F}_\alpha\sub\mathscr{F}_\beta$ for all $\alpha<\beta<\kappa^+$.
 At limit stages $\lambda$ we just take a union, setting $\mathscr{F}_\lambda=\bigcup_{\alpha<\lambda}\mathscr{F}_\alpha$; note that this maintains the inductive hypotheses (except that the cardinality of $\mathscr{F}_{\kappa^+}$ will be $\kappa^+$).
 At the end, defining $\mathscr{F}=\mathscr{F}_{\kappa^+}$, this will be a maximal independent family.

 Now let us define $\mathscr{F}_0$.
 Let $\left<I_\alpha\right>_{\alpha<\kappa}$
 be the least enumeration of pairwise disjoint intervals $\sub\kappa$, such that each $I_\alpha$ has length $\alpha$. Say that $\alpha<\kappa$
  is \emph{even} if it has form $\eta+2n$
  for some ordinal $\eta$ which is either $0$ or a limit ordinal and some $n<\om$,
  and \emph{odd} otherwise. Let $\OR_{\mathrm{even}}$
  be the class of even ordinals, and $\OR_{\mathrm{odd}}$ the odd ordinals. Let $B_0=\bigcup_{\alpha\in\kappa\cap\OR_{\mathrm{even}}}I_\alpha$ and $B_1=\bigcup_{\alpha\in\kappa\cap\OR_{\mathrm{odd}}}I_\alpha$. Let $B_0'$ be the $M$-least subset of $B_0$
 such that every bounded subset of $\kappa$
 is directly encoded as a subset of $I_\alpha$
 for some even $\alpha$.
 (By GCH, there are only $\kappa$-many such bounded sets. So since $I_\alpha$ has length $\alpha$
 and the $I_\alpha$'s are pairwise disjoint,  $B'_0$ exists.)
 Let $\mathscr{G}$ be the $M$-least independent family at $\kappa$
 of cardinality $\kappa$
 with $\bigcup\mathscr{G}\sub B_1$. Now set
 \begin{equation*}\label{eqn:label_for_eqn_34}\mathscr{F}_0=\{B_0'\cup A\bigm|A\in\mathscr{G}\}.\end{equation*}
 Note that $\mathscr{F}_0$ is an independent family  of cardinality $\kappa$, and every $A\in\mathscr{F}_0$ directly encodes every bounded subset of $\kappa$.

 Now suppose we have defined $\mathscr{F}_\alpha$ with the right properties; it  remains to  construct $\mathscr{F}_{\alpha+1}$.
 Let $A\sub\kappa$ be  $M$-least
  such that
 $A\notin\mathscr{F}_\alpha$
 and $\mathscr{F}_\alpha\cup\{A\}$
 is independent, if such an $A$ exists. We will define $\mathscr{F}_{\alpha+1}$
 to be (an independent family) such that if $A\notin\mathscr{F}_{\alpha+1}$ then $\mathscr{F}_{\alpha+1}\cup\{A\}$ is not independent, maintaining the inductive hypotheses.
 We will do this by adding exactly two new sets, $B_0$ and $B_1$
 (we will set $\mathscr{F}_{\alpha+1}=\mathscr{F}_\alpha\cup\{B_0,B_1\}$),
 with $A\sub B_0\cup B_1$,
 which certainly takes care of making $\mathscr{F}_{\alpha+1}\cup\{A\}$ non-independent, if $A\notin\{B_0,B_1\}$. For the purposes of the definability of $\mathscr{F}$, we will also ensure that $A$ itself can easily be computed from each of $B_0$ and $B_1$.

 Let $\sigma:\kappa\to[\mathscr{F}_\alpha]^{<\om}\cross[\mathscr{F}_\alpha]^{<\om}$ be the $M$-least function such that for each $\beta<\kappa$,
 if $\sigma(\beta)=(X,Y)$ then $X\cap Y=\emptyset$,
 and  $\sigma^{-1}(\{(X,Y)\})$ has cardinality $\kappa$
 for each  $(X,Y)\in[\mathscr{F}_\alpha]^{<\om}\cross[\mathscr{F}_\alpha]^{<\om}$ such that $X\cap Y=\emptyset$. Let\begin{equation*}\label{eqn:label_for_eqn_35}\tau(\beta)=\Big(\bigcap_{X\in\sigma(\beta)_0}X\Big)\cap\Big(\bigcap_{X\in\sigma(\beta)_1}(\kappa\cut X)\Big)\end{equation*}
 where $(x,y)_0=x$ and $(x,y)_1=1$. So $\tau(\beta)\cap A$ and $\tau(\beta)\cut A$ are both unbounded in $\kappa$ for each $\beta$ (by choice of $A$). We construct $B_0,B_1$ through $\kappa$-many stages $\beta<\kappa$ by recursion on $\beta$,
 determining  some $\eta_\beta<\kappa$ with $\beta\leq\eta_\beta$, and $B_0\cap\eta_\beta$ and $B_1\cap\eta_\beta$, before stage $\beta$. Set $\eta_0=0$.
  Consider stage $\beta$, given $\eta_\beta$, $B_0\cap\eta_\beta$
  and $B_1\cap\eta_\beta$.
  Let us now specify $B_0\cap[\eta_\beta,\om\eta_\beta)$
  and $B_1\cap[\eta_\beta,\om\eta_\beta)$.

  We will first arrange that $\pi(\beta)$ will be directly encoded into $B_0$ and into $B_1$.
  For $n<\om$,
  we first put \begin{equation*}\label{eqn:label_for_eqn_36}[(2n+1)\eta_\beta,(2n+2)\eta_\beta)\sub B_1,\end{equation*}
  \begin{equation*}\label{eqn:label_for_eqn_37}[(2n+2)\eta_\beta,(2n+3)\eta_\beta)\sub B_0.\end{equation*}
  Note that this ensures
  that $[\eta_\beta,\om\eta_\beta)\sub B_0\cup B_1$,
  so that we certainly won't violate the requirement that $A\sub B_0\cup B_1$ within the interval $[\eta_\beta,\om\eta_\beta)$.
  Now we directly encode $\pi(\beta)$ in $B_0$ at $\eta_\beta$ (and since $\pi(\beta)\sub\beta\leq\eta_\beta$, this requires at most the interval $[\eta_\beta,2\eta_\beta)$; we don't put any further ordinals into $B_0\cap [\eta_\beta,2\eta_\beta)$ beyond those needed for this).
  And we also directly encode $\pi(\beta)$ in $B_1$ at $2\eta_\beta$ (within the interval $[2\eta_\beta,3\eta_\beta)$).
  We now want to ensure that $A\cap[\eta_\beta,\om\eta_\beta)$
  can be easily recovered from $B_i\cap[3\eta_\beta,\om\eta_\beta)$, for $i\in\{0,1\}$. We do this by sliding the sets $A_{[(n+1)\eta_\beta,(n+2)\eta_\beta)}$
  up appropriately, staggered between $B_0$ and $B_1$,
  into the remaining intervals we have available.
  For $\eta\leq\delta$,
  let $A_{[\eta,\delta)}=\{\xi\bigm|\eta+\xi\in A\cap\delta\}$.
  Then we directly encode $A_{[\eta_\beta,2\eta_\beta)}$
  in $B_0$ at $3\eta_\beta$
  (which determines $B_0\cap[3\eta_\beta,4\eta_\beta)$)
  and in $B_1$ at $4\eta_\beta$
  (determining $B_1\cap[4\eta_\beta,5\eta_\beta)$).
  Likewise, we directly encode $A_{[(n+1)\eta_\beta,(n+2)\eta_\beta)}$
  in $B_0$ at $(3+2n)\eta_\beta$
  and in $B_1$ and $(4+2n)\eta_\beta$, for each $n<\om$.
  This determines $B_0\cap[\eta_\beta,\om\eta_\beta)$ and $B_1\cap[\eta_\beta,\om\eta_\beta)$.

  We next want to take measures
  to help ensure the independence
  of $\mathscr{F}_\alpha\cup\{B_0,B_1\}$.
   We find the least  3 ordinals $\xi_0<\xi_1<\xi_2$ such that $\xi_i\geq \omega\eta_\beta$
  and  $\xi_0,\xi_1,\xi_2\in\tau(\beta)$,
  and then the least ordinal $\xi_3$ such that $\xi_2<\xi_3$ and $\xi_3\in \tau(\beta)\cut A$.
  We then put all ordinals in the interval $[\omega\eta_\beta,\xi_3]$
  into both $B_0,B_1$, excepting that we declare $\xi_1,\xi_3\notin B_0$ and $\xi_2,\xi_3\notin B_1$.
  This determines $B_0\cap(\xi_3+1)$ and $B_1\cap(\xi_3+1)$. Let $\gamma$ be such that $\om\eta_\beta+\gamma=\xi_3+1$. We then set
  \begin{equation*}\label{eqn:label_for_eqn_38} [\om\eta_\beta+(2n+1)\gamma,\om\eta_\beta+(2n+2)\gamma)\sub B_1,\end{equation*}
  \begin{equation*}\label{eqn:label_for_eqn_39} [\om\eta_\beta+(2n+2)\gamma,\om\eta_\beta+(2n+3)\gamma)\sub B_0.\end{equation*}
  We then, much like before, directly encode $A_{[\om\eta_\beta+n\gamma,\om\eta_\beta+(n+1)\gamma)}$
  in $B_0$ at $\om\eta_\beta+(2n+1)\gamma$, and in $B_1$ at $\om\eta_\beta+(2n+2)\gamma$. Defining $\eta_{\beta+1}=\om\eta_\beta+\om\gamma$,
  this determines $B_0\cap\eta_{\beta+1}$ and $B_1\cap\eta_{\beta+1}$, completing this round.
  Note that  $A\cap[\eta_\beta,\eta_{\beta+1})\sub B_0\cup B_1$, since in fact
  $\xi_3\notin A$ and
\begin{equation*}\label{eqn:label_for_eqn_40} [\eta_\beta,\eta_{\beta+1})\cap(B_0\cup B_1)=[\eta_\beta,\eta_{\beta+1})\cut\{\xi_3\}.\end{equation*}

  At limits $\lambda<\kappa$
  set $\eta_\lambda=\sup_{\alpha<\lambda}\eta_\alpha$.

  This completes the construction of $B_0,B_1$.
  Clearly each bounded subset of $\kappa$ is encoded directly into both $B_0,B_1$, and $A\sub B_0\cup B_1$. And note that every finite Boolean combination of elements $\mathscr{F}_\alpha\cup\{B_0,B_1\}$ is unbounded in $\kappa$ (recall
  that $\sigma^{-1}(X,Y)$ is unbounded in $\kappa$ for each $(X,Y)\in[\mathscr{F}_\alpha]^{<\om}\cross[\mathscr{F}_\alpha]^{<\om}$ with $X\cap Y=\emptyset$). So we have maintained the inductive requirements. Note also that $A$ is easily recoverable from $B_0$. To do this, recover $\eta_\beta$, $\eta_{\beta+1}$ and $A\cap[\eta_\beta,\eta_{\beta+1})$,
  recursively in $\beta$.
  Suppose we have $\eta_\beta$, and adopt notation as above. Then $\xi_3$ is just the second ordinal $\xi>\om\eta_\beta$
  such that $\xi\notin B_0$. From this, we can recover $\gamma$, and $\eta_{\beta+1}$. From these ordinals,
  we can recover $A\cap[\eta_\beta,\om\eta_\beta)$ and $A\cap[\om\eta_\beta,\eta_{\beta+1})$,
  considering how it was encoded above.
  Likewise, $A$ is also easily recoverable from $B_1$.

  This completes the construction of $\mathscr{F}$. Clearly it is an independent family. And its maximality is established just like  in the proof of Theorem~\ref{tm:regular_kappa_Pi_1_kappa_mad_families}.

  Now we want to see that $\mathscr{F}$ is appropriately definable. This is much like in the proof of Theorem~\ref{tm:regular_kappa_Pi_1_kappa_mad_families}, using the following observation. Given  $B\sub\kappa$, we have $B\in\mathscr{F}$ iff
   for all $Q\pins M$ with $\rho_\om^Q=\kappa$
   and $Q\sats$ ``$\kappa^+$ exists'' and $B\in Q$,
   we have $B\in\mathscr{F}^Q$,
   where $\mathscr{F}^Q$ is defined over $Q$ just as $\mathscr{F}$ was over $M$.
   For it's easy to see that $\mathscr{F}^Q\sub \mathscr{F}$
   (in fact $\mathscr{F}^Q_\alpha=\mathscr{F}_\alpha$
   for all $\alpha<\kappa^{+Q}$).
   And conversely, if $B\in\mathscr{F}$ then fix $\alpha<\kappa^{+}$ such that $B\in\{B_0,B_1\}$ where these are the sets produced at stage $\alpha$, and let $A$ be the corresponding set at that stage. Then because $A$ is easily recovered from $B$,
   we have $A\in Q|\kappa^{+Q}$,
   and it follows that $\alpha<\kappa^{+Q}$,
   and so $B\in\mathscr{F}_{\alpha+1}=\mathscr{F}_{\alpha+1}^Q$.
   Using this observation,
   the methods used in the proof of Theorem~\ref{tm:regular_kappa_Pi_1_kappa_mad_families} allow us to produce an appropriate definition of $\mathscr{F}$.
\end{proof}

\section{Filters above  Measurables}\label{sec:ultrafilters}

We make a couple of simple observations on definability of filters.

\begin{tm}\label{tm:9.1}Assume ZFC, $\kappa$ is a limit of measurable cardinals
and $\cof(\kappa)=\om$.
Then there is no $\Sigma_1(V_\kappa\cup\OR)$ ultrafilter over $\kappa$
which contains no bounded subsets of $\kappa$.
\end{tm}
\begin{proof}
 After arranging for fixing the defining elements, just consider a sequence of measurables $\kappa_0<\kappa_1<\ldots$ cofinal in $\kappa$, and the sets $X=\cup_{n<\om}[\kappa_{2n},\kappa_{2n+1})$ and $Y=\bigcup_{n<\om}[\kappa_{2n+1},\kappa_{2n+2})$.
 Either $X\in U$ or $Y\in U$.
 Say $X\in U$ (otherwise it is similar). Then we can easily iterate
 at the $\kappa_n$'s, sending $X$ to some $X'$ with $X\cap X'=\emptyset$ (and we have chosen the $\kappa_n$'s so that the relevant iteration will fix the defining parameters).
 This is a contradiction.
\end{proof}

Let $\mathrm{Club}_\kappa$ denote the club filter at $\kappa$.
It was shown by Friedman
and  Wu in \cite[Proposition 2.1]{lcdnsi}
that if $\kappa$ is weakly compact then $\mathrm{Club}_\kappa$
is not $\Pi_1(\her_{\kappa^+})$.
And it was shown by  L\"ucke,
 Schindler and  Schlicht
in \cite[Theorem 1.9]{luecke_schindler_schlicht}
that if $\kappa$ is a regular cardinal which is a stationary
limit of $\om_1$-iterable cardinals
then $\mathrm{Club}_\kappa$
is not $\Pi_1(\{\kappa\})$. We prove a variant of these results here.
Neither part of Theorem~\ref{tm:club_filter} is a direct consequence of the results just mentioned; in particular in  part~\ref{item:club_filter_at_limit_of_measurables}, we do not demand that $\kappa$ be a \emph{stationary} limit of measurables (and we also allow arbitrary parameters in $V_\kappa\cup\OR$).
If $\kappa$ is the least inaccessible limit of measurables,
then the set $C$ of all  limits of measurables $<\kappa$
is club in $\kappa$
and consists of singular cardinals. But every $\om_1$-iterable cardinal is inaccessible
(see \cite[Propositions 2.8, 3.1]{gitman_ramsey-like_cardinals}),
so \cite[Theorem 1.9]{luecke_schindler_schlicht}
requires that $\kappa$
be a stationary limit of inaccessibles (and more).
\begin{tm}\label{tm:club_filter}
Let $\kappa$ be an  uncountable  cardinal.  Then:
\begin{enumerate}\item\label{item:club_filter_above_one_measurable}
 Suppose $\mu<\kappa$ and $\mu$ is measurable, $\kappa$ is $\mu$-steady and $\cof(\kappa)>\mu$.
 Then $\mathrm{Club}_\kappa$ is not $\Pi_1(V_\mu\cup\{\kappa,\kappa^+\})$. In fact, let $U$ be any normal measure on $\mu$
 and let $S_j=\{x\bigm|j(x)=x\}$
 where $j=i^V_U:V\to M=\Ult(V,U)$ is the ultrapower map \tu{(}so $V_\mu\cup\{\kappa,\kappa^+\}\sub S_j$
 \tu{)}.
 Then $\mathrm{Club}_\kappa$
 is not $\Pi_1(S_j)$.
 \item\label{item:club_filter_at_limit_of_measurables} Suppose $\kappa$ is an inaccessible limit of measurables.
 Then $\mathrm{Club}_\kappa$ is not $\Pi_1(V_\kappa\cup\OR)$.
 \end{enumerate}
\end{tm}
\begin{proof}Part~\ref{item:club_filter_above_one_measurable}:
 Suppose otherwise. Then the set of stationary subsets of $\kappa$ is $\Sigma_1(S_j)$,
 since given $S\sub\kappa$,
 $S$ is stationary iff $\kappa\cut S\notin\Club_\kappa$. Since this equivalence is preserved by $j$ and $j(\kappa)=\kappa$,
 we have that for each $S\in\pow(\kappa)\cap M$, if $M\sats$ ``$S$ is stationary''
 then $S$ is stationary.
 Let $S=\{\alpha<\kappa\bigm|\cof^{M}(\alpha)=\mu^+\}$.
 Then $M\sats$ ``$S$ is stationary'' (since $\cof(\kappa)>\mu$, we have $\cof^M(\kappa)>j(\mu)$).
 But $S$ is not stationary, a contradiction.
 For let $C^-=j``\kappa$ and let $C$  be the closure of $C^-$. So $C$ is club in $\kappa$, so it suffices to see that $C\cap S=\emptyset$.
 But $S\cap\rg(j)=\emptyset$,
 and every $\alpha\in C\cut C^-$
 is such that $\cof(\alpha)=\mu$,
 and since $M$ is closed under $\mu$-sequences,
 therefore $\cof^{M}(\alpha)=\mu$ also, so $\alpha\notin S$. This completes the proof.

 Part~\ref{item:club_filter_at_limit_of_measurables}: This is an immediate corollary of part~\ref{item:club_filter_above_one_measurable} and Lemma~\ref{lem:ordinals_ev_stable_meas_limit}.
\end{proof}

\section{Global Regularity Properties}\label{sec:global}

In this section we describe a variant of Schlicht \cite[Theorem 2.19]{schlicht_perfect_set_property}.
The adaptation involves standard forcing techniques; see for example \cite{cummings_handbook}.
\begin{tm}\label{tm:supercompact_forcing}
 Assume ZFC + $\kappa$ is $\lambda$-supercompact, where $\lambda>\kappa$ is inaccessible
 and $2^\lambda=\lambda^+$. Then in a forcing extension $V[G]$, $\kappa$ is $\lambda$-supercompact, $(2^\kappa)^{V[G]}=\kappa^{+V[G]}=\lambda$ and
 for every $X\sub\pow(\kappa)^{V[G]}$
 such that $X$ is definable over $V[G]$
 from elements of $V\cup \pow(\kappa)^{V[G]}$ and $X$ has cardinality $>\kappa$, we have that
  $X$ has a perfect subset,
  and  $X$ is not a wellorder.
\end{tm}
\begin{proof}
 Let $\mathbb{P}$ be length $\kappa$ Easton support iteration $\big(\left<\PP_\alpha\right>_{\alpha\leq\kappa},\left<\dot{\QQ}_\alpha\right>_{\alpha<\kappa}\big)$,
 where at inaccessible stages $\alpha$,
 if there is an inaccessible $\lambda'>\alpha$ such that $\alpha$ is ${<\lambda'}$-supercompact but not $\lambda'$-supercompact, then letting $\lambda_\alpha$ be the least such $\lambda'$,
 we force with the name $\dot{\QQ}_\alpha$ for the Levy collapse $\mathrm{Coll}(\alpha,{<\lambda_\alpha})$
 of $V[\dot{G}\rest\alpha]$ (where $\dot{G}\rest\alpha$ is the canonical name for the generic for $\PP_\alpha$),
and at other stages $\alpha$,
$\dot{\QQ}_\alpha$ is trivial.
 Let $\PP^+=\PP *\dot{\QQ}$
 where $\dot{\QQ}$ is the $\PP$-name for $\mathrm{Coll}(\kappa,{<\lambda})^{V[\dot{G}]}$.
 Let $G^+$ be $(V,\PP^+)$-generic
 and $G=G^+\rest\kappa$ (so $G$ is $(V,\PP)$-generic).

 Note that there is a $\lambda$-supercompactness embedding $j:V\to M$ such that $M\sats$ ``$\kappa$ is not $\lambda$-supercompact''. (Otherwise we can produce a linear iteration of $V$, of length $\om$, via $\lambda$-supercompactness measures with critical point $\kappa$, whose direct limit is illfounded.)
So fix such a $j$. Clearly $M\sats$ ``$\kappa$ is ${<\lambda}$-supercompact''.
 It follows that $\PP^+=j(\PP)\rest(\kappa+1)$.
And $G^+$ is $(M,\PP^+)$-generic.
 We have $\lambda=\kappa^{+V[G^+]}=\kappa^{+M[G^+]}$.
 Now $V\sats$ ``${^\lambda}M\sub M$'',
 and considering $\PP^+$-names in $V$
 for $\lambda$-sequences in $V[G^+]$,
 a straightforward calculation gives that $V[G^+]\sats$ ``${^\lambda}(M[G^+])\sub M[G^+]$''.

 Let $H^+=j``G^+$. So $H^+\in M[G^+]$.
 Working in $M[G^+]$,
 let $\widetilde{q}$ be the canonical $j(\PP)\rest(\kappa+1,j(\kappa))$-name for the union
 of those conditions of form $(j(p)_{j(\kappa)})_{G^+*\dot{I}}$
 for $p\in G^+$,
 where $\dot{I}$ is the canonical
 name for the generic filter for the tail forcing $j(\PP)\rest(\kappa+1,j(\kappa))$,
 and where $j(p)_\alpha$
 is the $\alpha$th component of $j(p)$.
 Note that in $M[G^+]$,
 $\widetilde{q}$ is forced (by $j(\PP)\rest(\kappa+1,j(\kappa))$)
 to be a $\mathrm{Coll}(j(\kappa),{<j(\lambda)})$-name, because the collection of conditions is forced to be pairwise compatible, even closed under witnesses for such compatibility,
 and because the collection has size $\lambda$, which is smaller than the closure of the forcing
 (which is $j(\kappa)$). So in $M[G^+]$, $\widetilde{q}$ yields a condition
 in $j(\PP)\rest(\kappa+1,j(\kappa))*\dot{\mathrm{Coll}}(j(\kappa),{<j(\lambda)})$.

 Now working in $V[G^+]$,
 we construct an $M[G^+]$-generic
 filter $J$ for this tail forcing $\QQ$ with $\widetilde{q}\in J$. This is done by enumerating the relevant maximal antichains in ordertype $\lambda^+$ and then recursively building a descending sequence of conditions meeting them, below $\widetilde{q}$. Let us observe why this is possible. Note that $M[G^+]\sats$ ``$j(\lambda)$ is inaccessible, $\QQ$
 has cardinality $j(\lambda)$ and has the $j(\lambda)$-cc''.
 So $M[G^+]$ has an enumeration of its maximal antichains of $\QQ$ in ordertype $j(\lambda)$.
 But $\mathrm{card}^{V[G^+]}(j(\lambda))=\card^V(j(\lambda))=2^\lambda=\lambda^+$, so $V[G^+]$ has an enumeration in ordertype $\lambda^+$. And $M[G^+]\sats$ ``$\QQ$ is $(\lambda+1)$-closed''. But $M[G^+]$ is closed under $\lambda$-sequences in $V[G^+]$, so $V[G^+]\sats$ ``$\QQ$ is $(\lambda+1)$-closed''. Therefore working in $V[G^+]$, an $M[G^+]$-generic filter can indeed be constructed, as desired.

 So $M[G^+,J]\sub V[G^+]$.
 But now $j:V\to M$
 extends to $j^+:V[G^+]\to M[G^+,J]$
 with $j^+(G^+)=G^+*J$,
 and  $j^+$ witnesses that $\kappa$ is $\kappa^{+V[G^+]}$-supercompact in $V[G^+]$.

 It remains to verify the desired regularity properties of sets $X$ in $V[G^+]$. So suppose that $X\sub\pow(\kappa)$ is definable over $V[G^+]$ from some $A\sub\kappa$
 and some $a\in V$, and $X$ has cardinality $\lambda=\kappa^{+V[G^+]}$
 in $V[G^+]$.

 The fact that $X$ has a perfect subset follows from  \cite[Theorem 2.19]{schlicht_perfect_set_property},
 since $\kappa$ is regular in $V[G]$, $\lambda$ inaccessible in $V[G]$, and the last factor of $\PP^+$ is just $\mathrm{Coll}(\kappa,{<\lambda})^{V[G]}$.

 Let us see that $X$ is not a wellorder.
 Suppose otherwise.
 There is some $\beta<\lambda$
 such that $A\in V[G,(G^+)_\kappa\rest\beta]$,
 where $(G^+)_\kappa\rest\beta$
 is just the restriction of $(G^+)_\kappa$
 to the conditions with support $\sub\beta$.
 Since $X$ has cardinality $>\kappa$ in $V[G^+]$,
 there must be some $B$ in the field of $X$ which is not in $V[G,(G^+)_\kappa\rest\beta]$.
 Let $\xi\in\OR$ be the rank of $B$ in the wellorder.
 Then by the homogeneity of the tail forcing adding $(G^+)_\kappa\rest[\beta,\lambda)$,
 one can define $B$ from the parameters $A,a,\xi$ over $V[G,(G^+)_\kappa\rest\beta]$, by consulting the forcing relation. So $B\in V[G,(G^+)_\kappa\rest\beta]$, a contradiction.
\end{proof}

\section{Questions}\label{sec:questions}
I now list some questions to which I would like to know the answers. (ZFC is still the background.)

\begin{enumerate}
 \item Let $\mu$ is measurable,
 and $U$ a $\mu$-complete ultrafilter over $\mu$. Suppose
 $\mu<\kappa$ and $\kappa$ is a $\mu$-steady cardinal,
 and $\kappa$ is not a limit of measurables.
 Can there be an ultrafilter $\mathscr{F}$ over $\kappa$ which
 contains all unbounded subsets of $\kappa$, and which is $\Sigma_1(V_\mu\cup S_\kappa)$-definable,
 where $S_\kappa=\{x\bigm|i^{\Tt_U}_{0\alpha}(x)=x\text{ for all }\alpha<\kappa\}$?
 What if $\mathrm{Club}_\kappa\sub\mathscr{F}$?
 What if $V=L[U]$?
 What about ``$\Pi_1$'' replacing ``$\Sigma_1$''?
 \item Assume $\kappa$ is a limit of measurables.
 Can there be a $\Sigma_1(\her_\kappa\cup\OR)$ maximal independent family $\sub\pow(\kappa)$?

 \item Does Theorem~\ref{tm:regular_kappa_Pi_1_kappa_mad_families}
 generalize to the case that $\kappa$ is singular? What about Theorem~\ref{tm:regular_kappa_Pi_1_kappa_mi_families}?

 \item Assume $V=L[U]$ where $U$ is a normal measure on $\mu$.
 Let $\kappa>\mu^+$ be $\mu$-steady.
 Let $X$ be as in the proof of Theorem~\ref{tm:L[U]_non-basis}
 part~\ref{item:Pi_1_singletons_non-basis}. Is $X$ $\Pi_1(\{\kappa\})$-definable? It does not seem clear to the author that $A\in X$
 iff for every $\kappa$-good filter $D$ and every ordinal $\alpha>\kappa$, if $A\in L_\alpha[D]$ then there is a limit ordinal $\eta<\kappa$
 such that $L_\alpha[D]\sats$ ``$A\in\bigcap_{\beta<\eta}M^{\Tt_D}_\beta$ but $A\notin M^{\Tt_D}_\eta$''.
 For it seems that we might have $A\in X$ and some $\kappa$-good filter $D$
 and some $\alpha<\kappa^+$
 such that $A\in L_\alpha[D]$
 but $A\notin\Ult(L_\alpha[D],D)$; in this case, since $A\in X$,
 we must actually have $A\in\Ult(L[D],D)$;
 it is just that $\alpha$ is not large enough to see where $A$ is constructed in $\Ult(L[D],D)$.
 But the author doesn't know whether this actually occurs,
 and whether one could compute any useful bound on where in $\Ult(L[D],D)$ a set $A\sub\kappa$
 can be constructed,
 given where it is constructed in $L[D]$.
 \item Under the hypotheses of Theorem~\ref{tm:3.11} part~\ref{item:I_2_PSP},
 can one conclude that there is a perfect embedding $\pi:{^{\om}}\lambda\to\pow(\lambda)$ with $\rg(\pi)\sub D$? (So here the corresponding tree should be $\lambda$-splitting, instead of just binary splitting.)
\end{enumerate}

%% file: low_level_definability_author_accepted.bbl
\begin{thebibliography}{10}

\bibitem{schindler_star}
Benjamin Claverie and Ralf Schindler.
\newblock Woodin's axiom {{\((*)\)}}, bounded forcing axioms, and precipitous
  ideals on {{\({{\omega}} _{1}\)}}.
\newblock {\em J. Symb. Log.}, 77(2):475--498, 2012.

\bibitem{cummings_handbook}
James Cummings.
\newblock Iterated forcing and elementary embeddings.
\newblock In {\em Handbook of set theory, Volume 2}, pages 775--883. Dordrecht:
  Springer, 2010.

\bibitem{descriptive_properties_I2-embeddings}
Vincenzo Dimonte, Martina Iannella, and Philipp L\"ucke.
\newblock Descriptive properties of {I2}-embeddings.
\newblock {\em J. Symb. Log.}, 2025.
\newblock Published online, DOI:10.1017/jsl.2024.75.

\bibitem{lcdnsi}
Sy-David Friedman and Liuzhen Wu.
\newblock Large cardinals and {$\Delta_1$} definability of the nonstationary
  ideal.
\newblock Available at
  \url{http://www.logic.univie.ac.at/~sdf/papers/joint.liuzhen.delta-1.ns.pdf}.

\bibitem{PCF_and_Woodins}
Moti Gitik, Ralf Schindler, and Saharon Shelah.
\newblock {PCF} theory and {Woodin} cardinals.
\newblock In {\em Logic colloquium '02}, pages 172--205. Association for
  Symbolic Logic, 2006.

\bibitem{gitman_ramsey-like_cardinals}
Victoria Gitman.
\newblock Ramsey-like cardinals.
\newblock {\em J. Symb. Log.}, 76(2):519--540, 2011.

\bibitem{jech}
Thomas Jech.
\newblock {\em Set theory: The third millennium edition, revised and expanded}.
\newblock Springer Monographs in Mathematics. Springer, 2003.

\bibitem{Kwithoutmeas}
Ronald Jensen and John Steel.
\newblock {{\(K\)}} without the measurable.
\newblock {\em J. Symb. Log.}, 78(3):708--734, 2013.

\bibitem{kunen}
Kenneth Kunen.
\newblock Some applications of iterated ultrapowers in set theory.
\newblock {\em Annals of Mathematical Logic}, 1(2):179--227, 1970.

\bibitem{Kunen_model_negation_AC}
Kenneth Kunen.
\newblock A model for the negation of the axiom of choice.
\newblock In A.~D.~R. Matthias and H.~Rogers, editors, {\em Cambridge Summer
  School in Mathematical Logic}, pages 489--494. Springer Verlag, 1973.

\bibitem{luecke_schindler_schlicht}
Philipp L\"{u}cke, Ralf Schindler, and Philipp Schlicht.
\newblock {$\Sigma_1(\kappa)$}-definable subsets of {$H(\kappa^+)$}.
\newblock {\em J. Symb. Log.}, 82(3):1106--1131, 2017.

\bibitem{meas_cards_good_Sigma_1_wo}
Philipp L\"ucke and Philipp Schlicht.
\newblock Measurable cardinals and good ${\Sigma}_1(\kappa)$ wellorderings.
\newblock {\em Math. Log. Q.}, 64(3):207--217, 2018.

\bibitem{Sigma_1_def_at_higher_cardinals}
Philipp Lücke and Sandra Müller.
\newblock {$\Sigma_1$}-definability at higher cardinals: Thin sets, almost
  disjoint families and long well-orders.
\newblock {\em Forum Math. Sigma}, 11:e103:1--36, 2023.
\newblock DOI:10.1017/fms.2023.102.

\bibitem{happy_families}
A.~R.~D. Mathias.
\newblock Happy families.
\newblock {\em Ann. Math. Logic}, 12(1):59--111, 1977.

\bibitem{miller}
Arnold~W. Miller.
\newblock Infinite combinatorics and definability.
\newblock {\em Ann. Pure Appl. Logic}, 41(2):179--203, 1989.

\bibitem{maxcore}
E.~Schimmerling and J.~R. Steel.
\newblock The maximality of the core model.
\newblock {\em Trans. Amer. Math. Soc.}, 351(8):3119--3141, 1999.

\bibitem{schlicht_perfect_set_property}
Philipp Schlicht.
\newblock Perfect subsets of generalized {B}aire spaces and long games.
\newblock {\em J. Symb. Log.}, 82(4):1317--1355, 2017.

\bibitem{vmom_v2}
Farmer Schlutzenberg.
\newblock Varsovian models {$\om$}.
\newblock ar{X}iv:2212.14878v2.

\bibitem{mim}
Farmer Schlutzenberg.
\newblock {\em Measures in mice}.
\newblock PhD thesis, University of California, Berkeley, 2007.
\newblock arXiv:1301.4702v1.

\bibitem{iter_for_stacks}
Farmer Schlutzenberg.
\newblock Iterability for (transfinite) stacks.
\newblock {\em J. Math. Log.}, 21(2), 2021.
\newblock 2150008.

\bibitem{reinhardt_iterates}
Farmer Schlutzenberg.
\newblock Reinhardt cardinals and iterates of {$V$}.
\newblock {\em Ann. Pure Appl. Logic}, 173, 2022.
\newblock 103056.

\bibitem{extmax}
Farmer Schlutzenberg.
\newblock The definability of $\es$ in self-iterable mice.
\newblock {\em Ann. Pure Appl. Logic}, 174(2), 2023.
\newblock 103208.

\bibitem{V=HODX_pub}
Farmer Schlutzenberg.
\newblock The definability of the extender sequence {$\mathbb{E}$} from
  {$\mathbb{E}\upharpoonright\aleph_1$} in {$L[\mathbb{E}]$}.
\newblock {\em J. Symb. Log.}, 89(2):427--459, 2024.

\bibitem{odle_v2}
Farmer Schlutzenberg.
\newblock Ordinal definability in {$L[\es]$}.
\newblock {\em J. Symb. Log.}, 2026.
\newblock Published online, DOI:10.1017/jsl.2025.10099.

\bibitem{steel_dmt}
J.~R. Steel.
\newblock The derived model theorem.
\newblock In S.~Barry Cooper, Herman Geuvers, Anand Pillay, and Jouko
  V\"a\"an\"anen, editors, {\em Logic colloquium 2006}, pages 280--327.
  Cambridge University Press, 2009.

\bibitem{cmip}
John~R. Steel.
\newblock {\em The core model iterability problem}, volume~8 of {\em Lecture
  Notes in Logic}.
\newblock Springer-Verlag, 1996.

\bibitem{outline}
John~R. Steel.
\newblock An outline of inner model theory.
\newblock In {\em Handbook of set theory, Volume 3}, pages 1595--1684.
  Springer, Dordrecht, 2010.

\bibitem{ACPFMP}
John~R. Steel.
\newblock {\em A {C}omparison {P}rocess for {M}ouse {P}airs}, volume~51 of {\em
  Lecture Notes in Logic}.
\newblock Cambridge University Press, 2023.

\bibitem{power_set_Sigma_1_Card}
Jouko V\"a\"an\"anen and Philip~D. Welch.
\newblock When cardinals determine the power set: inner models and {H}\"artig
  quantifier logic.
\newblock {\em Math. Log. Q.}, 69(4):460--471, 2023.

\end{thebibliography}
